\pdfoutput=1
\documentclass[11pt]{article}

\usepackage{enumerate}
\usepackage[OT1]{fontenc}
\usepackage[usenames]{color}
\usepackage{smile}
\usepackage[colorlinks,
linkcolor=red,
anchorcolor=blue,
citecolor=blue
]{hyperref}

\usepackage{mathrsfs}
\usepackage{fullpage}
\usepackage{hyperref}
\usepackage[protrusion=true, expansion=true]{microtype}
\usepackage{float}
\usepackage{subfigure}
\usepackage{amsfonts,amsmath,amssymb,amsthm,url,xspace}
\usepackage{tikz}
\usepackage{verbatim}
\usetikzlibrary{arrows,shapes}
\usepackage{mathtools}
\usepackage{authblk}
\usepackage[bottom]{footmisc}
\usepackage{enumitem}
\usepackage{lscape}
\newcommand{\Sigmab}{\mathbf{\Sigma}}
\usepackage{arydshln,leftidx,mathtools,mathtools}
\usepackage{multirow}
\usepackage{algorithm}
\usepackage{algorithmicx,algpseudocode}
\ifx\counterwithout\undefined\usepackage{chngcntr}\fi
\counterwithout{equation}{section}
\allowdisplaybreaks[1]
\usepackage{math}

\mathtoolsset{showonlyrefs=true}

\usepackage{xargs}
\usepackage[colorinlistoftodos,prependcaption,textsize=tiny]{todonotes}
\newcommandx{\unsure}[2][1=]{\todo[linecolor=red,backgroundcolor=red!25,bordercolor=red,#1]{#2}}
\newcommandx{\change}[2][1=]{\todo[linecolor=blue,backgroundcolor=blue!25,bordercolor=blue,#1]{#2}}
\newcommandx{\info}[2][1=]{\todo[linecolor=OliveGreen,backgroundcolor=OliveGreen!25,bordercolor=OliveGreen,#1]{#2}}
\newcommandx{\improvement}[2][1=]{\todo[linecolor=Plum,backgroundcolor=Plum!25,bordercolor=Plum,#1]{#2}}

\begin{document}
	
\title{Hessian Averaging in Stochastic Newton Methods \\ Achieves Superlinear Convergence}
	
\author[1]{Sen Na}
\author[2]{Micha{\l } Derezi\'{n}ski}
\author[1]{Michael W. Mahoney}
	
\affil[1]{ICSI and Department of Statistics, University of California, Berkeley}
\affil[2]{Computer Science and Engineering, University of Michigan}

\date{}
	
\maketitle

\begin{abstract}
We consider minimizing a smooth and strongly convex objective~function using a stochastic Newton method. At each iteration, the algorithm is~given an oracle access to a stochastic estimate of the Hessian matrix. The oracle~model includes popular algorithms such as Subsampled Newton and Newton Sketch, which can efficiently construct stochastic Hessian estimates for many tasks,~e.g., training machine learning models. Despite using second-order information, these existing methods do not exhibit superlinear convergence, unless the stochastic noise is gradually reduced to zero during the iteration, which would lead to a computational blow-up in the per-iteration cost. We propose to address this limitation with Hessian averaging: instead of using the most recent Hessian estimate, our algorithm maintains an average of all the past estimates. This reduces the stochastic noise while avoiding the computational blow-up.~We show that this scheme exhibits local $Q$-superlinear convergence with a non-asymptotic~rate of $(\Upsilon\sqrt{\log (t)/t}\,)^{t}$, where $\Upsilon$ is proportional to the level of stochastic noise in the Hessian oracle. A potential drawback of this (uniform averaging) approach is that the averaged estimates contain Hessian information from the global phase of the method, i.e., before the iterates converge to a~local neighborhood. This leads to a distortion that may substantially delay the superlinear convergence until long after the local neighborhood is reached. To address this drawback, we study a number of weighted averaging schemes that assign larger weights to recent Hessians, so that the superlinear convergence arises sooner, albeit with a slightly slower rate. Remarkably, we show that~there exists a universal weighted averaging scheme that transitions to local convergence at an optimal stage, and still exhibits a superlinear convergence rate nearly (up to a logarithmic~factor) matching that of uniform Hessian averaging.
	
\end{abstract}

\section{Introduction}\label{sec:1}

We consider minimizing a smooth and strongly convex objective function:
\begin{equation}\label{pro:1}
\min_{\xb\in \mR^d}\; f(\xb),
\end{equation}
where $f:\mR^d\rightarrow \mR$ is twice continuously differentiable with $\Hb(\xb) = \nabla^2 f(\xb) \in \mR^{d\times d}$ being its Hessian matrix. We suppose the Hessian $\Hb(\xb)$ satisfies
\begin{equation*}
\lambda_{\min}\cdot \Ib \preceq \Hb(\xb) \preceq\lambda_{\max}\cdot\Ib, \quad\;  \forall \xb\in\mR^d,
\end{equation*}
for some constants $0<\lambda_{\min}\leq\lambda_{\max}$, and we let $\kappa = \lambda_{\max}/\lambda_{\min}$ denote its~condition number. We also let $\xb^\star = \arg\min_{\xb\in \mR^d} f(\xb)$ be the unique global~solution.

Problem \eqref{pro:1} is arguably the most basic optimization problem, which~nevertheless arises in many applications in machine learning and statistics \citep{ShalevShwartz2009Understanding, Bubeck2015Convex, Bottou2018Optimization, Lan2020First}. There is a plethora of methods for solving \eqref{pro:1}, and each (type of) method has its own benefits under favorable settings. This paper particularly focuses on solving \eqref{pro:1} via second-order methods, where a (approximated) Hessian matrix is involved in the scheme. Consider the classical Newton's method of the form 
\begin{equation}\label{equ:Newton}
\xb_{t+1} = \xb_t - \mu_t \Hb_t^{-1}\nabla f_t,
\end{equation}
where $\nabla f_t = \nabla f(\xb_t)$, $\Hb_t = \Hb(\xb_t)$, and $\mu_t$ is selected by passing a line search condition. Classical results indicate that Newton's method in \eqref{equ:Newton} \mbox{exploits}~a~global $Q$-linear convergence in the error of function value $f(\xb_t)-f(\tx)$; and a local $Q$-quadratic convergence in the iterate error $\|\xb_t-\tx\|$. More precisely, Newton's method has two phases, separated by a neighborhood 
\begin{equation*}
\N_\nu = \{\xb: \|\xb - \xb^\star\|\leq \nu\}, \quad \quad \text{ for some } \nu>0.
\end{equation*}
When $\xb_t\not\in \N_\nu$, \eqref{equ:Newton} is in the \textit{damped Newton phase}, where the objective $f(\xb_t)$ is decreased by at least a fixed amount in each iteration, and converges linearly. In this phase, $\|\xb_t-\tx\|$ may converge slowly (e.g., it provably converges $R$-linearly using the fact that $\|\xb_t-\tx\|^2\leq 2/\lambda_{\min}(f(\xb_t)-f(\tx))$). When $\xb_t\in \N_\nu$, \eqref{equ:Newton} transits to the \textit{quadratically convergent phase}, where the unit stepsize is accepted and $\|\xb_t-\xb^\star\|$ converges quadratically. See \cite[Section 9.5]{Boyd2004Convex} for the analysis. Compared to first-order methods, although second-order methods often exploit a faster convergence rate and behave more robustly to tuning parameters, they hinge on a high computational cost of forming the exact Hessian matrix $\Hb_t$ at each step. To resolve such a computational bottleneck, which is particularly pressing in large-scale data applications,~a variety of~deterministic and stochastic methods have been proposed for constructing different alternatives of the Hessian matrix.

When a deterministic alternative of $\Hb_t$, say $\cHb_t$, is employed in \eqref{equ:Newton}, which~may come from a finite difference approximation of the second-order derivative, or from a quasi-Newton update such as BFGS or DFP, the convergence behavior is well understood. Specifically, if $\{\cHb_t\}_t$ are positive definite with uniform upper and lower bounds, the damped phase is preserved by the same analysis as Newton's method. Furthermore, the quadratically convergent phase is weakened to a superlinearly convergent phase, and the superlinear convergence occurs if and only if  the celebrated Dennis-Mor\'e condition
\citep{Dennis1974characterization} holds, i.e.,
\begin{equation}\label{equ:DM:cond}
\lim\limits_{t\rightarrow \infty} \frac{\|( \cHb_t - \Hb_t) \cHb_t^{-1}\nabla f_t\|}{\| \cHb_t^{-1}\nabla f_t\|} = 0.
\end{equation}
See \cite[Theorem 3.7]{Nocedal2006Numerical} for the analysis. Recently, a deeper understanding of quasi-Newton methods for minimizing smooth and strongly convex objectives has been reported in \cite{Jin2020Non, Rodomanov2021Greedy, Rodomanov2021New, Rodomanov2021Rates}. These works  enhanced the analyses based on Dennis-Mor\'e condition in \eqref{equ:DM:cond} by performing a \textit{local, non-asymptotic} convergence analysis, and provided explicit superlinear rates for different potential functions of quasi-Newton methods. The non-asymptotic superlinear results are more informative than asymptotic superlinear results established via \eqref{equ:DM:cond}, i.e., $\|\xb_{t+1} - \ttxb\|/\|\xb_t - \ttxb\|\rightarrow 0$ as $t\rightarrow\infty$. However, the analyses in \cite{Jin2020Non, Rodomanov2021Greedy, Rodomanov2021New, Rodomanov2021Rates}  highly rely on specific properties of quasi-Newton updates in the Broyden class, and do not apply to general Hessian approximations (e.g., finite difference~approximation).

\subsection{Stochastic Newton methods}
\label{subsec:stochastic-newton}

A parallel line of research explores stochastic Newton methods, where a stochastic approximation $\hHb_t$ is used in place of $\Hb_t$ in \eqref{equ:Newton}. The stochasticity of $\hHb_t$ may come from evaluating the Hessian on a random subset of data points~(i.e., subsampling), or from projecting the Hessian onto a random subspace to~achieve~the dimension reduction (i.e., sketching). To unify different approaches, we consider in this paper a general Hessian approximation given by a \textit{stochastic oracle}. In particular, we express the approximation $\hHb(\xb)$ at any point $\xb$ by
\begin{equation}\label{equ:error}
\hHb(\xb) = \Hb(\xb) + \Eb(\xb),
\end{equation}
where $\Eb(\xb)\in \mS^{d\times d}$ is a symmetric random noise matrix following a certain distribution (conditional on $\xb$) with mean zero. At iterate $\xb_t$, we query the oracle to obtain an approximation $\hHb_t = \hHb(\xb_t)$, which (implicitly) comes from generating a realization of the random matrix $\Eb_t = \Eb(\xb_t)$, and then adding $\Eb_t$ to the true Hessian $\Hb_t$. We do not assume $\Eb_t$ and $\Hb_t$ are accessible to us.

As mentioned, the popular specializations of stochastic oracle include subsampling and sketching. For Hessian subsampling, a finite-sum objective is considered
\begin{equation}\label{equ:SSN:obj}
f(\xb) = \frac{1}{n}\sum_{i=1}^{n}f_i(\xb).
\end{equation}
In the $t$-th iteration, a subset $\xi_t \subseteq\{1,2,\ldots, n\}$ is sampled uniformly at random, and the subsampled Hessian is defined as
\begin{equation}\label{equ:SSN}
\hHb_t = \frac{1}{|\xi_t|}\sum_{i\in \xi_t}\nabla^2f_i(\xb_t).
\end{equation}
We note that the components $f_i$ in \eqref{equ:SSN:obj} may not be convex even if the function~$f$ is strongly convex.\footnote{A concrete example is finding the leading eigenvector of a covariance matrix $\Ab = \frac{1}{n} \sum_{i=1}^{n}\ab_i\ab_i^\top$. Here, the objective can be structured as $f_{\bb, \nu}(\xb) = \frac{1}{n}\sum_{i=1}^{n}(\xb^\top(\nu\Ib - \ab_i\ab_i^\top)\xb - \bb^\top\xb)$, where $\nu>\|\Ab\|$ and $\bb$ are given. See \cite{Garber2015Fast} for details.} This is the so-called \emph{sum-of-non-convex} setting \cite[e.g., see][]{Garber2015Fast, Garber2016Faster, ShalevShwartz2016SDCA, AllenZhu2016Improved}. Our oracle model \eqref{equ:error} allows for this, since $\hHb_t$ is not required to be positive semidefinite, while the existing Subsampled Newton methods generally do not allow it. See Section \ref{subsec:literature} and Example \ref{exp:1} for further discussion.

For Hessian sketching, one first forms the square-root Hessian matrix~$\Mb_t\in \mR^{n\times d}$ satisfying~$\Hb_t = \Mb_t^\top\Mb_t$, where $n$ is the number of data points. Then, one generates a random sketch matrix $\Sb_t\in \mR^{s\times n}$ with the sketch size $s$ and~the property $\mE[\Sb_t^\top\Sb_t] = \Ib$, and the sketched Hessian is defined as
\begin{equation}\label{equ:SKN}
\hHb_t = \Mb_t^\top\Sb_t^\top\Sb_t\Mb_t.
\end{equation}
In some cases, $\Mb_t$ can be easily obtained. One particular example is a generalized linear model, where the objective has the form 
\begin{equation}\label{equ:SKN:obj}
f(\xb) = \frac{1}{n}\sum_{i=1}^nf_i(\ab_i^\top\xb)
\end{equation}
with $\{\ab_i\}_{i=1}^n\in \mR^d$ being $n$ data points. In this case, $\Mb_t = \frac{1}{\sqrt{n}}\cdot\diag(f_1''(\ab_1^\top\xb_t)^{1/2}, \ldots, f_n''(\ab_n^\top\xb_t)^{1/2})\Ab$ where $\Ab = (\ab_1, \ldots, \ab_n)^\top\in \mR^{n\times d}$ is the data matrix.

Another family of stochastic Newton methods is based on the so-called~Sketch-and-Project framework \citep[e.g.,][]{Gower2015Randomized}, where a low-rank Hessian estimate is used to construct a Newton-like step (with the Moore-Penrose pseudoinverse instead of the matrix inverse). For example, in one approach \citep{Gower2019RSN}, the Newton-like step in the $t$-th iteration is obtained by generating a sketching matrix $\Sb_t\in\mR^{s\times d}$ to construct a rank-$s$ approximate of the inverse Hessian, resulting in the update:
\vskip-6pt
\begin{equation*}
\xb_{t+1} = \xb_t - \mu_t\Sb_t^\top(\Sb_t\Hb_t\Sb_t^\top)^\dagger\Sb_t\nabla f(\xb_t).
\end{equation*}
While linear convergence rates have been derived for these methods \citep[e.g.,][]{Derezinski2022Sharp}, they do not fit into our stochastic oracle framework due to an intrinsic bias arising in the Hessian estimates (see Section \ref{subsec:literature} for further discussion).

Numerous stochastic Newton methods using \eqref{equ:SSN} or \eqref{equ:SKN} with various types of convergence~guarantees have been proposed \citep{Byrd2011Use, Byrd2012Sample, Friedlander2012Hybrid, Erdogdu2015Convergence, RoostaKhorasani2018Sub, Bollapragada2018Exact, Pilanci2017Newton, Agarwal2017Second, Kovalev2019Stochastic, Derezinski2019Distributed, Derezinski2020Debiasing, Derezinski2020Precise, Lacotte2021Adaptive}. We review these related references later and point to \cite{Berahas2020investigation} for a recent brief survey. However, the existing approaches to stochastic Newton~methods share common limitations, which we now discuss.

The majority of literature studies the convergence of stochastic Newton by establishing a \textit{high probability error recursion}. For example, \cite{Erdogdu2015Convergence, RoostaKhorasani2018Sub, Pilanci2017Newton, Agarwal2017Second} all showed, roughly speaking, that stochastic Newton methods exploit a local~linear-quadratic recursion: 
\begin{equation}\label{equ:LQ:error}
\|\xb_{t+1}-\xb^\star\| \leq c_1\|\xb_t - \xb^\star\| + c_2\|\xb_t-\xb^\star\|^2 \quad\quad  \text{with probability } 1 - \delta,
\end{equation}
for some constants $c_1, c_2>0$ and $\delta\in(0, 1)$. These constants depend on~the~sample/sketch size at each step. Based on \eqref{equ:LQ:error}, one often applies the result for~$t = 0,1,\ldots,T$ recursively, uses the union bound, and establishes local convergence with probability $1 - T\delta$. This approach has the following key~drawbacks.
\begin{enumerate}[label=(\alph*),topsep=1pt]
\item The presence of the linear term with coefficient $c_1>0$ in the recursion~means that, once $\xb_t$ is sufficiently close to $\tx$, the algorithm can only achieve the linear convergence, as opposed to the quadratic or superlinear convergence achieved by deterministic methods. Prior works \citep[e.g.,][]{RoostaKhorasani2018Sub, Bollapragada2018Exact} discussed how to achieve local superlinear convergence by diminishing the coefficient $c_1 = c_{1,t}$ gradually as $t$ increases. However, since $c_1$ is proportional to~the magnitude of stochastic noise in the Hessian estimates, diminishing it requires increasing the sample size for estimating the Hessian, which results in a blow-up of the per-iteration computational cost. To be specific, $c_1$ is proportional to the reciprocal of the square root of the sample size; thus this blow-up in terms of the sample size can be as fast as exponential if we wish to attain the quadratic convergence rate, and linear if we wish to attain the superlinear convergence rate established in this paper later.

\vskip5pt
\item The presence of the failure probability $\delta$ in \eqref{equ:LQ:error} means that after $T$ iterations, a convergence guarantee may only hold with probability $1-T\delta$. Thus, the failure probability explodes when $T\rightarrow \infty$. To resolve this issue~one~can~gradually diminish $\delta$, e.g., $\delta = \delta_t = O(1/t^2)$ such that $\sum_t \delta_t<\infty$. However,~once again, such adjustment on $\delta$ leads to an increasing sample/sketch size as~the algorithm proceeds, although not as drastically as in (a) (it suffices to~increase the sample size logarithmically). Thus, in our stochastic oracle model, where the noise level (and hence the sample size) remains constant, it is problematic to show any convergence with high probability from \eqref{equ:LQ:error}~as~$T\rightarrow \infty$.
\end{enumerate}

We note that some prior works \citep{Bollapragada2018Exact, Meng2020Fast} showed the convergence in expectation guarantees. Although the explosion of the failure probability in (b) is suppressed by the direct characterization of the~expectation, the drawback (a) still remains. In addition, the convergence in expectation results require stronger assumptions. For example, \cite[(2.17)]{Bollapragada2018Exact} and \cite[Theorem 2]{Meng2020Fast} showed similar recursions to \eqref{equ:LQ:error} for $\mE[\|\xb_t - \ttxb\|]$. However, these works assumed a \textit{bounded moments of iterates} condition, i.e., $\mE[\|\xb_t-\tx\|^2]\leq \gamma(\mE[\|\xb_t-\tx\|])^2$ for some constant $\gamma>0$, and assumed each component $f_i$ in \eqref{equ:SSN:obj} to be strongly convex. Both conditions~are not needed in high probability analyses. Note that the majority of high probability analyses assume that each component $f_i$ is convex though not necessarily strongly convex \citep{Agarwal2017Second, RoostaKhorasani2018Sub}, while we do not impose any conditions on $f_i$, which significantly weakens the conditions in the existing literature. Furthermore, the stepsize in \cite{Bollapragada2018Exact, Meng2020Fast} was prespecified without any adaptivity. \cite{Bellavia2019Subsampled} introduced a non-monotone line search to address the adaptivity issue, while extra conditions such as the compactness of $\{\xb_t\}_t$ were imposed for their expectation analysis.

\subsection{Main results: Stochastic Newton with Hessian averaging}

Concerned by the above limitations of stochastic Newton methods, we ask the following question: 
\begin{quote}
\begin{em}
Can we design a stochastic Newton method that exploits global linear and local superlinear convergence in high probability, even for an infinite iteration sequence, and without increasing the computational cost in each iteration?
\end{em}
\end{quote}

We provide an affirmative answer to this question by studying a class of stochastic Newton methods with Hessian averaging. A simple intuition~is~that~the approximation error $\Eb_t$ can be de-noised by aggregating all the errors $\{\Eb_i\}_{i=0}^{t}$, inspired by the central limit theorem for martingale differences. Thus, if we could reuse the past samples and replace $\Eb_t$ by $\frac{1}{t+1}\sum_{i=0}^t\Eb_i$, then the matrix $\Hb_t + \frac{1}{t+1}\sum_{i=0}^t\Eb_i$ would be more precise than $\hHb_t = \Hb_t+\Eb_t$. However,~since~$\{\Eb_i\}_{i=0}^t$ are unknown and only $\{\hHb_i\}_{i=0}^t$ are known to us, in order~to~de-noise~$\Eb_t$, we can only aggregate all the Hessian estimates $\{\hHb_i\}_{i=0}^t$~and~derive~$\frac{1}{t+1}\sum_{i=0}^t\hHb_i = \frac{1}{t+1}\sum_{i=0}^t\Hb_i + \frac{1}{t+1}\sum_{i=0}^t\Eb_i$. Compared to $\hHb_t$, such a matrix does not preserve the information of the true Hessian $\Hb_t$. Thus, we observe a trade-off between de-noising the error $\Eb_t$ and preserving the Hessian information $\Hb_t$: on one hand, we want to assign equal weights to all the past errors to achieve the fastest concentration for the error average (see Remark \ref{rem:2}); on the other hand, we want to assign all weights to $\Hb_t$ to fully preserve the most recent Hessian information. In this paper, we investigate this trade-off, and show how the Hessian averaging resolves the drawbacks of existing stochastic Newton methods.

\subsubsection{The proposed method}

We consider an online averaging scheme (we let $\tHb_{-1} = \0$):
\begin{equation}\label{equ:ave:Hessian}
\tHb_t  = \frac{w_{t-1}}{w_t} \tHb_{t-1} + \rbr{1 - \frac{w_{t-1}}{w_t}}\hHb_t, \quad t = 0,1,2,\ldots,\\
\end{equation}
where $\{w_t\}_{t=-1}^{\infty}$ is a prespecified increasing non-negative weight sequence~starting with $w_{-1} = 0$. By \emph{online} we mean that we only keep an average Hessian $\tHb_t$ in each iteration, and update it as \eqref{equ:ave:Hessian} when we obtain a new Hessian estimate. This scheme is in contrast to keeping all the past Hessian estimates. By the scheme \eqref{equ:ave:Hessian}, we note that, at iteration $t$, we re-scale the weights of $\{\hHb_i\}_{t=0}^{t-1}$ equally by a factor of $w_{t-1}/w_t$, instead of completely redefining all the weights. We note that such an online averaging scheme is memory and computation efficient: compared to stochastic Newton methods without averaging, we only require an extra $O(d^2)$ space to save $\tHb_t$ and $O(d^2)$ flops to update it, which is negligible in the setting of stochastic Newton methods where the Hessian vector product requires $O(d^2)$ flops and solving the exact Newton system requires $O(d^3)$ flops. In \eqref{equ:ave:Hessian}, we use the ratio factors $w_{t-1}/w_t$ instead of direct weights merely to simplify our later presentation. Furthermore, \eqref{equ:ave:Hessian} can be reformulated as a general weighted average as follows:\vskip-0.32cm
\begin{equation}\label{equ:z}
\tHb_t =\sum_{i=0}^tz_{i,t}\hHb_i,\qquad\text{with}\quad
  z_{i,t}\coloneqq(w_i-w_{i-1})/w_t.
\end{equation}\vskip-0.1cm
\noindent In particular, by setting $w_t=t+1$, we obtain $z_{i,t} = 1/(t+1)$ and further have $\tHb_t = \frac1{t+1}\sum_{i=0}^{t}\hHb_i$. Thus, we recover simple uniform averaging. If we let the sequence $w_t$ grow faster-than-linearly,~this results in a weighted average that is skewed towards the more recent Hessian estimates.

\begin{algorithm}[!t]
\caption{Stochastic Newton Method with Hessian Averaging}
\label{alg:1}
\begin{algorithmic}[1]
\State \textbf{Input}: iterate $\xb_0$, weights $\{w_t\}_{t=-1}^{\infty}$, scalars $\beta\in(0, 1/2)$, $\rho\in(0, 1)$, and $\tHb_{-1} = \0$;
\For{ $t = 0,1,2,\ldots$}
\State Obtain a stochastic Hessian approximation $\hHb_t = \Hb_t + \Eb_t$;
\State Compute the Hessian average via \eqref{equ:ave:Hessian}: $\tHb_t  = \frac{w_{t-1}}{w_t} \tHb_{t-1} + \rbr{1 - \frac{w_{t-1}}{w_t}}\hHb_t$;
\State Compute the Newton direction $\pb_t$ by solving: $\tHb_t\pb_t = -\nabla f_t$;\vspace{1mm}
\State \textbf{if} $\tHb_t\pb_t = -\nabla f_t$ is not solvable \textbf{or} $\nabla f_t^\top\pb_t\geq 0$, \ \textbf{then} \ skip iteration with~$\xb_{t+1}=\xb_t$; 
\State Compute a stepsize $\mu_t = \rho^{j_t}$, where $j_t$ is the smallest nonnegative integer such that
\begin{equation*}
\text{Armijo condition}:\hskip 0.5cm f(\xb_t+\mu_t\pb_t) \leq f(\xb_t) + \mu_t\beta \nabla f_t^\top\pb_t
\end{equation*}
\State Update $\xb_{t+1} = \xb_t +\mu_t\pb_t$;
\EndFor
\end{algorithmic}
\end{algorithm}

Our proposed method replaces $\Hb_t$ by $\tHb_t$ when computing the Newton direction \eqref{equ:Newton}. The detailed procedure is displayed in Algorithm~\ref{alg:1}.~We make two comments about the algorithm.

First, Algorithm \ref{alg:1} supposes that, unlike the Hessian, the function values~and gradients are~known deterministically. As a result, the method \mbox{generates}~a~monotonically decreasing sequence of $f(\xb_t)$. If, on the other hand, $f(\xb_t)$ and $\nabla f(\xb_t)$ were known with random noise, we would have to relax the Armijo condition by adding extra error terms (hence, $f(\xb_t)$ would be potentially non-monotonic), and analyze the resulting method under a random model framework such as in \cite{Blanchet2019Convergence, Chen2017Stochastic}. We leave these non-trivial extensions to future work.

Second, for early iterates, the average Hessian $\tHb_t$ may not be a good~approximate of $\Hb_t$. For example, $\tHb_t$ may be indefinite (or even singular) so that $\pb_t$ is not a descent direction. In that case, we skip the line search step~and~let $\xb_{t+1}=\xb_t$ (see Line 6 of Algorithm \ref{alg:1}). Further, even if $\tHb_t$ is positive definite, it may not be properly bounded, and thus leads to an extremely small stepsize~$\mu_t$. However, as analyzed later in Lemmas \ref{lem:3} and \ref{lem:8}, the errors~$\{\Eb_i\}_{i=0}^t$~are~sufficiently concentrated for large $t$ with high probability, so that $\tHb_t$ is properly bounded from above and below. Thus, Algorithm \ref{alg:1} is well defined and can be performed for any iteration $t$, while it behaves like the classical Newton~method only after a few iterations (the threshold is provided in Lemmas \ref{lem:3} and \ref{lem:8}).

\subsubsection{Main results}

\vskip-0.2cm
We conduct convergence analysis for Algorithm \ref{alg:1} with different weight sequences $\{w_t\}_{t=-1}^\infty$. Throughout the analysis, we only assume the oracle noise $\Eb(\xb)$ has a sub-exponential tail, which in particular includes Hessian subsampling and Hessian sketching as special cases. Our convergence guarantees rely on the quality of the average Hessian approximation $\tHb_t$; thus, we do not require $\hHb_t$ to be a good~approximation of $\Hb_t$. In other words, our scheme is applicable even if we generate a \textit{single} sample for forming the subsampled Hessian estimator, and applicable even if some components $f_i$ are non-convex. This is because the noise $\Eb_t$ can always be reduced by averaging (ensured by the central limit theorem) even if it has a large variance.

We show that, with high probability, Algorithm \ref{alg:1} has four convergence~phases with three transition points:
\begin{enumerate}[label=(\alph*),topsep=1pt] 
\item Global phase: $\xb_t$ converges from any initial iterate $\xb_0$ to a local neighborhood of $\tx$, in which the unit stepsize starts being accepted.
\item Steady phase: $\xb_t$ stays in the neighborhood.
\item Slow superlinear phase: $\xb_t$ starts converging superlinearly with a rate~gradually increasing.
\item Fast superlinear phase: the superlinear acceleration reaches its full potential and is maintained for~all~the following iterations.
\end{enumerate}
\noindent We mention that the superlinear rate is measured with respect to the error $\|\xb_t-\ttxb\|_{\ttHb} \coloneqq \sqrt{(\xb_t-\ttxb)^\top\ttHb(\xb_t-\xb_0)}$ where~$\ttHb = \Hb(\ttxb)$. Before introducing the main results in Theorems \ref{thm:main-1} and \ref{thm:main-2}, we summarize them in Table \ref{tab:1}, showing the transition points for~two weight sequences. The transitions for general weights are provided~in Section~\ref{sec:4}. We use $\Upsilon$ to denote the \textit{noise level} of stochastic Hessian oracle (see Assumption \ref{ass:3} for a formal definition; typically $\Upsilon = O(\kappa)$). We also use $O(\cdot)$ to suppress the logarithmic factors and the dependence on other constants except $\Upsilon$ and $\kappa$. We emphasize that all results of the~paper require that the (weighted) average of the errors $\{\Eb_i\}_{i=0}^t$ is sufficiently concentrated; thus, they hold \textit{with high probability}.

\begin{table}
\centering
\begin{tabular}{c|c|c|c}
Weights $w_t$ & First transition & Second transition & Third transition \\ 
\hline \noalign{\vskip 2pt}
~$t+1$ \hfill {\scriptsize (Thm.~\ref{thm:main-1})\hspace{-2mm}~}& $O(\kappa^2+\Upsilon^2)$ & $O(\kappa\cdot \{\kappa^2+\Upsilon^2\})$ & $O(\kappa^2/\Upsilon^2\cdot\{\kappa^2+\Upsilon^2\}^2 )$\\
\hline \noalign{\vskip 2pt}
$(t+1)^{\log(t+1)}$ \hfill {\scriptsize (Thm.~\ref{thm:main-2})\hspace{-2mm}~}& $O(\kappa^2+\Upsilon^2)$ & $O(\kappa^2+\Upsilon^2)$ & $O(\kappa^2+\Upsilon^2)$ 
\end{tabular}
\caption{Three transitions of two averaging schemes: uniform averaging ($w_t=t+1$, see Theorem~\ref{thm:main-1}) and our proposed weighted averaging scheme ($w_t=(t+1)^{\log(t+1)}$, see Theorem~\ref{thm:main-2}). For each weight sequence, the three transitions occur with high probability.} \label{tab:1}

\vspace{-3mm}\end{table}
%\vskip-0.7cm
The first main result studies the uniform averaging scheme, which is informally~stated below. We refer to Theorem \ref{thm:3} for a formal statement. 
%\vspace{-2mm}
\begin{theorem}[Uniform averaging; informal]\label{thm:main-1}
Consider Algorithm \ref{alg:1} with $w_t = t+1$. With high probability, the algorithm satisfies: \vspace{-1mm}
\begin{enumerate}
\item After $\T = O(\kappa^2+\Upsilon^2)$ iterations, $\xb_t$ converges to a local neighborhood of~$\ttxb$.
\item After $O(\T\kappa)$ iterations, $\xb_t$ converges superlinearly to $\tx$.
\item After $O(\T^2\kappa^2/\Upsilon^2)$ iterations, the superlinear rate reaches $(\Upsilon\sqrt{\log(t)/t}\,)^{t}$, and this rate is maintained as $t\rightarrow \infty$.
\end{enumerate}
\end{theorem}

First, we observe that our convergence guarantee holds with high probability for the entire infinite iteration sequence, which addresses the issue of the blow-up of failure probability associated with the existing stochastic Newton methods (see part (b) in Section \ref{subsec:stochastic-newton}). 

Second, the parameter $\Upsilon$ characterizes the noise level of the stochastic~Hessian oracle. When the Hessian is generated by subsampling or sketching, $\Upsilon$ depends on the adopted sample/sketch sizes. As illustrated in Examples \ref{exp:1}~and \ref{exp:2}, $\Upsilon=O(\kappa)$ for popular Hessian estimators when sample/sketch sizes are independent of $\kappa$. In this case, $\T = O(\kappa^2)$. We require $\T \geq O(\Upsilon^2)$ only to~ensure that $\{\tHb_t\}_{t\geq \T}$ are positive definite with condition numbers scaling as $\kappa$, so that the Newton step based on $\tHb_t$ leads to a usual decrease of the objective. On the other hand, popular stochastic Newton methods often generate a larger sample size (which depends on $\kappa$) to enforce $\|\Eb_t\|\leq \epsilon\lambda_{\min}$ for any $t\geq 0$ with an $\epsilon\in(0,1)$ \citep[e.g., see Lemma 2 and Equation 3.10 in][respectively]{RoostaKhorasani2018Sub,Pilanci2017Newton}. In that case, $\{\hHb_t\}_{t\geq 0}$ are positive definite and so are $\{\tHb_t\}_{t\geq 0}$. Importantly, our method does not require such well-conditioned Hessian oracles.

Third, we notice that the uniform averaging approach has three transitions outlined in Theorem \ref{thm:main-1}. For the iterations before $\T$, the error in function value decreases linearly, implying that $\xb_t$ converges $R$-linearly. The \emph{first transition} occurs after $\T$ iterations when $\xb_t$ reaches the local neighborhood,~where~second-order information starts being useful. $\T$ is also where the exact Newton~methods would reach quadratic convergence. However, the averaged Hessian estimates still carry inaccurate Hessian information from the global phase, which is only gradually forgotten. From $\T$ to $O(\T\kappa)$ iterations, the algorithm gradually~forgets the Hessian estimates in the global phase, while $\xb_t$ still converges~$R$-linearly. The \emph{second transition} occurs after $O(\T\kappa)$ iterations, when $\xb_t$ starts converging superlinearly (with a slow rate). The \emph{third transition} occurs after $O(\T^2\kappa^2/\Upsilon^2)$, when the superlinear rate is accelerated to $(\Upsilon\sqrt{\log(t)/t}\,)^{t}$ and stabilized. This rate comes from the central limit theorem of averaging out the oracle noise;~thus, this rate cannot be further improved.

Note that if we were to reset the averaged Hessian estimate $\tHb_t$ after the first transition (so that the Hessian average does not include information~from the global phase), then the algorithm would immediately reach the superlinear rate $(\Upsilon\sqrt{\log(t)/t}\,)^{t}$ after $\T$ iterations (i.e., all transitions would occur at once). However, the algorithm does not know a priori when this transition will occur. As a result, the uniform Hessian averaging incurs a potentially significant delay of up to $O(\T^2\kappa^2/\Upsilon^2)$ iterations before reaching the desired superlinear rate.

In our second main result, we address the delayed transition to superlinear convergence that occurs in the uniform averaging. Specifically, we ask:

\begin{quote}
\begin{em}
Does there exist a universal weighted averaging scheme that achieves superlinear convergence without a delayed transition, and without any prior knowledge about the objective?
\end{em}
\end{quote}

Remarkably, the answer to this question is affirmative. The weighted/non-uniform averaging scheme we present below puts more weight on recent Hessian estimates, so that the second-order information from the global phase is forgotten more quickly, and the transition to a fast superlinear rate occurs after only $O(\T)$ iterations. Thus, the superlinear convergence occurs without any~delay (up to constant factors). Such a scheme will necessarily have a slightly weaker superlinear rate than the uniform averaging as $t\rightarrow\infty$, but we show that this difference is merely an additional $O(\sqrt{\log t})$ factor (see Theorem \ref{thm:4} and Example \ref{exp:5} for a formal statement).

\begin{theorem}[Weighted averaging; informal]\label{thm:main-2}
Consider Algorithm \ref{alg:1} with $w_t=(t+1)^{\log(t+1)}$.
%and suppose that the oracle noise satisfies $\Upsilon \geq 1/\text{poly}(\kappa)$. 
With high probability, after $O(\T) = O(\kappa^2+\Upsilon^2)$ iterations,  $\xb_t$ achieves a superlinear convergence rate $(\Upsilon\log(t)/\sqrt{t}\,)^t$, which is maintained as $t\rightarrow\infty$.
\end{theorem}

%\begin{remark}
%The above condition on $\Upsilon$ only guards against a corner case where the noise is infinitesimally small (see Example \ref{exp:5}), in which case for all practical purposes we can treat the Hessian oracle as exact and we can use the standard Newton's method.
%\end{remark}

We note that, given some knowledge about the global/local transition~points (e.g., if the algorithm knows $\kappa$, or if some convergence criterion is used for estimating the transition point), it is possible to switch from the more conservative weighted averaging to the asymptotically more effective uniform averaging within one run of the algorithm. However, since knowing transition points is difficult and rare in practice, we leave such considerations to future work, and only focus on problem-independent averaging schemes.

It is also worth mentioning that this paper only considers a basic stochastic Newton scheme based on \eqref{equ:Newton}, where we suppose exact function and gradient information and solutions to the Newton systems are known exactly. Some literature allows one to access inexact function values and/or gradients, and/or apply Newton-CG or MINRES to solve the linear systems inexactly \citep{Fong2012CG, RoostaKhorasani2018Sub, Liu2021Convergence, Yao2021Inexact}. Applying our sample aggregation technique under these setups is~promising, but we defer it to future work. The basic scheme purely reflects the benefits of Hessian averaging, which is the main interest of this work. Additionally, some literature deals with non-convex objectives via stochastic trust region methods \citep{Chen2017Stochastic, Blanchet2019Convergence} or stochastic Levenberg-Marquardt methods \citep{Ma2019Globally}. The averaging scheme may not directly apply for these methods due to potential bias brought by Hessian modifications for addressing non-convexity, while our sample aggregation idea is still inspiring. We leave the generalization to non-convex objectives to future work as well. Some literature addressed the superlinearity of stochastic Newton methods under distributed or federated learning settings \citep{Islamov2021Distributed, Safaryan2021FedNL, Qian2022Basis}. These works are not fully compatible with our Hessian oracle framework, since they exploit some distributed nature of problem to produce Hessian estimates with noise diminishing to zero (as opposed to the bounded noise in this paper).

\subsection{Literature review}
\label{subsec:literature}

\vskip-0.14cm 
Stochastic Newton methods have recently received much attention. The popular Hessian approximation methods include subsampling and sketching.  

For subsampled Newton methods, aside from extensive empirical studies~on different problems \citep{Martens2010Deep, Martens2011Learning, Kylasa2019GPU, Xu2020Second}, the pioneering work in \cite{Byrd2011Use} established~the very first asymptotic global convergence by showing that $\|\nabla f_t\|\rightarrow 0$~as $t\rightarrow \infty$, while the quantitative rate is unknown. Furthermore, \cite{Byrd2012Sample, Friedlander2012Hybrid} studied~Newton-type algorithms with subsampled gradients and/or subsampled Hessians, and established global $Q$-linear convergence in the error of function value $f(\xb_t)-f(\tx)$. However, the above analyses neglected the underlying~probabilistic~nature of 
the subsampled Hessian $\hHb_t$, and required $\hHb_t$ to be lower bounded away~from zero \textit{deterministically}. Such a condition holds only if each $f_i$ in \eqref{equ:SSN:obj} is strongly convex, which is restrictive in general. \cite{Erdogdu2015Convergence} relaxed such a condition by developing a novel algorithm, where the subsampled Hessian is adjusted by a truncated eigenvalue decomposition. With the exact gradient information and properly prespecified stepsizes, the authors showed a linear-quadratic error recursion for $\|\xb_t-\tx\|$ in high probability. Arguably, the convergence of standard subsampled Newton methods is originally analyzed in \cite{RoostaKhorasani2018Sub} and \cite{Bollapragada2018Exact} from different perspectives. In particular, for both sampling and not sampling the gradient, \cite{RoostaKhorasani2018Sub} showed a global~$Q$-linear convergence for $f(\xb_t) - f(\tx)$ and a local linear-quadratic convergence for $\|\xb_t-\tx\|$ in high probability. Under some additional conditions, \cite{Bollapragada2018Exact} derived a global $R$-linear convergence for the expected function value $\mE[f(\xb_t) - f(\tx)]$ and a (similar) local~linear-quadratic convergence for the expected iterate error $\mE[\|\xb_t-\tx\|]$. For both works, the authors also discussed how to gradually increase the sample size for Hessian approximation to achieve a local $Q$-superlinear convergence with high probability and in expectation, respectively. Building on the two studies, various~modifications of subsampled Newton methods have been reported with similar convergence guarantees. We refer to \cite{Xu2016Sub,Ye2017Approximate, Bellavia2019Subsampled, Li2020Do} and references therein. We note that \cite{Kovalev2019Stochastic} designed a scheme that allows for a single sample in each iteration of subsampled Newton. That work established a local linear convergence in expectation, while we obtain a superlinear convergence in high probability.

As a parallel approach to subsampling, Newton sketch has also been broadly investigated. \cite{Pilanci2017Newton} proposed a generic Newton sketch method that approximates the Hessian via a Johnson–Lindenstrauss (JL) transform (e.g., the Hadamard transform), and the gradient is exact. Furthermore, \cite{Agarwal2017Second, Derezinski2019Distributed,  Derezinski2020Debiasing, Derezinski2020Precise} proposed different Newton sketch methods with debiased or unbiased  Hessian inverse approximations. \cite{Derezinski2021Newton} relied on a novel sketching technique called Leverage~Score Sparsified (LESS) embeddings \citep{Derezinski2021Sparse} to construct a sparse sketch~matrix, and~studied the trade-off between the computational cost of $\hHb_t$ and the convergence rate of the algorithm. Similar to subsampled Newton methods, the aforementioned literature established a local linear-quadratic (or linear) recursion for $\|\xb_t-\ttxb\|$ in high probability. A recent work \cite{Lacotte2021Adaptive} adaptively increased the sketch size to let the linear coefficient be proportional to the iterate error, which leads to a quadratic convergence. However, the per-iteration computational cost is larger than typical methods. See \cite{Berahas2020investigation} for a review of subsampled and sketched Newton, and their connections to, and empirical comparisons with, first-order~methods.

Sketching has also been used to construct \emph{low-rank} Hessian approximations through the Sketch-and-Project framework, originally developed by \cite{Gower2015Randomized} for~solving linear systems, and extended to general convex/nonconvex optimization by \cite{Luo2016Efficient, Doikov2018Randomized, Gower2019RSN, Na2022Asymptotic}. The convergence properties of this family of~methods~have~been thoroughly studied: they achieve linear convergence in~expectation, with the~rate controlled by a so-called \emph{stochastic condition number}, which is defined as the smallest eigenvalue of the expectation of the low-rank projection matrix defined by the sketch \citep{Gower2015Randomized}. While the per-iteration cost of stochastic Newton methods based on Sketch-and-Project is generally lower than that of the aforementioned Newton sketch methods, their convergence rates are more sensitive to the spectral properties of the Hessian. The precise characterizations of the convergence rates are given in \cite{Mutny2020Convergence,Derezinski2020Precise, Derezinski2022Sharp}. Moreover, the Sketch-and-Project estimates are generally biased, so they are not appropriate for averaging.

In summary, none of the aforementioned existing works achieve superlinear convergence with a fixed per-iteration computational cost. Additionally, high probability convergence guarantees generally fail as $t\rightarrow\infty$, with potent~exceptions of certain stochastic trust-region methods \citep{Chen2017Stochastic} that enjoy almost sure convergence. However, the per-iteration computation of the exceptions~is~not~fixed and the local rate is unknown. Further, for finite-sum~objectives, the existing~literature on stochastic Newton methods assumes each $f_i$ to be strongly convex~\citep{Bollapragada2018Exact} or convex \citep{RoostaKhorasani2018Sub}. However, $f_i$ needs not be convex even if $f$ is strongly convex.~See \cite{Garber2015Fast, Garber2016Faster, ShalevShwartz2016SDCA, AllenZhu2016Improved} and references therein for first-order algorithms designed under such a setting. We address the above limitations of stochastic Newton methods by reusing all the past samples to average Hessian estimates. Our scheme is especially preferable when we have a limited budget for per-iteration computation (e.g., when we use very few samples in subsampled Newton, resulting in an ill-conditioned Hessian estimate). Our established non-asymptotic superlinear rates are stronger than the existing results, and our numerical experiments demonstrate the superiority of Hessian averaging.

\vskip2pt
\noindent{\bf Notation}: Throughout the paper, we use $\Ib$ to denote the identity matrix,~and~$\0$ to denote the zero vector or matrix. Their dimensions are clear from the~context. We use $\|\cdot\|$ to denote the $\ell_2$ norm for vectors and spectral norm for matrices.  For a positive semidefinite matrix $\Ab$, we let $\|\xb\|_{\Ab} = \sqrt{\xb^\top\Ab\xb}$. For two scalars $a, b$, $a\vee b = \max(a,b)$ and $a\wedge b = \min(a, b)$. For two matrices $\Ab, \Bb$, $\Ab \prec (\preceq)\Bb$ if $\Bb - \Ab$ is a positive (semi)definite matrix. Recall that we reserve the notation $\lambda_{\min}, \lambda_{\max}$ to denote the lower and upper bounds of the true~Hessian, and~$\kappa = \lambda_{\max}/\lambda_{\min}$ is the condition number.

\vskip3pt
\noindent{\bf Structure of the paper}: In Section \ref{sec:2} we present the preliminaries on matrix concentration that are needed for our results. Then, we establish convergence results for the uniform averaging scheme in Section \ref{sec:3}. Section \ref{sec:4} establishes convergence for general weight sequences. Numerical experiments and conclusions are provided in Sections \ref{sec:5} and \ref{sec:6}, respectively.

\section{Preliminaries on Matrix Concentration}\label{sec:2}

In this section, we present the key results on the concentration of sums of random matrices, which we then use to bound the noise of the averaged~\mbox{Hessian}~estimates. The Hessian estimates constructed by our algorithm are not independent, and hence standard matrix concentration results do not apply. However, they~do satisfy a martingale difference condition, which we will exploit to derive useful concentration results.

Given a sequence  of stochastic iterates $\{\xb_t\}_{t=0}^{\infty}$, we let $\mF_0\subseteq\mF_1\subseteq\mF_2\subseteq\ldots$ be a filtration where $\mF_t = \sigma(\xb_{0:t})$, $\forall t\geq 0$, is the $\sigma$-algebra~generated~by~the randomness from $\xb_0$ to $\xb_t$. With such a filtration, we denote the~conditional~expectation by~$\mE_t[\cdot] = \mE[\cdot \mid \mF_t]$ and the conditional probability by $\mP_t(\cdot) = \mP(\cdot \mid \mF_t)$. We suppose $\xb_0$ is deterministic, so that $\mF_0$ is a trivial $\sigma$-algebra.

For a given weight sequence $\{w_t\}_{t=-1}^\infty$ with $w_{-1}=0$, the scheme \eqref{equ:ave:Hessian} leads~to \vskip-0.4cm
\begin{equation}\label{equ:simple}
\tHb_t\;\; \stackrel{\mathclap{\eqref{equ:z}}}{=}\; \sum_{i=0}^{t}z_{i,t}\hHb_i \ \stackrel{\mathclap{\eqref{equ:error}}}{=} \!\!\underbrace{\sum_{i=0}^{t}z_{i,t}\Hb_i}_{\text{Hessian averaging}}\!\!+\!\!\quad\barEb_t\qquad\quad\text{for} \quad \barEb_t\coloneqq\!\!\underbrace{\sum_{i=0}^{t}z_{i,t}\Eb_i }_{\text{noise averaging}}
\end{equation}\vskip-0.1cm
\noindent where $z_{i,t} = (w_i - w_{i-1})/w_t$. Note that $\sum_{i=0}^t z_{i,t}=1$ and $z_{i,t}\propto z_i\coloneqq w_i-w_{i-1}$, i.e., $z_{i,t}$ is proportional to an un-normalized weight $z_i$. We see from \eqref{equ:simple} that~$\tHb_t$ consists of the Hessian averaging and noise averaging with the same weights. In principle, the Hessian averaging $\sum_{i=0}^{t}z_{i,t}\Hb_i\rightarrow \ttHb$ as $\xb_t\rightarrow \tx$, while the noise averaging $\sum_{i=0}^{t}z_{i,t}\Eb_i\rightarrow \0$ due to the central limit~theorem. We will show that the Hessian averaging (eventually) converges faster than the noise averaging.

To study the concentration of noise averaging $\barEb_t$, we use the fact that~$\{\Eb_t\}_{t=0}^\infty$ is a martingale difference sequence, and rely on concentration inequalities for matrix martingales. These concentration inequalities require a sub-exponential tail condition on the noise.  We say that a random variable $X$ is $K$-sub-exponential if $\mE[|X|^p]\leq p!\cdot K^p/2$ for all $p=2,3,...\,$, which is consistent (up to constants) with all standard notions of sub-exponentiality \citep[see Section 2.7 in][]{Vershynin2018High}.

\begin{assumption}[Sub-exponential noise]\label{ass:3}
We assume that  $\Eb(\xb)$ is mean zero and $\|\Eb(\xb)\|$ is $\Upsilon_E$-sub-exponential for all $\xb$. Also, we define $\Upsilon\coloneqq\Upsilon_E/\lambda_{\min}$ to be the scale-invariant noise level. 
\end{assumption}

\begin{remark}\label{rem:1}
The sub-exponentiality of $\|\Eb(\xb)\|$ implies that $\Eb(\xb)$ has sub-exponential matrix moments: $\mE[\Eb(\xb)^p]\preceq p!\cdot\Upsilon_E^p/2\cdot\Ib$ for $p=2,3,...\,$. In fact, our analysis immediately applies under this slightly weaker condition. We impose the moment condition on $\|\Eb(\xb)\|$ purely because it is easier to check in practice. Also, we note that sometimes the noise $\Hb(\xb)^{-1/2}\Eb(\xb)\Hb(\xb)^{-1/2}$ is more natural to study, e.g., for sketching-based oracles where we additionally have $\hHb(\xb)\succeq\0$. Thus, we can alternatively impose $\tilde{\Upsilon}$-sub-exponentiality on $\|\Hb(\xb)^{-1/2}\Eb(\xb)\Hb(\xb)^{-1/2}\|$. Our analysis can also be adapted to this alternate condition, and leads to~tighter convergence rates (in terms of the dependence on $\kappa$) for particular sketching-based oracles. However, this adaptation loses certain generality, and thus we prefer to impose conditions directly on the oracle noise $\Eb(\xb)$.

\end{remark}

Assumption \ref{ass:3} is weaker than assuming $\|\Eb(\xb)\|$ to be uniformly bounded by $\Upsilon_E$, and it is~satisfied by all of the popular Hessian subsampling and sketching methods. For example, when using Gaussian sketching, the noise is not bounded but sub-exponential. Further, the sub-exponential constant~$\Upsilon_E$~(and hence $\Upsilon$) reflects how the stochastic noise depends on the sample/sketch size, as illustrated in the examples below (see Appendix \ref{a:examples} for rigorous proofs).

\begin{example}\label{exp:1}
Consider subsampled Newton as in \eqref{equ:SSN} with sample size $s=|\xi_t|$,~and suppose $\|\nabla^2f_i(\xb)\|\leq \lambda_{\max}R$ for some $R>0$ and for all $i$. Then, we have~$\Upsilon = O(\kappa R \sqrt{\log(d)/s})$. If we additionally assume that all $f_i(\xb)$ are convex, then $\Upsilon$ is improved to $\Upsilon = O(\sqrt{\kappa R\log(d)/s}+\kappa R\log(d)/s\,)$.
\end{example}

\begin{example}\label{exp:2}
Consider Newton sketch as in \eqref{equ:SKN} with $\Sb\in \mR^{s\times n}$ consisting of i.i.d. $\mathcal N(0,1/s)$ entries. Then, we have $\Upsilon = O(\kappa(\sqrt{d/s}+ d/s))$.
\end{example}

From the above two examples, we observe that $\Upsilon$ scales as $O(\kappa)$ when holding everything else fixed. Also, Example \ref{exp:1} illustrates that our Hessian~oracle model applies to subsampled Newton even when some components $f_i(\xb)$~are~non-convex (while $f(\xb)$ is still convex), although this adversely affects the sub-exponential constant. For Gaussian sketch in Example \ref{exp:2}, we can show that $\widetilde\Upsilon=O(\sqrt{d/s}+ d/s)$ (where $\widetilde\Upsilon$ was defined in Remark~\ref{rem:1}). Thus, the dependence on $\kappa$ can be avoided for the~sub-exponential constant of noise~$\Hb(\xb)^{-1/2}\Eb(\xb)\Hb(\xb)^{-1/2}$. Analogous noise bounds can be proved for other sketching matrices $\Sb$, including sparse sketches and randomized orthogonal transforms.

We now show concentration inequalities for $\barEb_t$ in \eqref{equ:simple} under Assumption~\ref{ass:3}. We state the following preliminary lemma, which is a variant of Freedman's inequality for matrix martingales.

\begin{lemma}[Adapted from Theorem 2.3 in \cite{Tropp2011Freedmans}]\label{lem:2}
Let $t\geq 0$ be a fixed integer. Consider a $d$-dimensional martingale difference $\{\Eb_i\}_{i=0}^{t}$ (i.e., $\mE_i[\Eb_i] = \0$). Suppose there exists a function $g_t:\Theta_t\rightarrow [0, \infty]$ with $\Theta_t\subseteq (0, \infty)$, and~a~sequence of matrices $\{\Ub_i\}_{i=0}^{t}$, such that for any $i = 0,1,\ldots,t$,\footnote{The matrix exponential is defined by power series expansion: $\exp(\Ab) = \Ib + \sum_{i=1}^{\infty} \Ab^i/ i!$.}
\begin{equation}\label{equ:exp:cond}
\mE_i\sbr{\exp\rbr{\theta\Eb_i}}\preceq \exp\rbr{g_t(\theta)\Ub_i} \quad \text{ almost surely for each } \theta\in \Theta_t.
\end{equation}
Then, we have for any scalars $\eta \geq 0$ and $\sigma^2>0$,
\begin{equation*}
\mP\rbr{\Big\|\sum_{i=0}^{t}\Eb_i\Big\| \geq \eta\; \text{ and } \; \Big\|\sum_{i=0}^{t}\Ub_i\Big\|\leq \sigma^2} \leq 2d\cdot \inf_{\theta\in \Theta_t}\exp\rbr{-\theta \eta + g_t(\theta)\sigma^2}.
\end{equation*}		
\end{lemma}

The function $g_t$ in \cite[Theorem 2.3]{Tropp2011Freedmans} is defined on the full positive~set $(0, \infty)$, but the proof applies to any subset $\Theta_t$. We use Lemma \ref{lem:2} to show~the~next~result.

\begin{theorem}[Concentration of sub-exponential martingale difference]\label{thm:2}
Under Assumption \ref{ass:3}, for any integer $t\geq 0$ and scalar $\eta\geq 0$, $\barEb_t$ in \eqref{equ:simple} satisfies
\begin{equation}\label{equ:subE}
\mP\rbr{\nbr{\barEb_t} \geq \eta} \leq 2d\cdot\exp\rbr{-\frac{\eta^2/2}{\Upsilon_E^2\sum_{i=0}^{t}z_{i,t}^2 + z_t^{(\max)}\Upsilon_E\eta } }
\end{equation}
where $z_t^{(\max)} = \max_{i\in \{0,\ldots,t\}}z_{i,t}$.
\end{theorem}

\begin{proof}
See Appendix \ref{pf:thm:2}.
\end{proof}

The martingale concentration in Theorem \ref{thm:2} matches the matrix Bernstein results for independent noises $\{\Eb_i\}_{i=0}^t$ \citep[cf. Theorems 6.1, 6.2 in][]{Tropp2011User}. For any~$\delta\in(0, 1)$, if we let the right hand side of \eqref{equ:subE} be $\delta/(t+1)^2$, then we obtain that, with probability at least $1 - \delta/(t+1)^2$,  
\begin{equation}\label{equ:sub:concen}
\small \nbr{\barEb_t} \leq 8\Upsilon_E\sqrt{\log\rbr{\frac{d(t+1)}{\delta}}}\rbr{\sqrt{\sum_{i=0}^{t}z_{i,t}^2} \; \vee \; \sqrt{\log\rbr{\frac{d(t+1)}{\delta}}} \cdot z_t^{(\max)}}.
\end{equation}
We provide the following remark to discuss the fastest concentration rate.

\begin{remark}\label{rem:2}
To achieve the fastest concentration rate, we minimize the right~hand side of \eqref{equ:sub:concen} under the~restriction $\sum_{i=0}^t z_{i,t} = 1$. Note that the minimum of both $\sum_{i=0}^t z_{i,t}^2$ and $z_t^{(\max)}$ is attained with equal weights, that is $z_{i,t} = 1/(t+1)$, for $ i = 0,1,\ldots, t$. Thus, the fastest concentration rate is attained~with~equal~weights. Furthermore, a union bound over $t$ leads to
\begin{equation}\label{equ:equal:con}
\mP\rbr{\forall t:\ \nbr{\barEb_t} \leq
8\Upsilon_E\bigg(\sqrt{\frac{\log(d(t+1)/\delta)}{t+1}} \vee
  \frac{\log(d(t+1)/\delta)}{t+1}\bigg)} 
\geq 1 - \sum_{t=0}^{\infty}
\frac{\delta}{(t+1)^2}= 1-\frac{\pi^2\delta}{6}.
\end{equation}
We note that the square root term $\sqrt{\log(d(t+1)/\delta)/t+1}$ dominates the error bound for large $t$. Recalling from \eqref{equ:simple} that $z_{i,t} = (w_i - w_{i-1})/w_t$, we~know~$w_t = t+1$ for the equal weights. If fact, the concentration rate of $\|\barEb_t\|$ relates to~the superlinear convergence rate of $\xb_t$ (see Theorem \ref{thm:main-1}), because, as shown in the following sections, the convergence rate of $\xb_t$ is proportional to $\|\barEb_t\|$ when $\xb_t$ is sufficiently close to $\tx$.
\end{remark}

\section{Convergence of Uniform Hessian Averaging}\label{sec:3}

We now study the convergence of stochastic Newton with Hessian averaging. We consider the uniform averaging scheme, i.e., $w_t = t+1$, $\forall t\geq0$. Our first result suggests that, with high probability, $\tHb_t \succ \0$ for all large $t$. This implies that the Newton direction $\pb_t = -(\tHb_t)^{-1}\nabla f_t$ will be employed from some $t$ onwards (cf. Line 6 of Algorithm \ref{alg:1}). We recall that $\Upsilon = \Upsilon_E/\lambda_{\min}$, and $e$ denotes the natural base.

\begin{lemma}\label{lem:3}
Consider Algorithm \ref{alg:1} with $w_t = t+1$, $\forall t\geq 0$. Under Assumption~\ref{ass:3}, we let $\delta, \epsilon\in(0, 1)$ with $d/\delta\geq e$. We also let
\begin{equation}\label{T1}
\T_1 =4\rbr{1\vee (8\Upsilon/\epsilon)}^2 \log\rbr{d/\delta\cdot\cbr{1\vee(8\Upsilon/\epsilon)}}.
\end{equation}
Then, with probability $1-\delta\pi^2/6$, the event
\begin{equation}\label{event:EE}
\EE = \bigcap_{t=\T_1}^{\infty}\cbr{\|\barEb_t\| \leq 8\Upsilon_E\sqrt{\frac{\log(d(t+1)/\delta)}{t+1}}}
\end{equation}
occurs, which implies $(1-\epsilon)\lambda_{\min}\cdot\Ib\preceq \tHb_t \preceq (1+\epsilon)\lambda_{\max}\cdot\Ib$, $\forall t\geq \T_1$.
\end{lemma}

\begin{proof}
See Appendix \ref{pf:lem:3}. 
\end{proof}

By Lemma \ref{lem:3}, we initialize the convergence analysis from the iteration~$t = \T_1$, and condition on the event $\EE$. For $0\leq t< \T_1$, the Newton system may or~may not be solvable and the lower and upper bounds of $\tHb_t$ may or may not scale~as $\lambda_{\min}$ and $\lambda_{\max}$ (cf. Line 6 of Algorithm \ref{alg:1}). Thus, for the iterates $\xb_{0:\T_1}$, we do not generally have guarantees on the convergence rate, but only know that the objective value is non-increasing, that is, $f(\xb_0)\geq \cdots \geq f(\xb_{\T_1})$.

We next provide a $Q$-linear convergence for the objective value $f(\xb_t)-f(\ttxb)$ for $t\geq \T_1$.

\begin{lemma}\label{lem:4}
Conditioning on the event \eqref{event:EE}, we let
\begin{equation*}
\phi = \frac{4\rho\beta(1-\beta)(1-\epsilon)}{\kappa^2(1+\epsilon)}
\end{equation*}
and have $f(\xb_{t+1})-f(\ttxb) \leq (1-\phi)(f(\xb_t) - f(\ttxb))$, $\forall t\geq \T_1$, which implies $R$-linear convergence of the iterate error,
\begin{align*}
\|\xb_t - \ttxb\| &\leq \cbr{\frac{2}{\lambda_{\min}}(f(\xb_0) - f(\ttxb))(1-\phi)^{t-\T_1}}^{1/2}, \quad t\geq \T_1,\\
\|\xb_t - \ttxb\|_{\ttHb} &\leq \cbr{2\kappa(f(\xb_0) - f(\ttxb))(1-\phi)^{t-\T_1}}^{1/2}, \quad\quad\; t\geq \T_1.
\end{align*}
\end{lemma}

\begin{proof}
See Appendix \ref{pf:lem:4}.
\end{proof}

We next show that $\xb_t$ stays in a neighborhood around $\ttxb$ for all large $t$.~For this, we need a Lipschitz continuity condition. 

\begin{assumption}[Lipschitz Hessian]\label{ass:4}
We assume $\Hb(\xb)$ is $L$-Lipschitz continuous. That is $\|\Hb(\xb_1) - \Hb(\xb_2)\|\leq L\|\xb_1 - \xb_2\|$ for any $\xb_1, \xb_2\in \mR^d$.
\end{assumption}

Combining Lemma \ref{lem:4} with Assumption \ref{ass:4} leads to the following corollary.

\begin{corollary}\label{cor:1}
Consider Algorithm \ref{alg:1} with $w_t=t+1$, $\forall t\geq 0$. Under Assumptions \ref{ass:3} and \ref{ass:4}, we let $\delta, \epsilon\in(0, 1)$ with $d/\delta \geq e$, and define the neighborhood $\N_{\nu}$ as 
\begin{equation}\label{Nnu}
\N_\nu = \{\xb: \|\xb - \ttxb\|_{\ttHb} \leq \nu \cdot \lambda_{\min}^{3/2}/L\}\quad \text{ for } \nu \in (0, 1].
\end{equation}
Then, with probability $1 - \delta\pi^2/6$, 
we have $\xb_t\in \N_\nu$, for all $t\geq \T$ where~$\T = \T_1 + \T_2$ with $\T_1$ defined in \eqref{T1} and
\begin{equation}\label{T2}
\T_2 =  \frac {\kappa^2(1+\epsilon)}{4\rho\beta(1-\beta)(1-\epsilon)}\log\rbr{\frac{3L^2(f(\xb_0) - f(\ttxb))}{\nu^2\lambda_{\min}^3}}.
\end{equation}	
\end{corollary}

\begin{proof}
See Appendix \ref{pf:cor:1}.
\end{proof}

Combining \eqref{T1} and \eqref{T2}, and using $O(\cdot)$ to neglect logarithmic factors~and all constants except $\kappa$ and $\Upsilon$, we have $\T = O(\Upsilon^2 + \kappa^2)$ with high probability. Building on Corollary \ref{cor:1}, we then show that the unit stepsize is accepted locally. 

\begin{lemma}\label{lem:5}
Under Assumption \ref{ass:4}, suppose $\pb_t = -(\tHb_t)^{-1}\nabla f_t$. Then $\mu_t = 1$ if $\xb_t \in \N_{\nu}$ and $(1-\psi)\Hb_t \preceq \tHb_t \preceq(1+\psi)\Hb_t$ with $\nu, \psi$ satisfying
\begin{equation}\label{cond:nupsi}
0<\nu \leq \frac{2}{3}(1/2-\beta), \quad\quad 0<\psi\leq \frac{1/2-\beta}{3/2-\beta}.
\end{equation}
\end{lemma}

\begin{proof}
See Appendix \ref{pf:lem:5}.
\end{proof}

The unit stepsize enables us to show a linear-quadratic error recursion.

\begin{lemma}\label{lem:6}
Under Assumption \ref{ass:4} and suppose $\pb_t = -(\tHb_t)^{-1}\nabla f_t$, $\xb_t \in \N_{\nu}$, and $(1-\psi)\Hb_t \preceq \tHb_t \preceq(1+\psi)\Hb_t$ with $\nu, \psi$ satisfying \eqref{cond:nupsi}. Then, we have
\begin{equation*}
\nbr{\xb_{t+1} - \ttxb}_{\ttHb} \leq 3\cbr{ \frac{L}{\lambda_{\min}^{3/2}}\|\xb_t-\ttxb\|_{\ttHb}^2 + \|\Ib - \Hb_t^{-1/2}\tHb_t \Hb_t^{-1/2}\| \cdot \|\xb_t - \ttxb\|_{\ttHb}}.
\end{equation*}
\end{lemma}

\begin{proof}
See Appendix \ref{pf:lem:6}.
\end{proof}

Lemma \ref{lem:6} suggests that $\|\xb_t-\ttxb\|$ exhibits local $Q$-linear convergence. 

\begin{corollary}\label{cor:linear}
Under Assumption \ref{ass:4} and suppose $\pb_t = -(\tHb_t)^{-1}\nabla f_t$, $\xb_t \in \N_{\nu}$, and $(1-\psi)\Hb_t \preceq \tHb_t \preceq(1+\psi)\Hb_t$ with $\nu, \psi$ satisfying \eqref{cond:nupsi}. Then, we have 
\begin{equation*}
\|\xb_{t+1} - \ttxb\|_{\ttHb} \leq 3(\nu+\psi)\|\xb_t - \ttxb\|_{\ttHb},
\end{equation*}
which implies linear convergence provided $3(\nu+\psi)<1$.
\end{corollary}

Given all the presented lemmas, we state the final convergence guarantee.

\begin{theorem}\label{thm:3}	
Consider Algorithm \ref{alg:1} with $w_t = t+1$, $\forall t\geq 0$. Under Assumptions \ref{ass:3}, \ref{ass:4}, we let $\delta \in(0, 1)$ satisfy $d/\delta\geq e$, and let $\epsilon, \nu\in(0, 1)$ satisfy
\begin{equation}\label{cond:epsnu}
\epsilon \vee \nu \leq \frac{1}{3}\frac{0.5-\beta}{1.5-\beta}\wedge \frac{1}{48}.
\end{equation}
Define the neighborhood $\N_{\nu}$ as in \eqref{Nnu}, and define $\T = \T_1 +\T_2$ with $\T_1$ given~by \eqref{T1} and $\T_2$ given by \eqref{T2}. We also let $\J = 4\T\kappa/\nu$. Then, with probability $1-\delta\pi^2/6$, we have that $\xb_{\T_1:\T+\J}$ converges $R$-linearly, $\xb_{\T:\T+\J}\in \N_{\nu}$, and
\begin{equation*}
\|\xb_{\T+\J +t +1} - \ttxb\|_{\ttHb} \leq 12 \rho_t\|\xb_{\T+\J +t} - \ttxb\|_{\ttHb}, \quad \forall t\geq 0
\end{equation*}
with
\begin{equation*}
\rho_t = \frac{4\T\kappa}{\T+\J+t+1} + 8\Upsilon\sqrt{\frac{\log(d(\T+\J+t+1)/\delta)}{\T+\J+t+1}}
\end{equation*}
satisfying $24\rho_t\leq 1$, $\forall t\geq 0$.

\end{theorem}

\begin{proof}
See Appendix \ref{pf:thm:3}.
\end{proof}

We note from Theorem \ref{thm:3} that the condition on $\epsilon$ and $\nu$ does not depend on unknown quantities of objective function and noise level. The convergence rate $\rho_t$ consists of two terms. The first term is due to the fact that $\T = O(\Upsilon^2+ \kappa^2)$ Hessians are accumulated in the global phase, and each of them contributes an error (in norm $\|\cdot\|_{\ttHb}$) as large as $\kappa$. Given these imprecise Hessians, the method cannot immediately converge superlinearly after $\T$ iterations. That is, $\rho_t \nleq 1$ if $\J =t=0$. We need $\J = O(\T\kappa)$ iterates to suppress the effect of these imprecise Hessians. The second term is due to the noise averaging, i.e., $\|\barEb_t\|$, which decays slower than the first term. Thus, for sufficiently large $t$, the noise averaging will finally dominate the convergence~rate.

We present the above observation in the following corollary. It suggests that the averaging scheme has three transition points; thus four convergence~phases.

\begin{corollary}\label{cor:2}
Under the setup of Theorem \ref{thm:3}, Algorithm \ref{alg:1} has three transitions:

\noindent(a): From $\xb_0$ to $\xb_{\T}$: the algorithm converges to a local neighborhood $\N_{\nu}$ from any initial point $\xb_0$. 

\noindent(b): From $\xb_{\T}$ to $\xb_{\T+\J}$: the sequence $\xb_t$ stays in the neighborhood $\N_{\nu}$.

\noindent\hskip1cm (Starting from $\xb_{\T_1}$, the sequence $\xb_t$ exhibits $R$-linear convergence)

\noindent(c): From $\xb_{\T+\J}$ to $\xb_{\T+\J+\K}$: the algorithm converges $Q$-superlinearly with
\begin{equation*}
\|\xb_{t +1} - \ttxb\|_{\ttHb} \leq 12 \rho_t^{(1)}\|\xb_{t} -
\ttxb\|_{\ttHb}
\qquad\text{for}\qquad \rho_t^{(1)} = \frac{8\T\kappa}{t+1},
\end{equation*}
where
\begin{equation*}
0\leq t-\T-\J\leq\K,\quad\quad
\K \coloneqq \frac{\T^2\kappa^2}{4\Upsilon^2\log(d\T/\delta)}
- \T - \J.
\end{equation*}
\noindent(d): From $\xb_{\T+\J+\K}$: the algorithm converges $Q$-superlinearly with 
\begin{equation*}
\|\xb_{t +1} - \ttxb\|_{\ttHb} \leq 12 \rho_t^{(2)}\|\xb_{t} -
\ttxb\|_{\ttHb}
\qquad\text{for}\qquad  \rho_t^{(2)} = 16\Upsilon\sqrt{\frac{\log(d(t+1)/\delta)}{t+1}},
\end{equation*}
where $t\geq\T+\J+\K$.
\end{corollary}

\begin{proof}
See Appendix \ref{pf:cor:2}. 
\end{proof}

For an infinite iteration sequence $\{\xb_t\}_{t=0}^\infty$ and with high probability, Corollary \ref{cor:2}(a) suggests that the first transition is $\T=O(\kappa^2+\Upsilon^2)$; Corollary \ref{cor:2}(b) suggests that the second transition is $\J = O(\T\kappa)$; Corollary \ref{cor:2}(c) suggests that the third transition is $\K = O(\T^2\kappa^2/\Upsilon^2)$; and Corollary \ref{cor:2}(d) suggests that the final rate is $\rho_t^{(2)} =O(\Upsilon\sqrt{\log t/t})$. This recovers Theorem \ref{thm:main-1}. Recalling that $\Upsilon$ typically does not exceed $\kappa$, in this case, we have $\T = O(\kappa^2)$, $\J = O(\kappa^3)$, and $\K = O(\kappa^6/\Upsilon^2)$. This suggests a trade-off between the final superlinear~rate and the final transition. When the oracle noise level $\Upsilon$ is small, a faster superlinear rate is finally attained, but $\K$ also increases, meaning that the time to attain the final rate is further delayed. By Examples \ref{exp:1} and \ref{exp:2}, $\Upsilon$ decays as sample/sketch size increases. Thus, the final superlinear rate is improved by increasing the sample/sketch size $s$, however the effect of this change may be delayed due~to the rate/transition trade-off. Fortunately, as we will see in the following section, this trade-off can be optimized via \emph{weighted} Hessian averaging.

\section{Convergence of Averaging with General Weights}\label{sec:4}

Although the superlinear convergence of stochastic Newton with uniform~Hessian averaging (Corollary~\ref{cor:2}) is promising, since the scheme eventually attains~the optimal superlinear rate implied by the central limit theorem, a clear drawback is the delayed transition---the scheme spends quite a long time before attaining the final rate. In this section, we study the relationship between transitions and general weight sequences. We consider performing Algorithm \ref{alg:1} with a weight sequence $w_t$ that satisfies the following general condition.

\begin{assumption}\label{ass:5}
We assume $w_t = w(t)$ for all integer $t\geq 0$, where $w(\cdot): \mR\rightarrow \mR$ is a real function satisfying (i) $w(\cdot)$ is twice differentiable; (ii) $w(-1)= 0$, $w(t)>0$, $\forall t\geq 0$; (iii) $w'(-1)\geq 0$; (iv) $w''(t)\geq 0$, $\forall t\geq -1$; (v) $w(t+1)/w(t) \vee w'(t+1)/w'(t) \leq \Psi$, $\forall t\geq 0$ for some $\Psi\geq 1$.
\end{assumption}

By the above assumption, we specialize the result of noise averaging in \eqref{equ:sub:concen} as follows.

\begin{lemma}\label{lem:7}
Under Assumptions \ref{ass:3} and \ref{ass:5}, for any $t\geq 0$, with probability $1-\delta/(t+1)^2$,
\begin{equation}\label{con:barE}
\|\barEb_t\| \leq 8\Psi\Upsilon_E\rbr{\sqrt{\log\rbr{\frac{d(t+1)}{\delta}}\frac{w'(t)}{w(t)}} \vee \log\rbr{\frac{d(t+1)}{\delta}}\frac{w'(t)}{w(t)}}.
\end{equation}	
\end{lemma}

\begin{proof}
See Appendix \ref{pf:lem:7}.
\end{proof}

Naturally, to have $\|\barEb_t\|$ concentrate, we require $\lim_{t\rightarrow\infty}\log(\frac{d(t+1)}{\delta})\frac{w'(t)}{w(t)}=0$. It is easy to see that for some weight sequences, such as $w(t) = \exp(t)-\exp(-1)$, such a requirement cannot be satisfied, which makes convergence fail. On the other hand, this is reasonable since $z_{i,t}\propto z_i = w_i-w_{i-1} = \exp(i)-\exp(i-1) = (1-1/e)\exp(i)$, which means that we assign an exponentially large weight to the current Hessian estimate. Such an assignment preserves recent Hessian information better, but it diminishes the previous estimates too quickly to~let~the noise averaging concentrate.

Given the above concentration results, we have a similar result to Lemma~\ref{lem:3}.

\begin{lemma}\label{lem:8}
Consider Algorithm \ref{alg:1} with $w_t$ satisfying Assumption \ref{ass:5}. Under~Assumption \ref{ass:3}, for any $\delta, \epsilon\in(0, 1)$, we let
\begin{equation}\label{I1}
\I_1 \coloneqq \I_1(\epsilon,\delta) = \sup_{t}\cbr{t: \log\rbr{\frac{d(t+1)}{\delta}}\frac{w'(t)}{w(t)} \geq \rbr{\frac{\epsilon}{8\Psi\Upsilon}\wedge 1}^2} + 1.
\end{equation}
Then, with probability $1-\delta\pi^2/6$, the event
\begin{equation}\label{event:EE:new}
\EE = \bigcap_{t=\I_1}^{\infty}\cbr{\|\barEb_t\| \leq 8\Psi\Upsilon_E\sqrt{\log\rbr{\frac{d(t+1)}{\delta}}\frac{w'(t)}{w(t)}}}
\end{equation}
occurs, which implies $(1-\epsilon)\lambda_{\min}\cdot\Ib\preceq \tHb_t \preceq (1+\epsilon)\lambda_{\max}\cdot\Ib$, $\forall t\geq \I_1$.
\end{lemma}

\begin{proof}
See Appendix \ref{pf:lem:8}. 
\end{proof}

We note that Lemma \ref{lem:4} and Corollary \ref{cor:1} still hold for general weight sequences. Thus, we let
\begin{equation}\label{I}
\I\coloneqq \I(\epsilon, \delta, \nu) = \I_1(\epsilon,\delta) + \T_2 \stackrel{\eqref{T2}}{=} \I_1(\epsilon, \delta) + \frac{\kappa^2(1+\epsilon)}{4\rho\beta(1-\beta)(1-\epsilon)}\log\rbr{\frac{3L^2(f(\xb_0) - f(\ttxb))}{\nu^2\lambda_{\min}^3}},
\end{equation}
and know that $\xb_t\in \N_{\nu}$ for all $t\geq \I$. Lemmas \ref{lem:5}, \ref{lem:6} and Corollary \ref{cor:linear} also carry over to the setting of general weight sequences. Building on these results, we state the final convergence guarantee.

\begin{theorem}\label{thm:4}
Consider Algorithm \ref{alg:1} with $w_t$ satisfying Assumption \ref{ass:5}. Under~Assumptions \ref{ass:3}, \ref{ass:4}, we let $\delta, \epsilon, \nu\in(0,1)$ and $\epsilon, \nu$ satisfy
\begin{equation}\label{epsnu}
\epsilon \vee \nu \leq  \frac{1}{5}\frac{0.5-\beta}{1.5-\beta} \wedge \frac{1}{48\Psi}.
\end{equation}
Define the neighborhood $\N_{\nu}$ as in \eqref{Nnu}, and define $\I$ as in \eqref{I}. We also let $\U$ be $w(\I+\U) = 2w(\I-1)\kappa/\nu$. Then, with probability $1-\delta\pi^2/6$, we have that $\xb_{\I_1:\I+\U}$ converges $R$-linearly, $\xb_{\I:\I+\U}\in \N_{\nu}$, and
\begin{equation}\label{result}
\|\xb_{\I+\U +t +1} - \ttxb\|_{\ttHb} \leq 6\theta_t\|\xb_{\I+\U +t} - \ttxb\|_{\ttHb},\quad \forall t\geq 0  ,
\end{equation}
with
\begin{equation}\label{theta}
\theta_t = \frac{6w(\I-1)\kappa}{w(\I+\U+t)} + 8\Psi\Upsilon\sqrt{\log\rbr{\frac{d(\I+\U+t+1)}{\delta}} \frac{w'(\I+\U+t)}{w(\I+\U+t)}}
\end{equation}
satisfying $12\Psi\theta_t\leq 1$, $\forall t\geq 0$.	
\end{theorem}

\begin{proof}
See Appendix \ref{pf:thm:4}. 
\end{proof}

The proof of Theorem \ref{thm:4} is more involved than the one of Theorem \ref{thm:3}.~In~particular, when dealing with a critical term $\sum_{j=\I}^{\I+\U+t}(w(j)-w(j-1))\|\xb_j-\tx\|_{\ttHb}$, the analysis of Theorem \ref{thm:3} uses the facts that $w(j)-w(j-1)=1$ and $\xb_j$~converges linearly (cf. Corollary \ref{cor:linear}). However, that analysis does not apply for~general weights, since plugging the setup $w(j)=j+1$ masks some critical~properties of $w(\cdot)$, such as $w'(j)\geq 0$ and $w(j+1)/w(j) \leq \Psi$, $\forall j\geq 0$ ($\Psi=2$ for $w(j)=j+1$). In contrast, for Theorem \ref{thm:4}, we separate the above term into two terms, $\sum_{j=\I}^{\I+\U}(w(j)-w(j-1))\|\xb_j-\tx\|_{\ttHb}$ and $\sum_{j=\I+\U+1}^{\I+\U+t}(w(j)-w(j-1))\|\xb_j-\tx\|_{\ttHb}$. The first term is bounded by $w(\I+\U)\nu$ since $\xb_j\in\N_{\nu}$ for $\I\leq j\leq \I+\U$.~The~second term, which is proven to have the same bound as the first term, not only utilizes the linear convergence of $\xb_j$ \textit{with a rate depending on $\Psi$} (proved~by~induction), but also utilizes general properties of $w(\cdot)$ in Assumption \ref{ass:5}.

We arrive at the following corollary for the iteration transitions. The proof~is the same as for Corollary \ref{cor:2}, and is omitted. 

\begin{corollary}\label{cor:3}
Under the setup of Theorem \ref{thm:4}, Algorithm \ref{alg:1} has three transitions:

\noindent(a): From $\xb_0$ to $\xb_\I$: the algorithm converges to a local neighborhood $\N_{\nu}$ from any initial point $\xb_0$.

\noindent(b): From $\xb_{\I}$ to $\xb_{\I + \U}$: the sequence $\xb_t$ stays in the neighborhood $\N_{\nu}$.

\noindent\hskip1cm (Starting from $\xb_{\I_1}$, the sequence $\xb_t$ exhibits $R$-linear convergence)

\noindent(c): From $\xb_{\I+\U}$ to $\xb_{\I+\U+\V}$: the algorithm converges $Q$-superlinearly with 
\begin{equation*}
 \|\xb_{t +1} - \ttxb\|_{\ttHb} \leq 6 \theta_t^{(1)}\|\xb_{t} - \ttxb\|_{\ttHb} \qquad\text{for}\qquad 	\theta_t^{(1)} \coloneqq \frac{14w(\I-1)\kappa}{w(t)},
\end{equation*}
where 
\begin{equation*}
0\leq t-\I-\U\leq \V,\qquad
\V  \coloneqq \argmin_{t\geq \I+\U}\cbr{ w(t)w'(t)\log\rbr{\frac{d(t+1)}{\delta}}\geq \frac{w(\I-1)^2\kappa^2}{\Psi^2\Upsilon^2}  }.
\end{equation*}
\noindent(d): From $\xb_{\I+\U+\V}$: the algorithm converges $Q$-superlinearly with 
\begin{equation*}
\|\xb_{t +1} - \ttxb\|_{\ttHb} \leq 6 \theta_t^{(2)}\|\xb_{t} - \ttxb\|_{\ttHb} \quad\text{for}\quad \theta_t^{(2)} = 14\Psi\Upsilon\sqrt{\frac{w'(t)}{w(t)}\log\rbr{\frac{d(t+1)}{\delta}}},
\end{equation*}
where $t\geq \I+\U+\V$.

\end{corollary}

Given Corollary \ref{cor:3}, we provide the following examples. Again, we use $O(\cdot)$ to neglect the logarithmic factors and the dependence on all constants except $\kappa$ and $\Upsilon$. We emphasize that for an iteration sequence, the three transition points occur only with high probability. For sake of presentation, we will not repeat ``with high probability" for each $O(\cdot)$. As mentioned, typically, $\Upsilon = O(\kappa)$ in practice. We note that for all considered $w_t$ sequences, Assumption \ref{ass:5} is satisfied and $\I_1$ has the same order as $\T_1$ (cf. \eqref{T1}), so that $\I = O(\T) = O(\Upsilon^2 + \kappa^2)$. In other words, the weighted averaging and uniform averaging take the same~order of iterations to get into the same local neighborhood. In the following examples, we show how the second and third transition points $\U$ and $\V$, as well as the superlinear rates $\theta_t^{(1)}$ and $\theta_t^{(2)}$, change for different weight sequences (see Table \ref{tab:transitions} for a summary of the case $\Upsilon \leq O(\kappa)$, i.e. $\I = O(\kappa^2)$).

\begin{example}[Uniform averaging]\label{exp:3}
Let $w_t = t+1$. Then, the superlinear rates~are:  \vskip-0.2cm
\begin{equation*}
\theta_t^{(1)} =O\rbr{\frac{\I\kappa}{t}},\qquad\qquad \theta_t^{(2)} =O\rbr{\Upsilon \sqrt{\frac{\log(dt/\delta)}{t}}}.
\end{equation*}
\vskip-0.1cm
\noindent Moreover, the second and third transition points are given by $\U = O(\I\cdot \kappa)$~and $\V =O(\I^2\kappa^2/\Upsilon^2)$. The above example recovers Corollary \ref{cor:2} (up to constant scaling factors). It achieves the best final rate $\theta_t^{(2)}$ ensured by the central limit theorem, but the initial rate $\theta_t^{(1)}$ is sub-optimal. Thus, this scheme takes a long time for the iterates to reach the final rate.
\end{example}

\begin{example}[Accelerated initial rate]\label{exp:4}
Let $w_t = (t+1)^p$ for $p \geq 1$. The superlinear rates are:  \vskip-0.3cm
\begin{equation*}
\theta_t^{(1)} =O\rbr{\frac{\I^p\kappa}{t^p}} ,\qquad\qquad \theta_t^{(2)} =O\rbr{ \Upsilon\sqrt{p\cdot \frac{\log(dt/\delta)}{t}}}.
\end{equation*}
Moreover, the second and third transition points are given by \begin{align*}
(\I+\U)^p = O(\I^p\kappa) & \Longleftarrow \U = O(\I\cdot \kappa^{1/p}), \\
p\V^{2p-1}\log(d\V/\delta) \geq \I^{2p}\kappa^2/\Upsilon^2 & \Longleftarrow  \V = O\Big(\I^{\frac{2p}{2p-1}}\cdot \cbr{\kappa^2/(p\Upsilon^2)}^{\frac{1}{2p-1}}\Big).
\end{align*}
In the above example, we substantially accelerate the initial superlinear rate $\theta_t^{(1)}$, while preserving the final rate $\theta_t^{(2)}$ up to a constant factor $p$. We observe that the transitions $\U$ and $\V$ are both improved upon Example \ref{exp:3}. In particular, the dependence of $\kappa$ in $\U$ is improved from $\kappa$ to $\kappa^{1/p}$, and the dependence of $\I$ in $\V$ is improved from $\I^2$ to $\I^{2p/(2p-1)}$. An important observation is that both $\U$ and $\V$ contract to $\I$ as $p$ goes to $\infty$, which inspires us to consider the following sequence. 
\end{example}

\begin{table}
\centering
\begin{tabular}{c||c|c|c|c}
\multirow{2}{*}{Weights $w_t$}
& \multicolumn{2}{c|}{Initial superlinear phase}
&\multicolumn{2}{c}{Final superlinear phase}\\
\cline{2-5}
& $\U$&rate $\theta_t^{(1)}$
& $\V$&rate $\theta_t^{(2)}$\\
\hline\hline
$t+1$\hfill {\scriptsize (Ex.~\ref{exp:3})\hspace{-2mm}~}&$\kappa^3$&$\kappa^3/t$ &$\kappa^6/\Upsilon^2$&$\Upsilon\sqrt{\log(t)/t}$
\\[1mm]
$(t+1)^p$\hfill {\scriptsize (Ex.~\ref{exp:4})\hspace{-2mm}~}&$\kappa^{2+1/p}$&$\kappa^{2p+1}/t^p$ &$\kappa^{\frac{4p+2}{2p-1}}/(p\Upsilon^2)^{\frac1{2p-1}}$
&$\Upsilon\sqrt{p\log(t)/t}$
\\[2mm]
$(t+1)^{\log(t+1)}$\hfill\ {\scriptsize (Ex.~\ref{exp:5})\hspace{-2mm}~}&$\kappa^2$&$\kappa^{4\log(\kappa)+1}/t^{\log(t)}$ &$\kappa^2$ &$\Upsilon\log(t)/\sqrt t$
\end{tabular}
\caption{Comparison of different averaging schemes, in terms of how many iterations it takes to transition to each superlinear phase, and the convergence rates achieved. We~drop constant factors as well as logarithmic dependence on $d$ and $1/\delta$, and assume that $1/\text{poly}(\kappa)\leq \Upsilon \leq O(\kappa)$.}\label{tab:transitions}
\end{table}

\begin{example}[Optimal transition points]\label{exp:5}
Let $w_t = (t+1)^{\log(t+1)}$. Then, the superlinear rates are: \begin{equation*}
\theta_t^{(1)}=O\rbr{\frac{\I^{\log(\I)}\kappa}{t^{\log (t)}}} ,\qquad\qquad \theta_t^{(2)}=O\rbr{\Upsilon\sqrt{\log(t)\frac{\log(dt/\delta)}{t}}},
\end{equation*}
We now derive the second and third transition points. In particular, $\U$ can be obtained from:
\begin{equation*}
(\I+\U)^{\log(\I+\U)} = O(\I^{\log \I} \kappa) \Longleftarrow \I + \U = O\rbr{\exp\rbr{\sqrt{\log\rbr{\I^{\log \I}\kappa}}}} = O(\I) \Longleftarrow \U = O(\I),
\end{equation*}
where the first implication uses the fact $y = x^{\log x} \Longleftrightarrow x = \exp(\sqrt{\log y})$, as~well as the fact $\log \kappa \leq (\log\I)^2$. The third transition $\V$ can be obtained from:
\begin{equation*}
 \V^{2\log \V - 1}\log \V \log d\V/\delta \geq \I^{2\log \I}\kappa^2/\Upsilon^2 
 \Longleftarrow (2\log \V-1)\log \V \geq (2\log \I+1)\log \I + 2\log\rbr{\frac{1}{\Upsilon}},
\end{equation*}
where the implication uses $\kappa^2\leq \I$. To let the right hand side hold, we let~$\V = \xi\cdot \I$ and require
\begin{align*}
& \rbr{\rbr{2\log\xi-2} + 2\log \I+1}\rbr{\log \I + \log \xi} \geq (2\log \I+1)\log \I + 2\log\rbr{\frac{1}{\Upsilon}} \\
& \Longleftarrow \begin{cases*}
2\log \xi -2 \geq 1\\
\log \xi \log \kappa \geq \log(1/\Upsilon)
\end{cases*} \Longleftarrow \xi \geq \exp\rbr{1.5\vee \log(1/\Upsilon)/\log \kappa}.
\end{align*}
Thus, the final transition point $\V$ can be chosen as \begin{equation*}
\V = O\bigg(\I\cdot\exp\Big(1+\frac{\log 1/\Upsilon}{\log\kappa}\Big)\bigg).
\end{equation*}
Note that, as long as the oracle noise $\Upsilon$ is bounded below as $\Upsilon\geq 1/\text{poly}(\kappa)$,~we have $\V=O(\I)$. Since, as we mentioned, in practice the oracle noise actually grows proportionally with $\kappa$, this is an extremely mild condition. However, it is technically necessary, since, if the Hessian oracle always returned the true Hessian, i.e., $\Upsilon=0$, then $\theta_t^{(2)}=0$, so we would necessarily have $\V=\infty$ (i.e., Hessian averaging is not helpful when we have the exact Hessian). In the above example, all of the transition points $\I, \U, \V$ are within constant factors of one another (not even logarithmic factors), while we sacrifice only $\sqrt{\log(t)}$ in the final superlinear rate. The above results recover Theorem~\ref{thm:main-2}. This suggests that, using the weight sequence $w_t=(t+1)^{\log(t+1)}$, the averaging scheme transitions from the global phase to the local superlinearly convergent phase smoothly,~and the local neighborhood that separates the two phases is consistent with the neighborhood of classical Newton methods that separate the global phase (i.e., the damped Newton phase) and the local quadratically convergent phase (cf. Lemma \ref{lem:6} with $\tHb_t=\Hb_t$). 
\end{example}

\section{Numerical Experiments}\label{sec:5}

We implement the proposed Hessian averaging schemes and compare them with popular baselines including BFGS, subsampled Newton methods and sketched Newton methods\footnote{All results can be reproduced via \url{https://github.com/senna1128/Hessian-Averaging}.}. We focus on the regularized logistic regression problem 
\begin{equation*}
\min_{\xb\in \mR^d} \;\; \frac{1}{n}\sum_{i=1}^{n}\log\Big(1+\exp\big(-b_i \ab_i^\top\xb\big)\Big) + \frac{\nu}{2}\nbr{\xb}^2,
\end{equation*}
where $\{(b_i, \ab_i)\}_{i=1}^n$ with $b_i\in\{-1,1\}$ are input-label pairs. We let $\Ab = (\ab_1, \ab_2, \ldots, \ab_n)^\top\in \mR^{n\times d}$ be the data matrix.

\vskip3pt
\noindent\textbf{Data generating process:} We generate several data matrices with $(n,d,\nu) = (1000,100,10^{-3})$, varying their properties. First, we vary the coherence of~$\Ab$, since higher coherence leads to higher variance for subsampled Hessian estimates, compared to sketched Hessian estimates. We use $\tau_{\Ab}$ to denote the coherence~of $\Ab$, which is defined as 
\begin{equation*}
\tau_{\Ab} = \frac{n}{d}\cdot \max_{i=1,\ldots,n}\|\Ub_{i,\cdot}\|^2 
\end{equation*}
where $\Ab =  \Ub\Sigmab\Vb^\top$ is the reduced singular value decomposition of $\Ab$ with~$\Ub\in \mR^{n\times d}$ and $\Vb\in \mR^{d\times d}$. Second, we vary the condition number $\kappa_{\Ab}$ of $\Ab$, since (as suggested by our theory), higher condition number leads to slower convergence and delayed transitions to superlinear rate. The detailed generalization of $\Ab$ is as follows. 
We fix $\Vb = \Ib\in \mR^{d\times d}$. For low coherence, we generate a $n\times d$~random matrix with each entry being independently generated from $\N(0,1)$, and let $\Ub$ be the left singular vectors of that matrix. For high coherence, we divide each row of $\Ub$ by $\sqrt{\zb_i}$, where $\zb_i$ is independently sampled from~Gamma~distribution with shape 0.5 and scale 2. We observe that the coherence of the low coherence matrix is $\approx 1$ (minimal), while the coherence of the high coherence matrix is $\approx 10=n/d$ (maximal). For either low or high coherence, we vary $\kappa_{\Ab}=\{d^{0.5}, d, d^{1.5}\}$.~For each $\kappa_{\Ab}$, we let singular values in $\Sigmab$ vary from $1$ to $\kappa_{\Ab}$ with equal spacing, and finally form the matrix $\Ab = \Ub\Sigmab$. We also generate~a~random~vector~$\xb\sim \N(\0, 1/d\cdot\Ib)$, and the response $b_i\in\{-1,1\}$ with $P(b_i = 1) = 1/(1+\exp(-\ab_i^\top\xb))$, $\forall i=1,\ldots,n$.  

\begin{table}[!th]
\begin{tabular}{c|c|c||c|c|c|c|c}
\multicolumn{3}{c||}{Setup} &  \multicolumn{3}{c|}{Newton Sketch} & Subsampled & \multirow{2}{*}{BFGS}\\
\cline{1-6}
  $\tau_{\Ab}$&$\kappa_{\Ab}$ & $s$ &  Gaussian& CountSketch&LESS-uniform& Newton &\\
\hline\hline
\multirow{12}{*}{1} & \multirow{4}{*}{$ d^{0.5}$}& $0.25d$ & 67/\textbf{29}/35 & 67/\textbf{29}/35 & 66/\textbf{29}/35 & 67/\textbf{29}/36 &\multirow{4}{*}{177}\\

&&$0.5d$ & 53/\textbf{22}/25 & 53/\textbf{22}/25 & 53/\textbf{21}/25 & 53/\textbf{22}/25 &\\
&& $d$ & 38/\textbf{17}/19 & 38/\textbf{17}/18 & 37/\textbf{17}/19 & 38/\textbf{16}/18 &\\ 
&& $5d$& 16/\textbf{11}/\textbf{11} & 16/\textbf{11}/\textbf{11} & 16/\textbf{11}/12 & 12/\textbf{9}/\textbf{9} &\\    

\cline{2-8}
&\multirow{4}{*}{$d$}& $0.25d$ & ---/\textbf{41}/52 & ---/\textbf{41}/52 & ---/\textbf{41}/53 & ---/\textbf{45}/60 &\multirow{4}{*}{219}\\

&&$0.5d$& ---/\textbf{31}/34 & ---/\textbf{31}/35 & ---/\textbf{31}/35 & ---/\textbf{34}/39 &\\
&&$d$ & 244/\textbf{24}/\textbf{24} & 246/\textbf{24}/\textbf{24} & 243/\textbf{24}/\textbf{24} & 315/\textbf{26}/\textbf{26} &  \\
&&$5d$ & 20/16/\textbf{14} & 20/17/\textbf{14} & 20/16/\textbf{14} & 18/15/\textbf{13} &  \\

 \cline{2-8}
&\multirow{4}{*}{$d^{1.5}$}& $0.25d$ & ---/---/\textbf{80} & ---/---/\textbf{81} & ---/---/\textbf{98} & ---/---/\textbf{371} &\multirow{4}{*}{295}\\

&& $0.5d$& ---/---/\textbf{67} &---/---/\textbf{66} &---/---/\textbf{72} &---/---/\textbf{217}&\\
&&$d$  & 270/---/\textbf{59} &---/---/\textbf{59} &---/---/\textbf{59} &---/---/\textbf{122}&  \\
&&$5d$  & \textbf{27}/---/54& 97/---/\textbf{54} & \textbf{38}/---/54&---/---/\textbf{54}&  \\
\hline

\multirow{12}{*}{10} & \multirow{4}{*}{$d^{0.5}$}& $0.25d$ & 60/\textbf{29}/34 & 60/\textbf{29}/33 & 60/\textbf{33}/37 & 57/60/\textbf{51} &\multirow{4}{*}{178}\\

&&$0.5d$ & 48/\textbf{22}/24 & 47/\textbf{22}/24 & 46/\textbf{25}/26 &  45/45/\textbf{41} &\\
&& $d$& 38/\textbf{17}/18 & 35/\textbf{17}/18 & 34/\textbf{20}/\textbf{20} & 34/36/\textbf{33} &\\    
&& $5d$& 15/\textbf{11}/\textbf{11} & 16/\textbf{11}/\textbf{11} & 16/\textbf{12}/\textbf{12} & 26/17/\textbf{15} &\\    
\cline{2-8}
 
&\multirow{4}{*}{$d$} & $0.25d$ & ---/91/\textbf{52} & ---/90/\textbf{51} & ---/117/\textbf{63} & ---/242/\textbf{147} &\multirow{4}{*}{252}\\

&&$0.5d$& ---/74/\textbf{35} & ---/74/\textbf{36} & ---/93/\textbf{44} & ---/164/\textbf{106} &\\
&&$d$ & 101/65/\textbf{27} & 121/63/\textbf{27} & 151/74/\textbf{31} & 308/114/\textbf{69} &  \\
&&$5d$ & 21/51/\textbf{19} & 27/52/\textbf{18} & 25/55/\textbf{19} & 77/74/\textbf{25} &  \\
\cline{2-8}
    
&\multirow{4}{*}{$d^{1.5}$}& $0.25d$ & ---/538/\textbf{80} & ---/549/\textbf{78} & ---/754/\textbf{110} & ---/---/\textbf{294} &\multirow{4}{*}{273}\\

&&$0.5d$& ---/518/\textbf{62} & ---/499/\textbf{63} &---/575/\textbf{77} & ---/946/\textbf{187} &\\
&&$d$  & 235/475/\textbf{51} & 429/489/\textbf{53} & 729/510/\textbf{59} &---/659/\textbf{121} &  \\
&&$5d$    & \textbf{35}/431/43 & 53/422/\textbf{43} & \textbf{42}/438/43 & 321/462/\textbf{51} &  
\end{tabular}
\caption{We show the median over 50 runs of the	number of iterations until \mbox{convergence}, i.e., until $ \|\xb_{t} - \ttxb\|_{\ttHb}\leq 10^{-6}$, where ``---'' indicates exceeding 999 iterations. In the case of stochastic methods, we provide three numbers for (the best one in bold): \textsf{NoAvg}, \textsf{UnifAvg}, and \textsf{WeightAvg}, which correspond to the standard method (without Hessian averaging), the method with uniform averaging ($w_t=t+1$, as in Theorem~\ref{thm:main-1}), and the method with weighted averaging ($w_t=(t+1)^{\log(t+1)}$, as in Theorem~\ref{thm:main-2}), respectively. For the setup, we use $\tau_{\Ab}$ to denote the coherence number of $\Ab$; $\kappa_{\Ab}$ to denote the condition number of $\Ab$; and  $s$ to denote the sample/sketch size for the stochastic~methods.} \label{tab:main}  
\end{table}

\vskip3pt
\noindent\textbf{Methods:} We implement the deterministic BFGS method, and stochastic subsampled and sketched Newton methods. Given a batch size $s$, the subsampled Newton method computes the Hessian as 
\begin{equation*}
\hHb(\xb) = \frac{1}{s}\sum_{j\in \xi}l_j\cdot \ab_j\ab_j^\top + \nu\cdot \Ib
\quad \text{ with }\quad l_j =
\frac{\exp(-b_j\ab_j^\top\xb)}{\big(1+\exp(-b_j\ab_j^\top\xb)\big)^2}, 
\end{equation*}
where the index set $\xi$ has size $s$ and is generated uniformly from $[1,n]$ without replacement. The sketched Newton method instead computes the Hessian as 
\begin{equation*}
\hHb(\xb) = \Ab^\top\Db^{1/2}\Sb^\top\Sb\Db^{1/2}\Ab + \nu\cdot\Ib \quad \text{ with } \quad \Db = \diag\rbr{l_1,\ldots, l_n},
\end{equation*}
where $\Sb\in \mR^{s\times n}$ is the sketch matrix. We consider three sketching methods, one dense (slow and accurate) and two sparse variants (fast but less accurate). 

\noindent (i) Gaussian sketch: we let $\Sb_{i,j}\stackrel{iid}{\sim} \N(0, 1)$. 

\noindent (ii) CountSketch \citep{Clarkson2017Low}: for each column of $\Sb$, we randomly pick a row and realize a Rademacher variable. 

\noindent (iii) LESS-Uniform \citep{Derezinski2021Newton}: for each row of $\Sb$, we randomly pick a fixed number of nonzero entries and realize a Rademacher variable in each nonzero entry.~We~let the number of nonzero entries in each row to be $0.1d$. 

\noindent For (i)--(iii), the sketches are scaled appropriately so that $\mE[\hHb(\xb)]=\Hb(\xb)$.

For each of the above Hessian oracles (both sketched and subsampled), we consider three variants of the stochastic Newton methods, depending on how the final Hessian estimate is constructed. 

\noindent (1) \textsf{NoAvg}: the standard method that uses the oracle estimate directly; 

\noindent (2) \textsf{UnifAvg}: the uniform Hessian averaging (i.e., $w_t=t+1$ from Theorem~\ref{thm:main-1}); 

\noindent (3) \textsf{WeightAvg}: the universal weighted Hessian averaging (i.e., $w_t=(t+1)^{\log(t+1)}$ from Theorem~\ref{thm:main-2}).

\begin{figure}[!th]
	\centering     %%% not \center
	\subfigure[$\tau_{\Ab}\approx1$, $\kappa_{\Ab}= \sqrt{d}$, $s=d$]{\label{kappa1}\includegraphics[width=50mm]{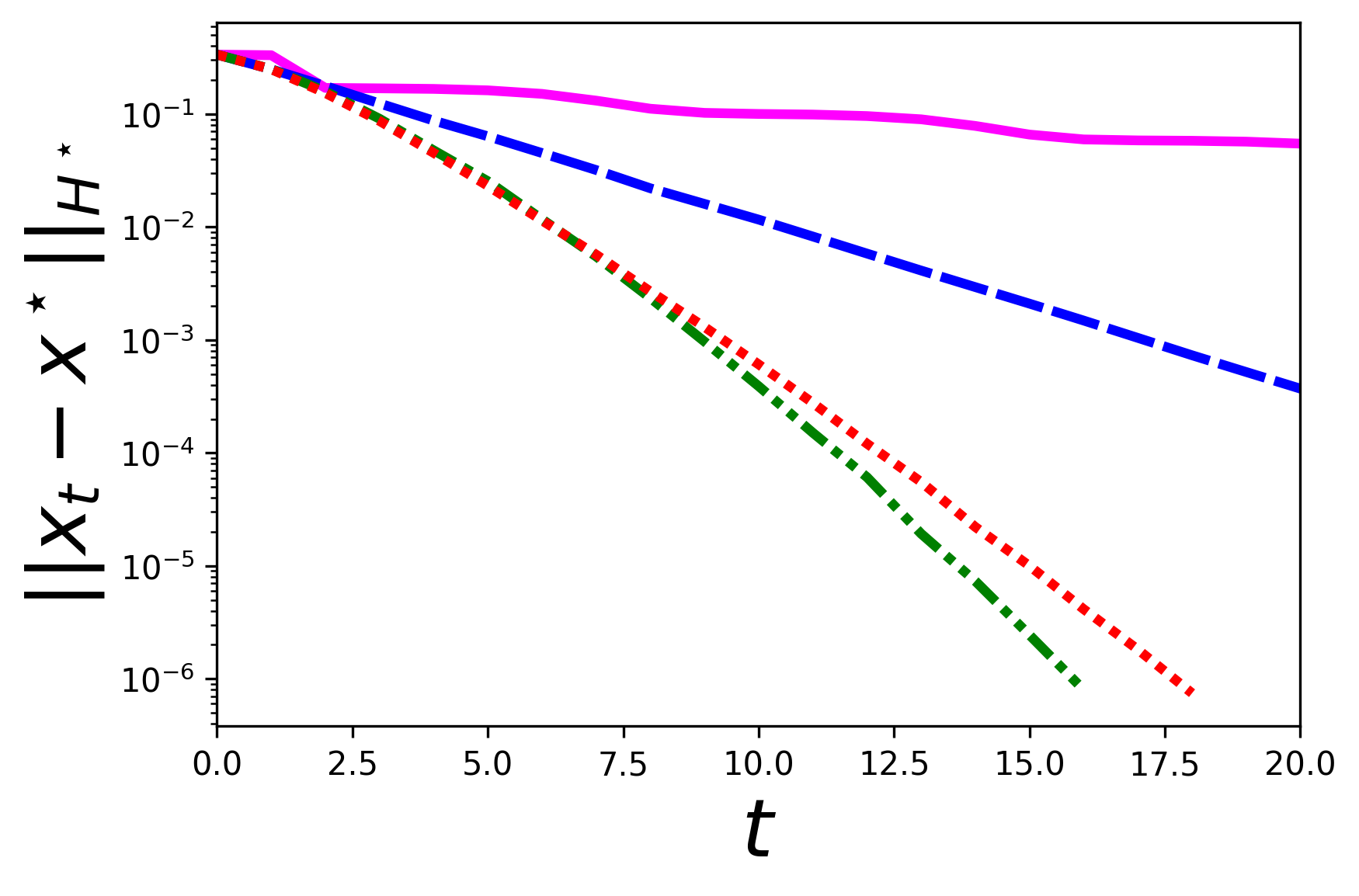}}
	\subfigure[$\tau_{\Ab}\approx1$, $\kappa_{\Ab} = d$, $s=d$]{\label{kappa2}\includegraphics[width=50mm]{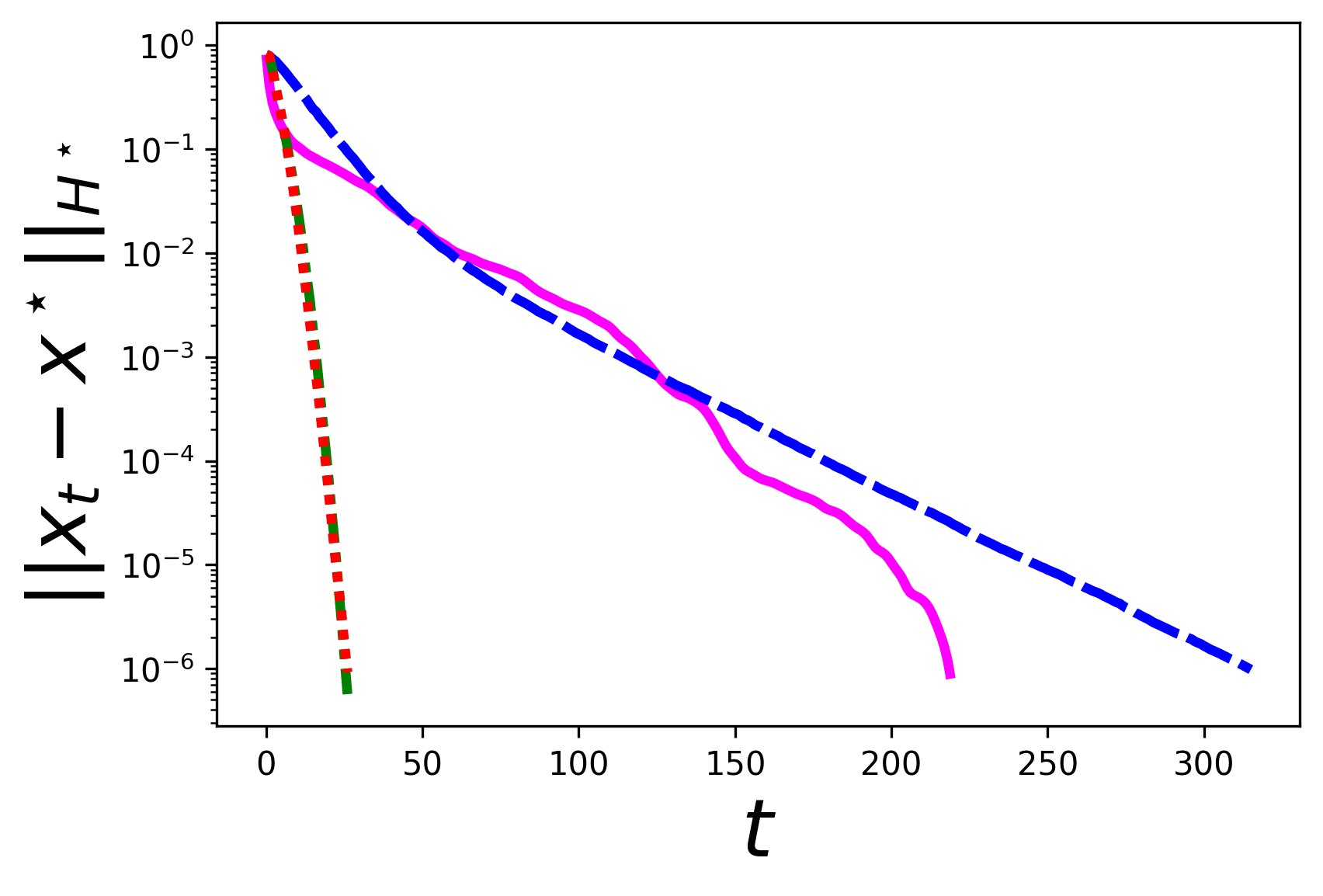}}
	\subfigure[$\tau_{\Ab}\approx1$, $\kappa_{\Ab} = d^{1.5}$, $s=d$]{\label{kappa3}\includegraphics[width=50mm]{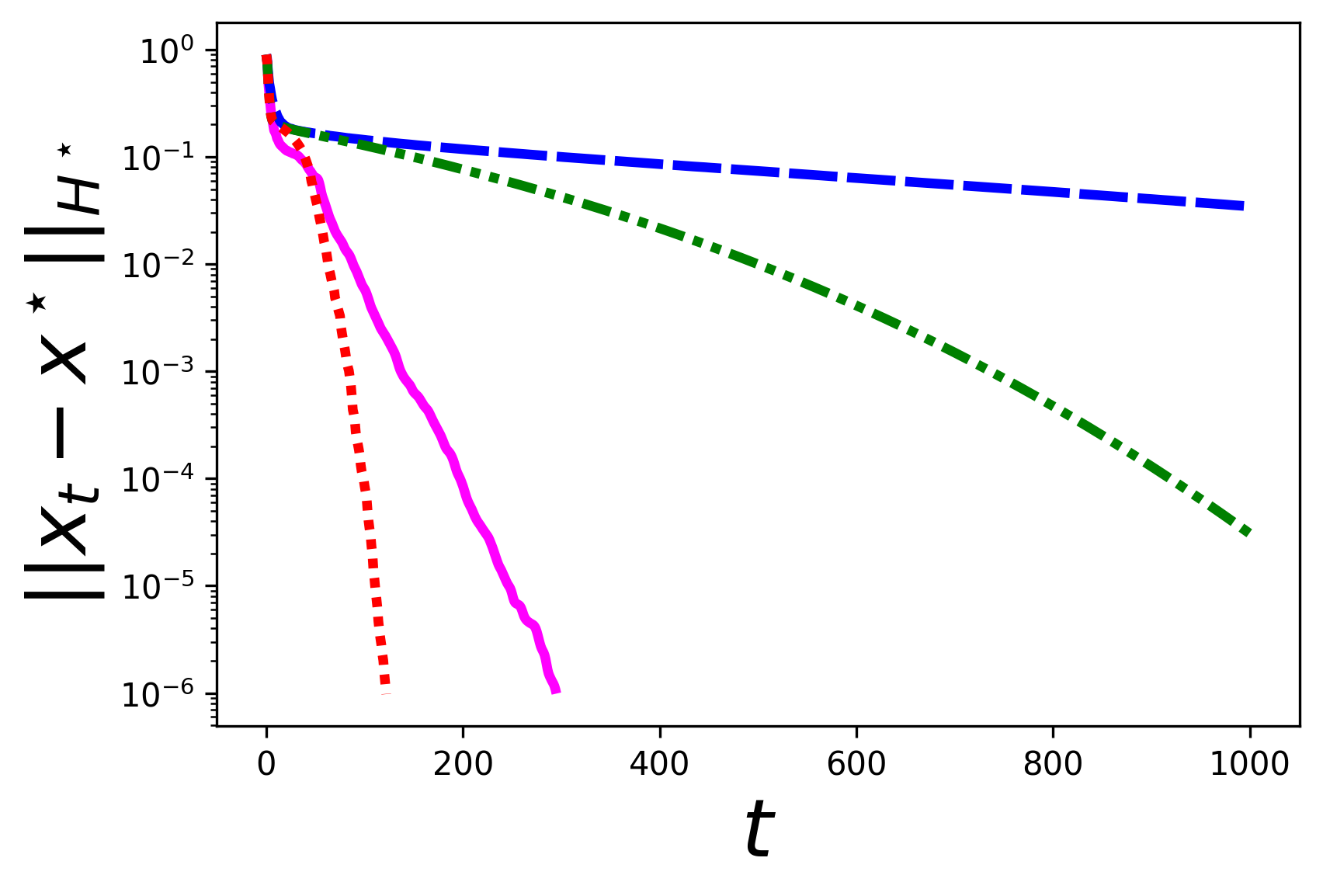}}
	
	\subfigure[$\tau_{\Ab}\approx1$, $\kappa_{\Ab} = \sqrt{d}$, $s=5d$]{\label{kappa4}\includegraphics[width=50mm]{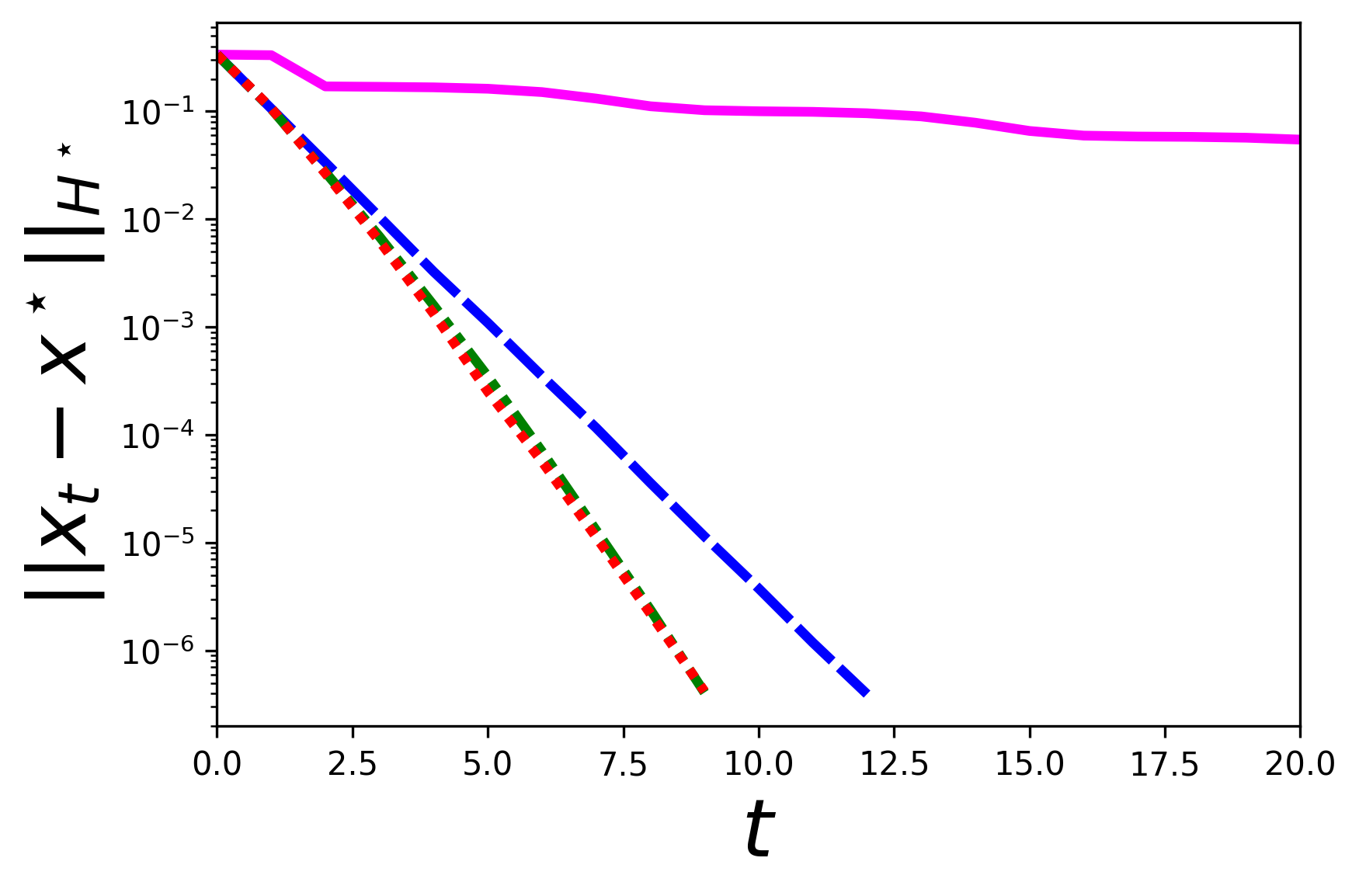}}
	\subfigure[$\tau_{\Ab}\approx1$, $\kappa_{\Ab} = d$, $s=5d$]{\label{kappa5}\includegraphics[width=50mm]{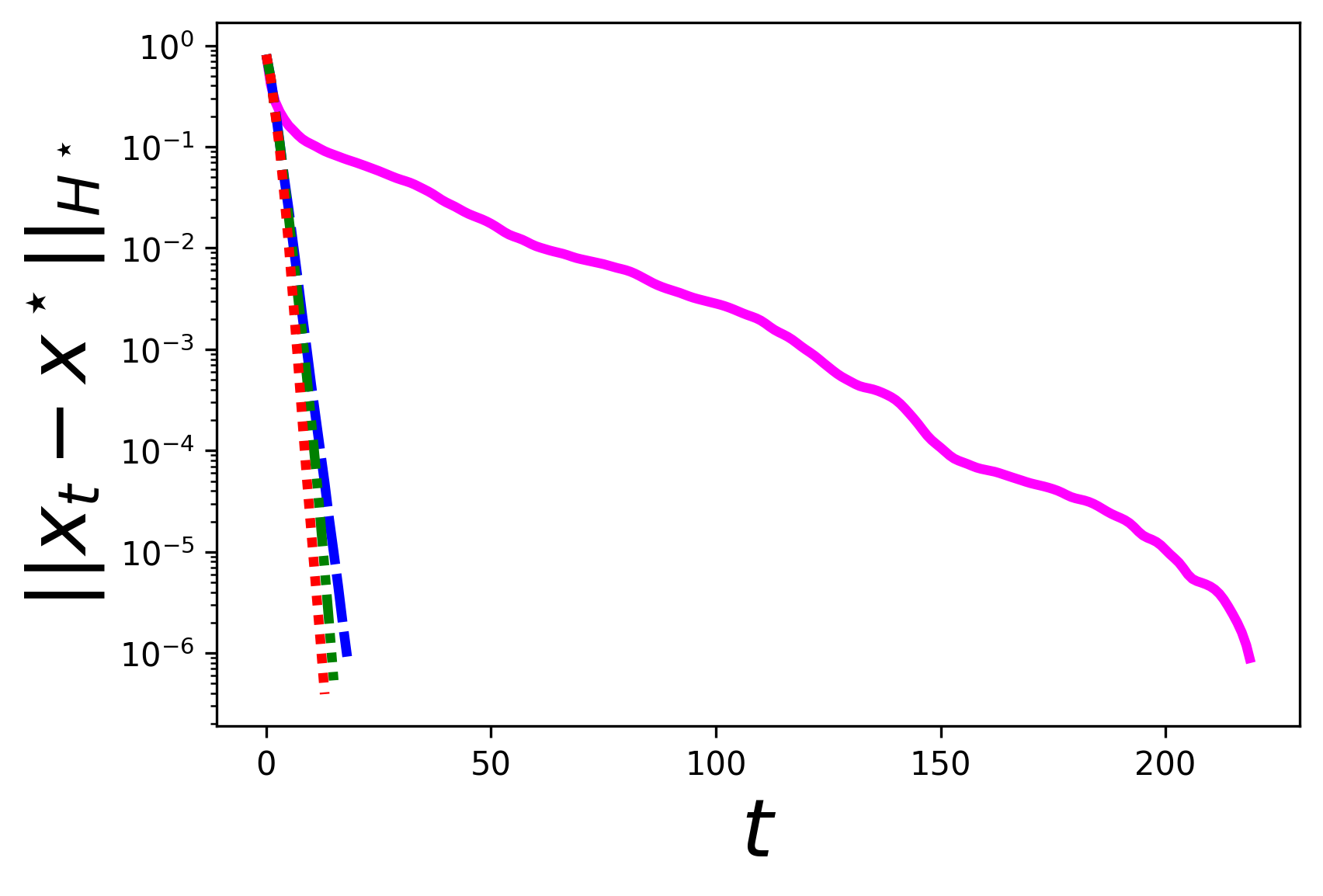}}
	\subfigure[$\tau_{\Ab}\approx1$, $\kappa_{\Ab} = d^{1.5}$, $s=5d$]{\label{kappa6}\includegraphics[width=50mm]{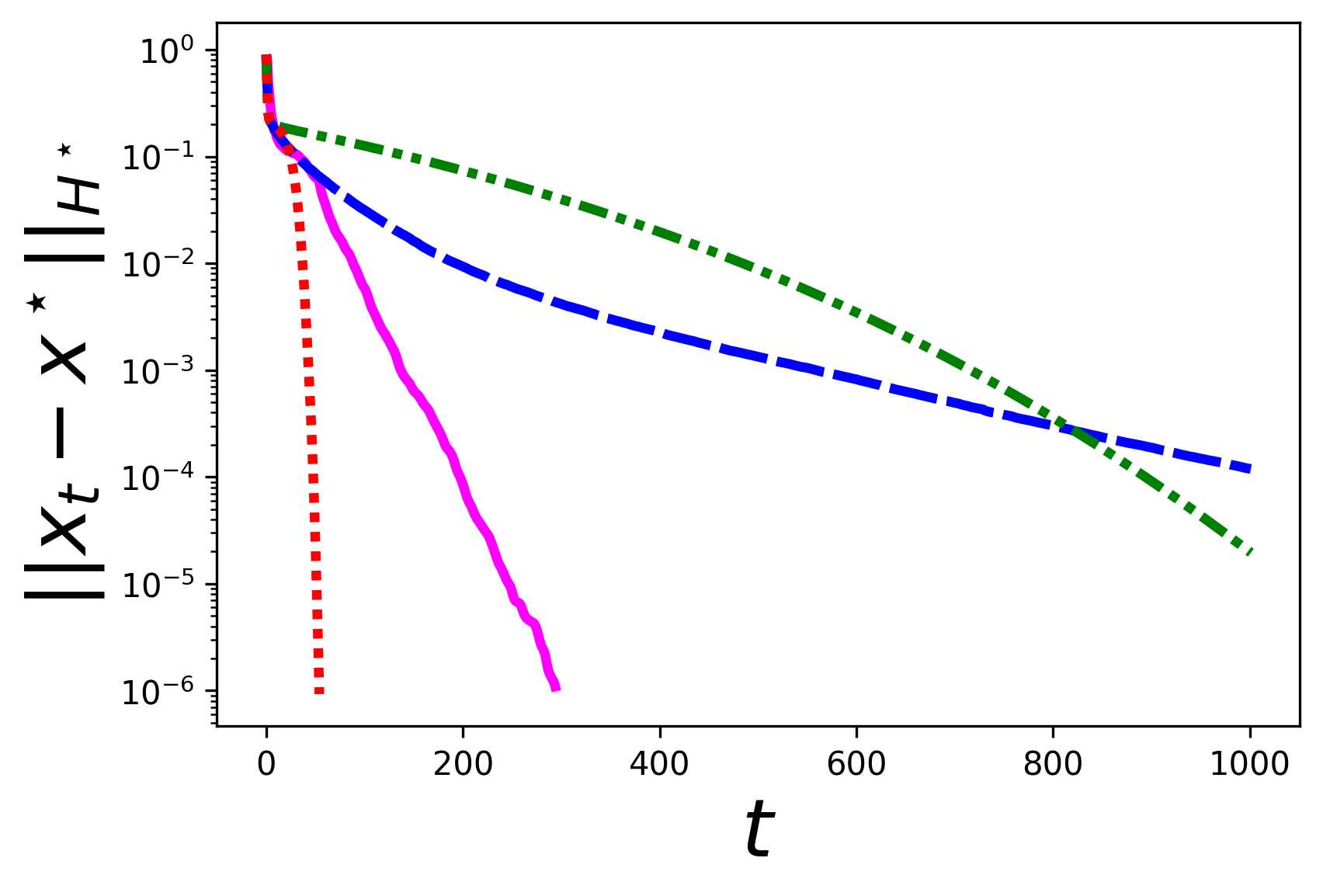}}
	
	\subfigure[$\tau_{\Ab}\approx10$, $\kappa_{\Ab}= \sqrt{d}$, $s=d$]{\label{kappa11}\includegraphics[width=50mm]{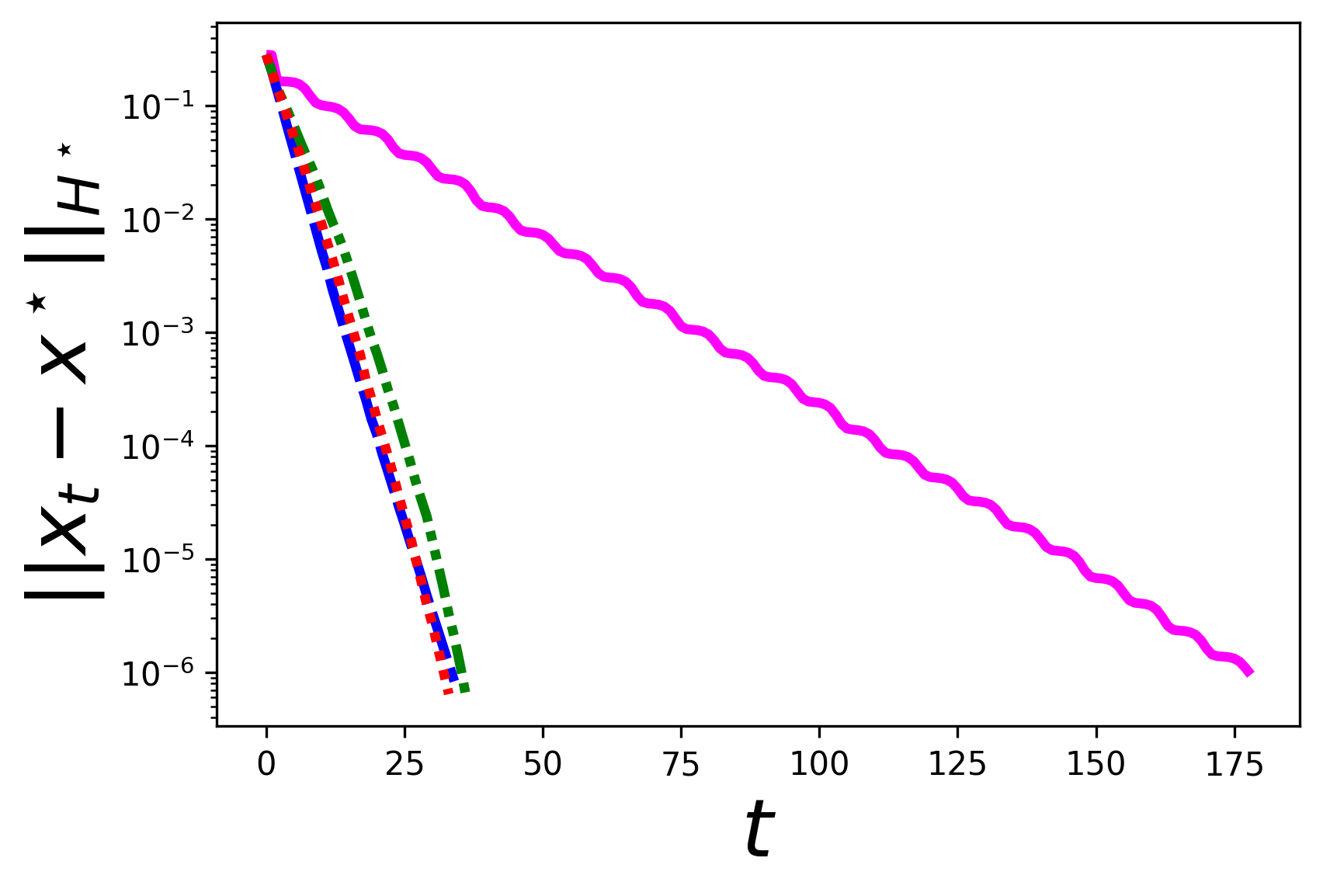}}
	\subfigure[$\tau_{\Ab}\approx10$, $\kappa_{\Ab} = d$, $s=d$]{\label{kappa12}\includegraphics[width=50mm]{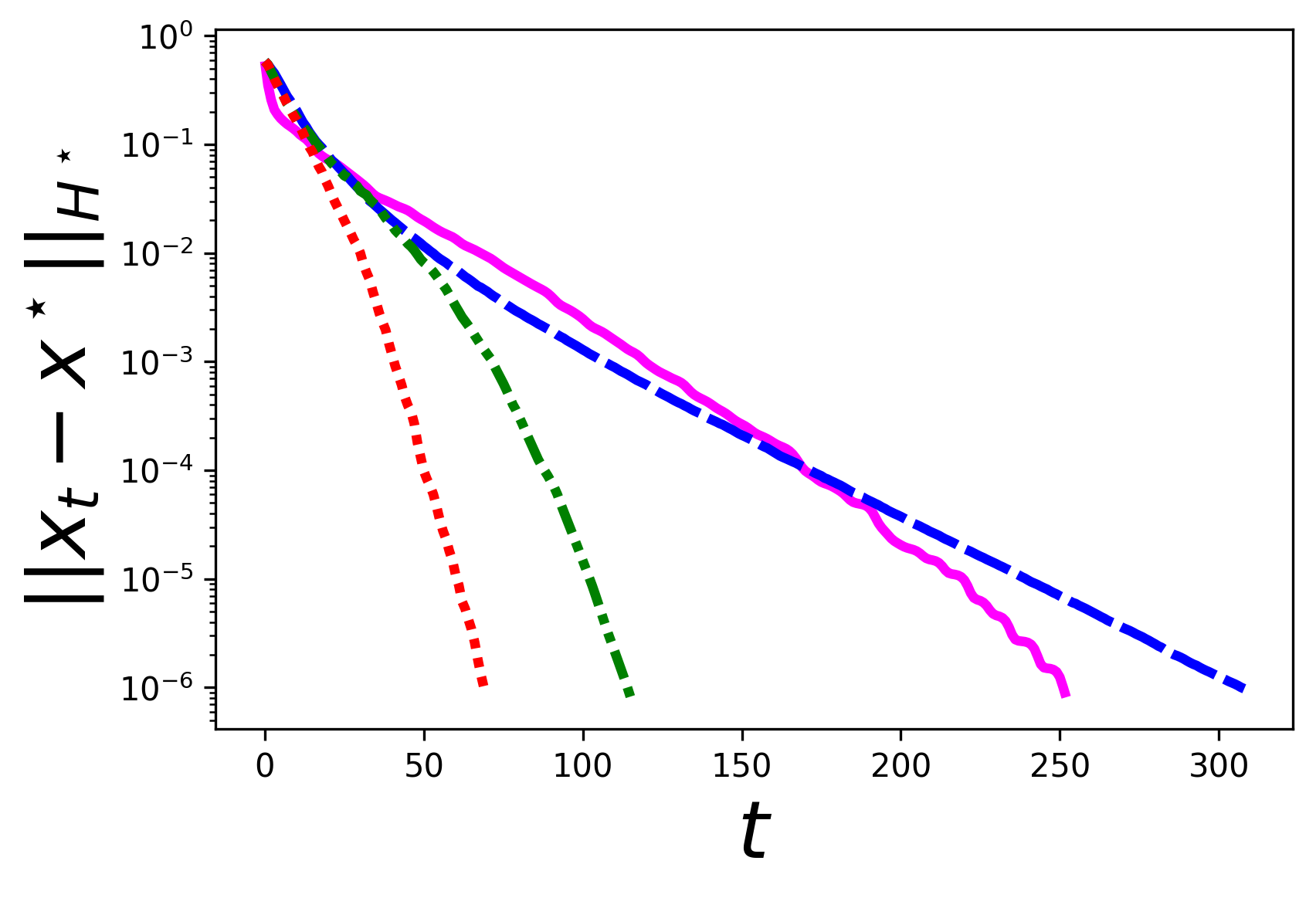}}
	\subfigure[$\tau_{\Ab}\approx10$, $\kappa_{\Ab} = d^{1.5}$, $s=d$]{\label{kappa13}\includegraphics[width=50mm]{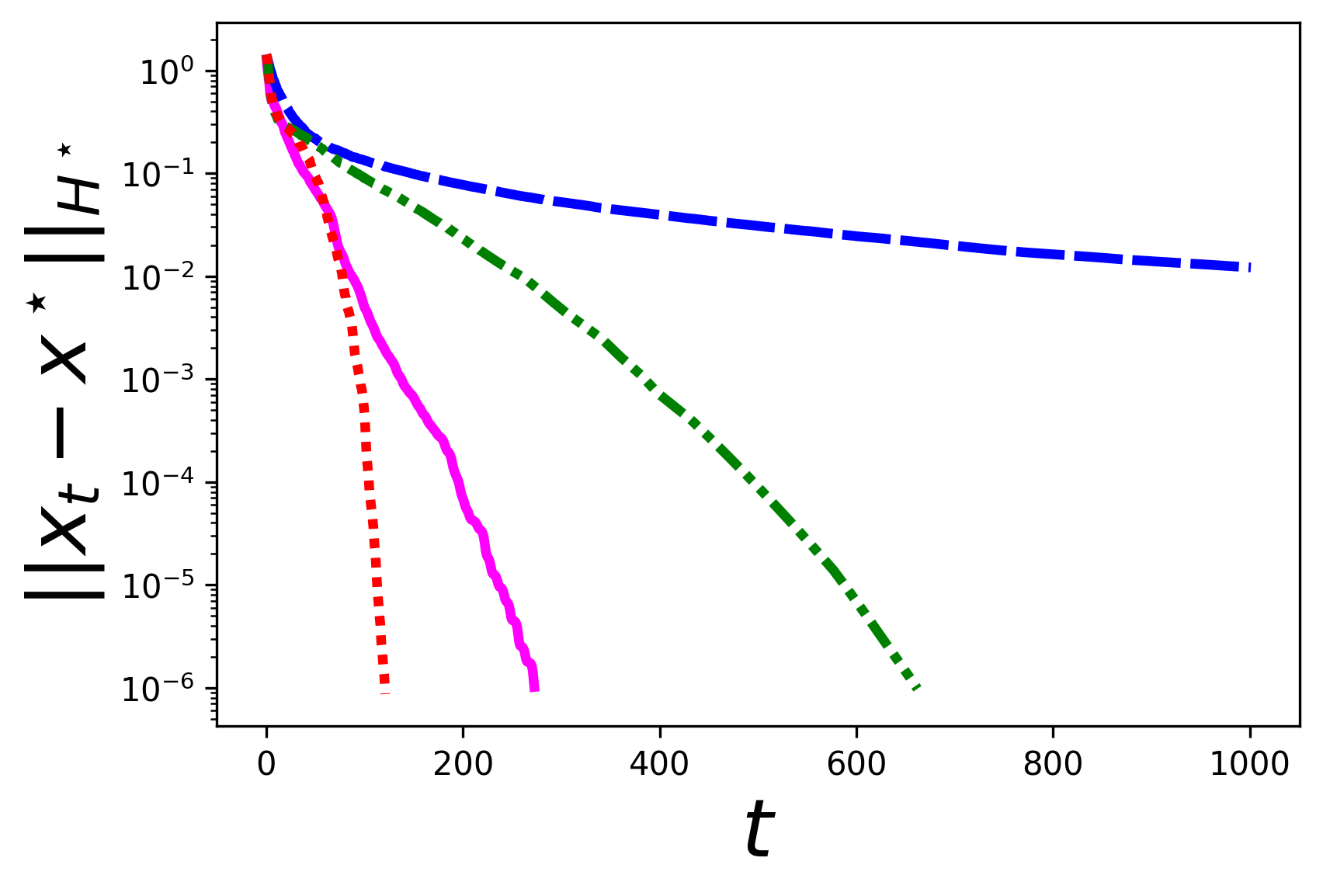}}
	
	\subfigure[$\tau_{\Ab}\approx10$, $\kappa_{\Ab} = \sqrt{d}$, $s=5d$]{\label{kappa14}\includegraphics[width=50mm]{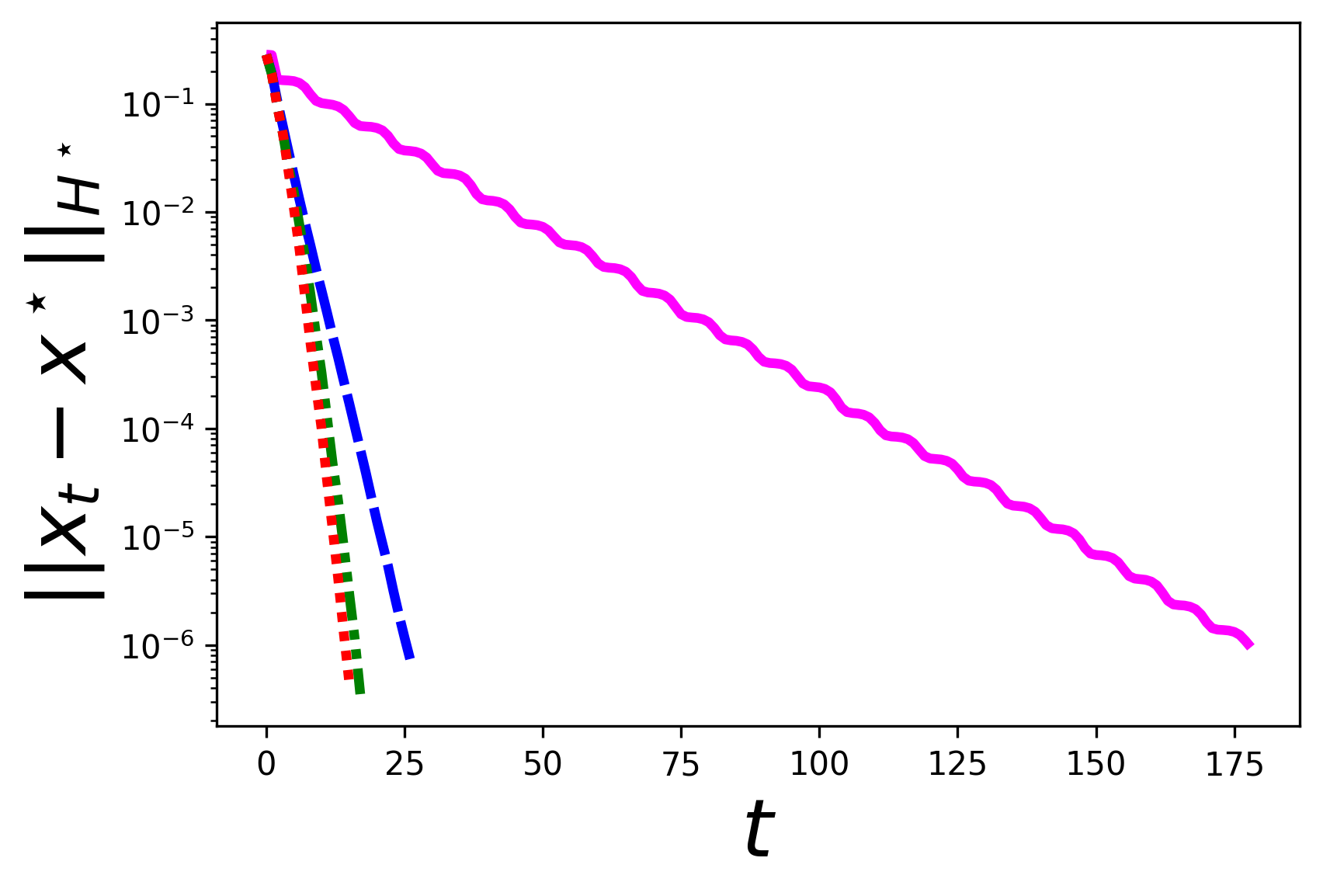}}
	\subfigure[$\tau_{\Ab}\approx10$, $\kappa_{\Ab} = d$, $s=5d$]{\label{kappa15}\includegraphics[width=50mm]{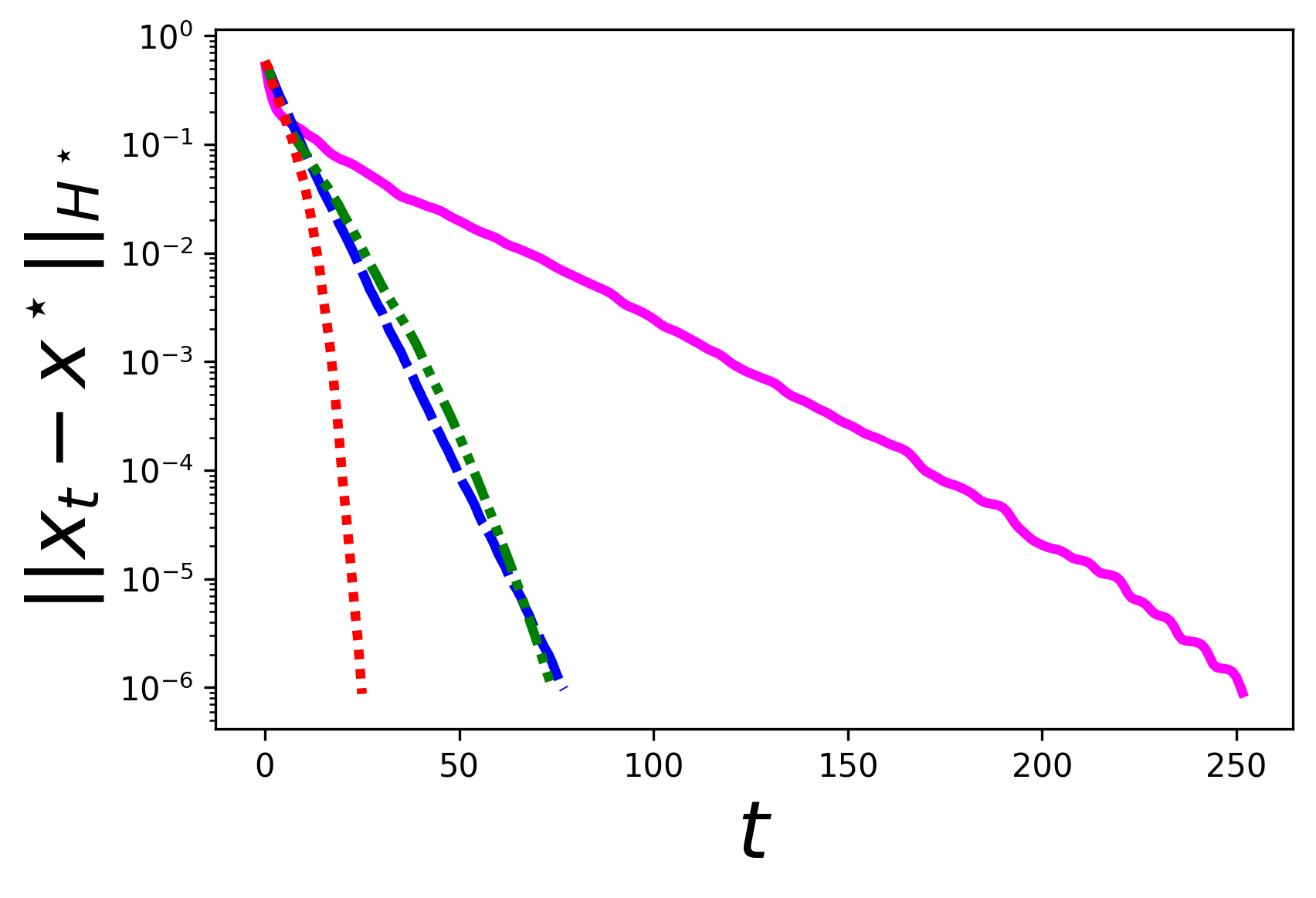}}
	\subfigure[$\tau_{\Ab}\approx10$, $\kappa_{\Ab} = d^{1.5}$, $s=5d$]{\label{kappa16}\includegraphics[width=50mm]{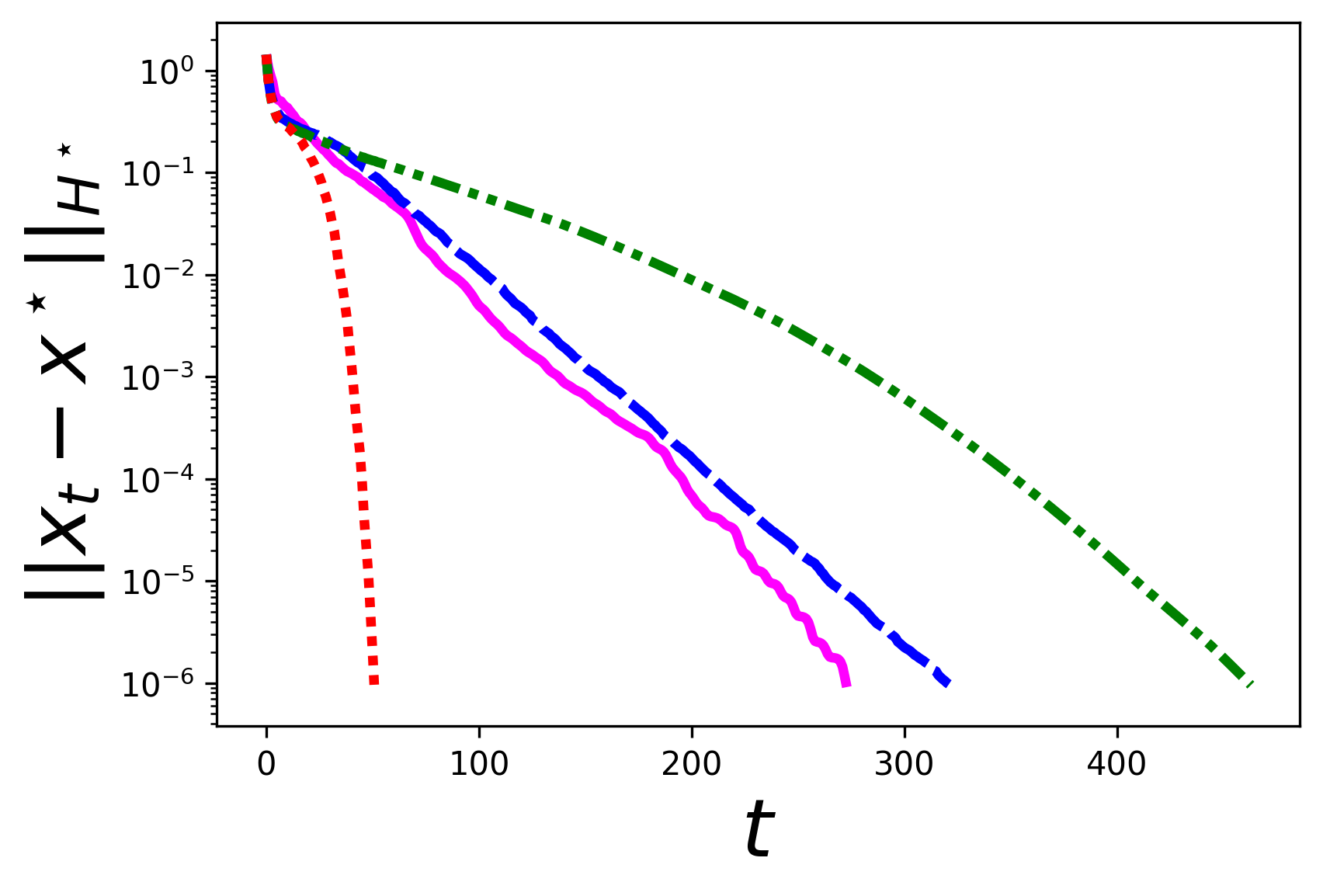}}
	
	\includegraphics[width=0.5\textwidth]{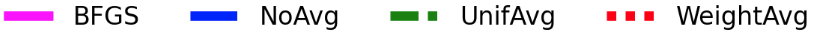}
	\caption{Convergence plots for Subsampled Newton with several Hessian averaging schemes, compared to the standard method without averaging (\noavg) and to BFGS. For all of the plots, we have $(n,d,\nu) = (1000, 100, 10^{-3})$. We vary the data coherence $\tau_{\Ab}$, condition number $\kappa_{\Ab}$ and the subsample size $s$.}\label{fig:cond}
\end{figure}

The overall results are summarized in Table \ref{tab:main}, where we show the number of iterations until convergence for each experiment (we use $\|\xb_{t} - \ttxb\|_{\ttHb}\leq 10^{-6}$ as the criterion). We display the median over 50 independent runs. The first general takeaway is that our proposed \textsf{WeightAvg} performs best overall, and it is the most robust to different problem settings, when varying the coherence $\tau_{\Ab}$, condition number $\kappa_{\Ab}$, sample size $s$, and sketch/sample type. In particular, among the three variants of \noavg/\unifavg/\weightavg, the latter is the only one to beat BFGS in all settings with sample size as small as $s=0.5 d$, as well as the only one to successfully converge within the iteration budget ($999$ iterations) for all Hessian oracles. Nevertheless, there are a few problem instances where either \noavg\ or \unifavg\ perform somewhat better than \weightavg, which can also be justified by our theory. In the cases where \noavg\ performs best, the oracle noise is particularly small (due to the use of a dense Gaussian sketch and/or a large sketch size), which means that averaging out the noise is not helpful until long after the method has converged very close to the optimum. On the other hand, the cases where \unifavg\ performs best are characterized by a well-conditioned objective function, where the superlinear phase of the optimization is reached almost instantly, so the slightly weaker superlinear rate of \weightavg\ compared to \unifavg\ manifests itself before reaching convergence (i.e., the additional factor of $\sqrt{\log(t)}$ in Theorem \ref{thm:main-2}).

To investigate the performance of Hessian averaging more closely, we present selected convergence plots in Figure \ref{fig:cond}. The figure shows decay of the error in the log-scale, so that a linear slope indicates linear convergence whereas a concave slope implies superlinear rate. Here, we used subsampling as the Hessian oracle, varying the coherence $\tau_{\Ab}$, the condition number $\kappa_{\Ab}$, and sample size $s$, and compared the Hessian averaging schemes \unifavg\ and \weightavg\ against the baselines of standard Subsampled Newton (i.e., \noavg) and BFGS.~We make the following observations:
 \begin{enumerate}[label=(\alph*),topsep=0pt]
\setlength\itemsep{0.0em}
\item Subsampled Newton with Hessian averaging (\unifavg\ and  \weightavg) exhibits a clear superlinear rate, observable in how its error plot curves~away from the linear convergence of \noavg. We note that BFGS also exhibits a superlinear rate, but only much later in the convergence process.

\item The gain from Hessian averaging (relative to \noavg) is more significant both for highly ill-conditioned problems (large condition number $\kappa$) and for noisy Hessian oracles (small sample size $s$). For example, in the setting of $(\kappa,s)=(d^{1.5},d)$ for both low and high coherence, standard Subsampled Newton (\noavg) converges orders of magnitude slower than BFGS, and~yet after introducing weighted averaging (\weightavg), it beats BFGS by a~factor of at least~2.5.

\item For small condition number, the two Hessian averaging schemes (\unifavg\ and \weightavg) perform similarly, although in the setting of $(\kappa,s)=(\sqrt d,d)$, the superlinear rate of \unifavg\ is slightly better than that of \weightavg\ (cf. Figure \ref{kappa1}), which aligns with a slightly better rate in Theorem \ref{thm:main-1} compared to Theorem \ref{thm:main-2}.

\item For highly ill-conditioned problems, \weightavg\ converges much faster than \unifavg. This is because, as suggested by Theorem~\ref{thm:main-1}, it takes much longer for \unifavg\ to transition to its fast rate. In fact, in the settings of $(\tau_{\Ab},\kappa_{\Ab},s)=(1, d^{1.5},5d)$ and $(\tau_{\Ab},\kappa_{\Ab},s)=(10,d,5d)$, \unifavg\ initially trails behind \noavg, which is a consequence of the distortion of the averaged Hessian by the noisy estimates from the early global convergence.
\end{enumerate}

\begin{figure}[!th]
\centering     %%% not \center
\subfigure[Gaussian]{\label{GG1}\includegraphics[width=75mm]{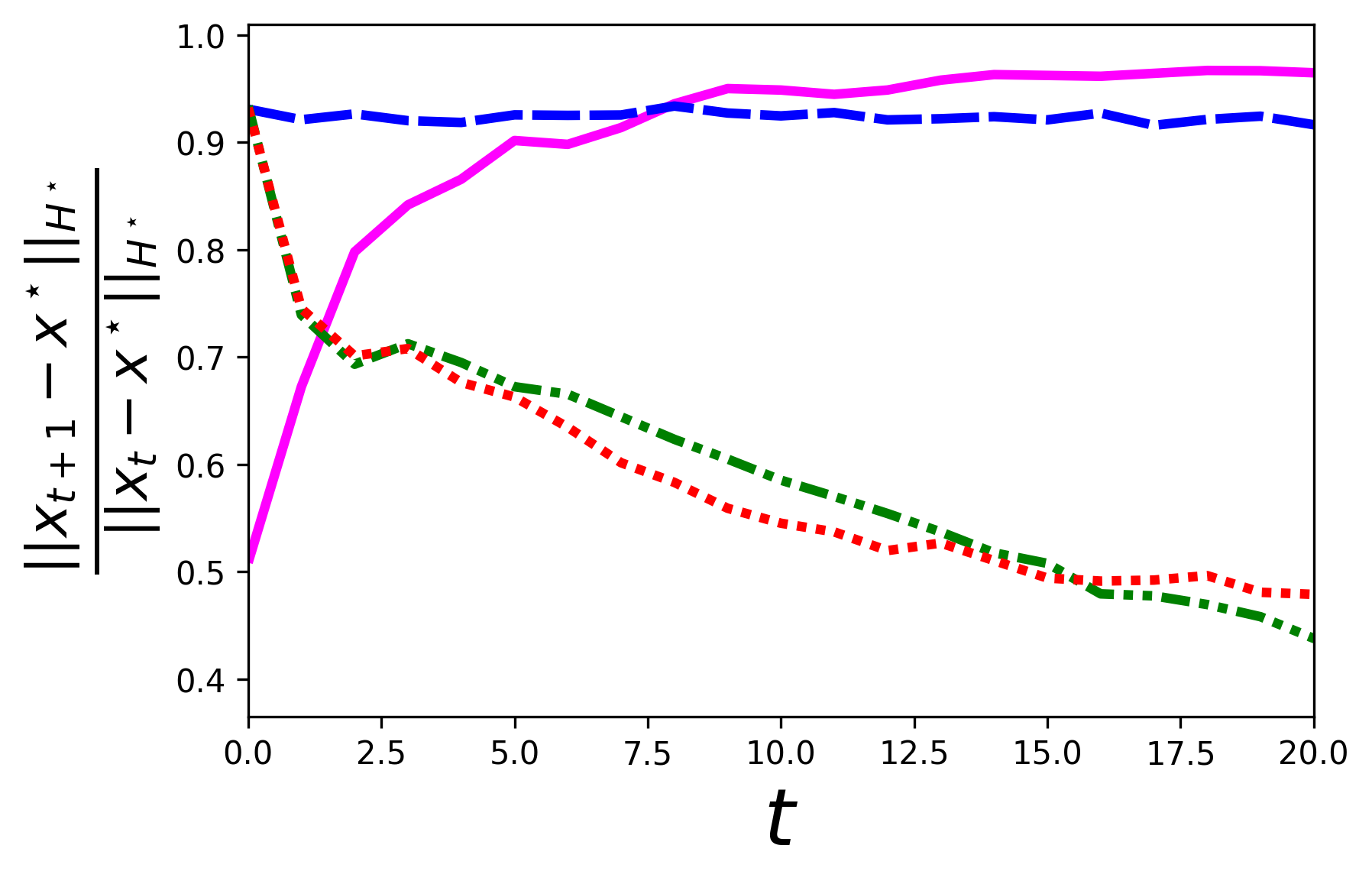}}
\subfigure[CountSketch]{\label{CC1}\includegraphics[width=75mm]{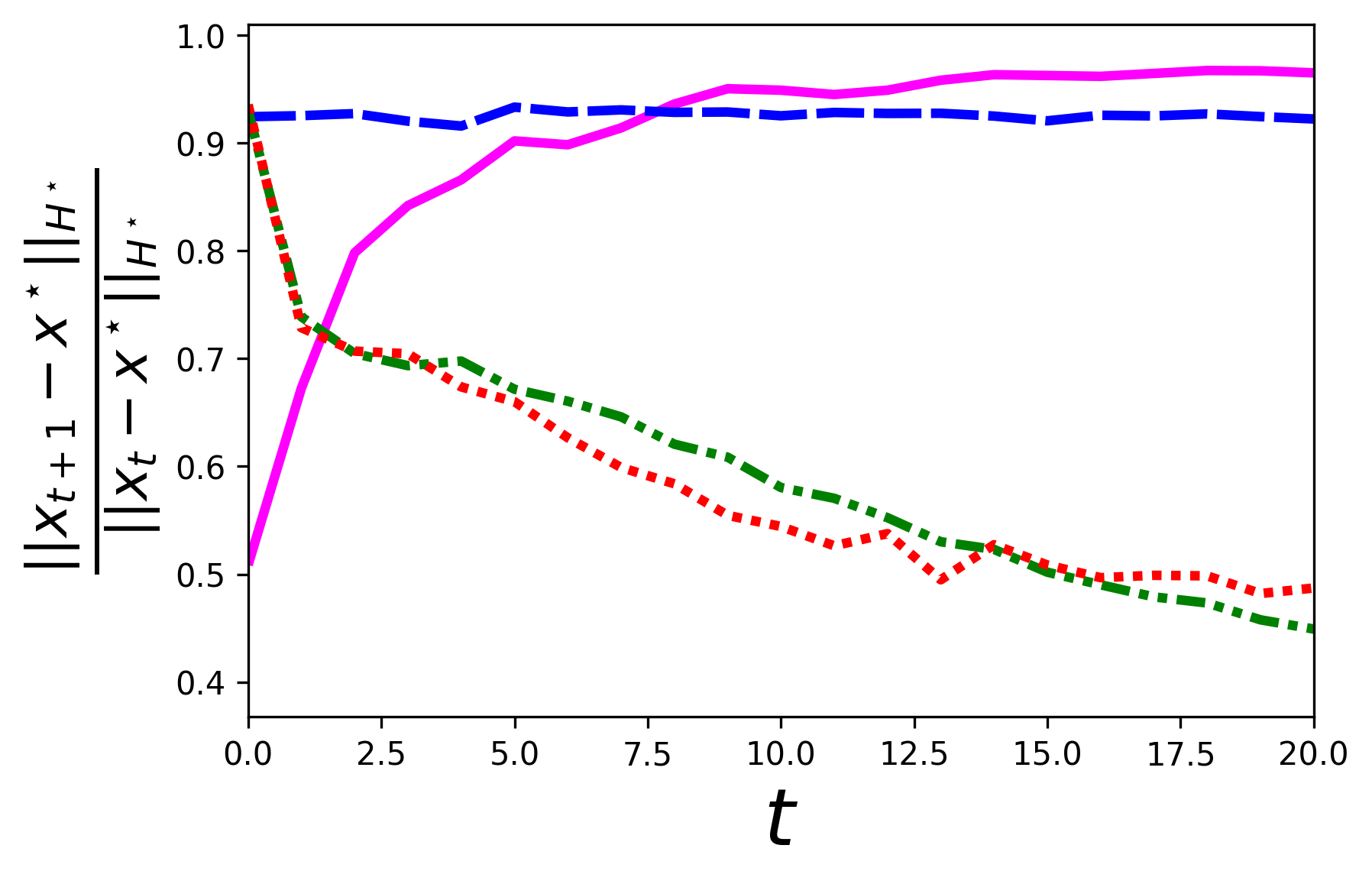}}
\subfigure[LESS-uniform]{\label{LL1}\includegraphics[width=75mm]{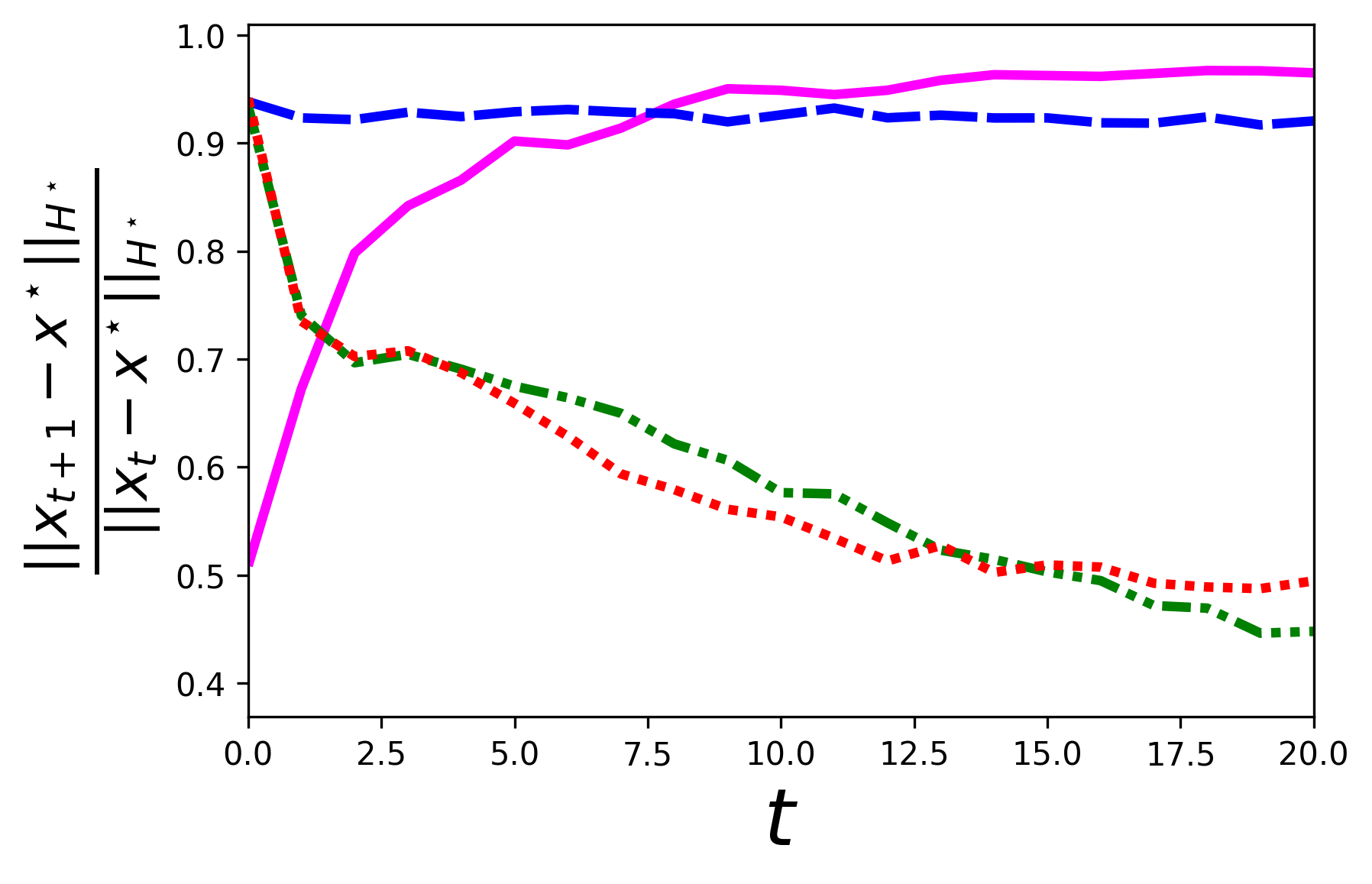}}
\subfigure[Subsampled]{\label{SS1}\includegraphics[width=75mm]{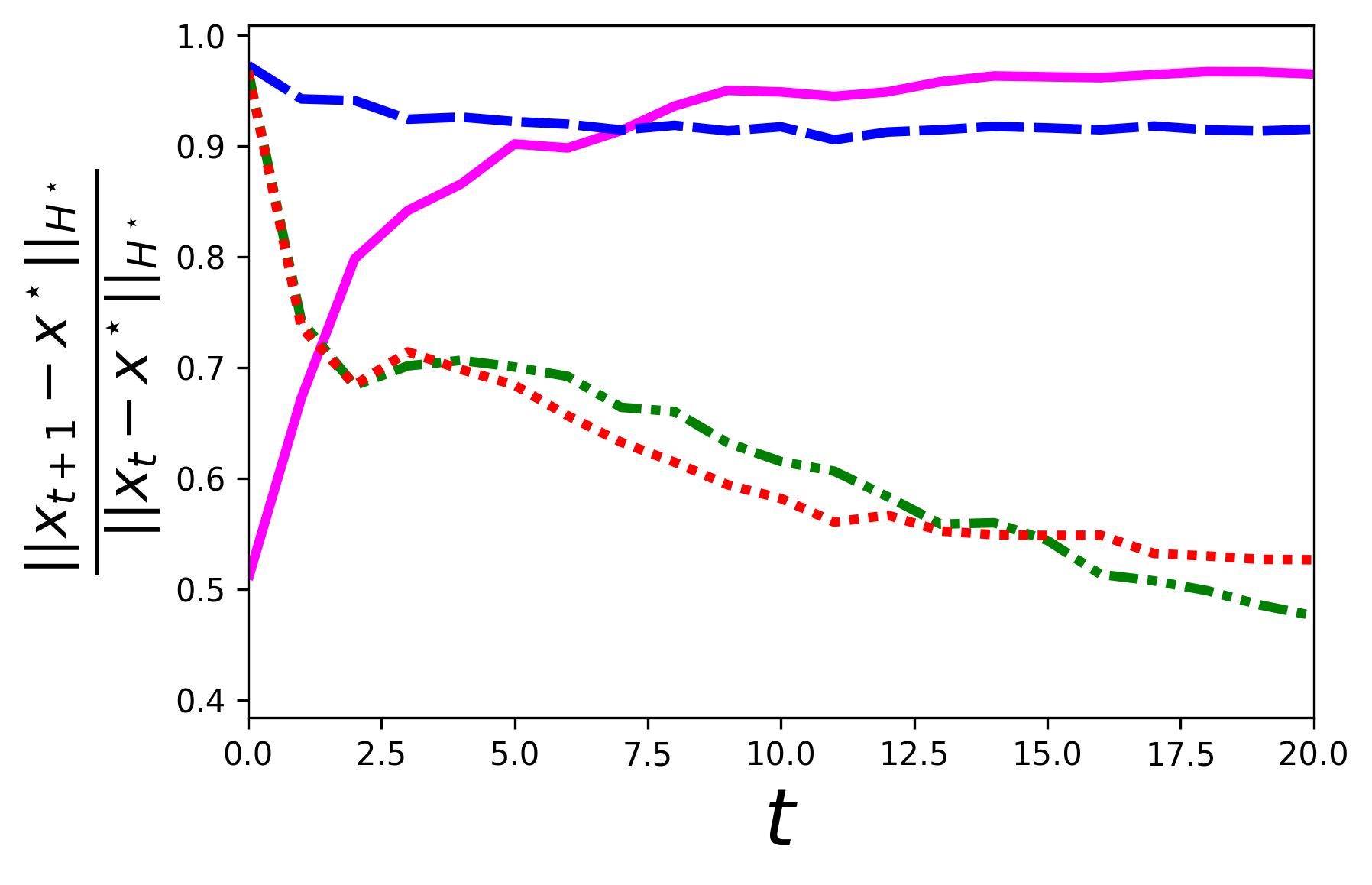}}
\includegraphics[width=0.5\textwidth]{Figure/legend}
\caption{Convergence rates for different Hessian oracles and
averaging schemes, with $(\tau_{\Ab},\kappa_{\Ab},s) = (1,d,d)$. We truncate iterations at 20 to highlight the superlinear rate of \unifavg\ and \weightavg.}\label{fig:rates}
\end{figure}

\vskip5pt
We next investigate the convergence rate of different methods, and aim to show that Hessian averaging leads to $Q$-superlinear convergence, while subsampled/sketched Newton (with fixed sample/sketch size) only exhibit $Q$-linear convergence. Figure \ref{fig:rates} plots $\|\xb_{t+1}-\ttxb\|_{\ttHb}/\|\xb_t-\ttxb\|_{\ttHb}$ versus $t$. From the figure, we indeed observe that our method with different weight sequences (\unifavg\ and \weightavg) always exhibits a $Q$-superlinear convergence, and~the superlinear rate exhibits the trend of $(1/t)^{t/2}$ matching the theory up to~logarithmic factors. We also observe that subsampled/sketched Newton methods with different sketch matrices (\noavg) always exhibit a $Q$-linear convergence. These observations are all consistent with our theory. 

\section{Conclusions}\label{sec:6}

This paper investigated the stochastic Newton method with Hessian averaging. In each iteration, we obtain a stochastic Hessian estimate and then average it with all the past Hessian estimates. We proved that the proposed method~exhibits a local $Q$-superlinear convergence rate, although the non-asymptotic rate and the transition points are different for different weight sequences. In particular, we proved that uniform Hessian averaging finally achieves $(\Upsilon\sqrt{\log t/t})^t$ superlinear rate, which is faster than the other weight sequences, but it may take as many as $O(\Upsilon^2 + \kappa^6/\Upsilon^2)$ iterations with high probability to get to this rate. We also observe that using weighted averaging $w_t = (t+1)^{\log(t+1)}$, the averaging scheme transitions from the global phase to the superlinear phase smoothly after $O(\Upsilon^2+\kappa^2)$ iterations with high probability, with only a slightly slower rate of $(\Upsilon\log t/\sqrt{t})^t$.

One of the future works is to apply our Hessian averaging technique on constrained nonlinear optimization problems. \cite{Na2022adaptive, Na2021Inequality} designed various stochastic second-order methods based on sequential quadratic programming (SQP) for solving constrained problems, where the Hessian of the Lagrangian function~was estimated by subsampling. These works established the global convergence for stochastic SQP methods, while the local convergence rate of these methods~remains unknown. On the other hand, based on our analysis, the local rate of Hessian averaging schemes is induced by the central limit theorem,~which~we~can also apply on the noise of the Lagrangian Hessian oracles. Thus, it is possible~to design a stochastic SQP method with Hessian averaging and show a similar local superlinear rate for solving constrained nonlinear optimization problems. We note that a recent work \cite{Na2022Asymptotic} utilized the Hessian averaging to prove a local sublinear rate for stochastic SQP methods, while that result is not as~strong~as the one in this paper. Further, for unconstrained convex optimization problems, generalizing our analysis to enable stochastic gradients and function evaluations as well as inexact Newton directions is also an important future work, which can further improve the applicability of the algorithm. Finally, given any iteration~threshold $\T$, establishing the probability that the first/second/third transition occurs before $\T$ is an interesting open question.

% \section*{Acknowledgments}

\bibliographystyle{my-plainnat}
\bibliography{ref}

\appendix
%\pagebreak
\numberwithin{equation}{section}
\numberwithin{theorem}{section}

\section{Examples of Sub-exponential Noise}
\label{a:examples}

\begin{lemma}
Consider subsampled Newton as in \eqref{equ:SSN} with sample size $s$ and suppose $\|\nabla^2f_i(\xb)\|\leq \lambda_{\max}R$ for some $R>0$ and for all $i$. Then, we have $\Upsilon = O(\kappa R\sqrt{\log(d)/s}$). If we additionally assume all $f_i$ are convex, then we can improve this to $\Upsilon = O(\sqrt{\kappa R\log(d)/s}+\kappa R\log(d)/s\,)$. 
\end{lemma}

\begin{proof}

The argument follows from the matrix Chernoff and Hoeffding inequalities \citep[Theorems 1.1 and 1.3]{Tropp2011User}. Note that the standard version of these concentration inequalities applies to sampling with replacement,~namely
\begin{equation*}
\hHb(\xb) = \frac1s \sum_{j=1}^s\nabla^2f_{I_j}(\xb),
\end{equation*}
where indices $I_1, I_2, ..., I_s$ are sampled from an index set uniformly at random. Alternate matrix concentration~inequalities for sampling without replacement are also known, e.g., see \cite{Gross2010Note, Wang2019Stochastic}. We thus take sampling with replacement as an example.

We start with the case where $f_i$ may not be convex. By the matrix Hoeffding inequality~\citep[Theorem~1.3, ][]{Tropp2011User}, we have for some small constants $c, c',c''>0$ and any $\eta>0$:
\begin{align*}
\mP\big(\|\hHb(\xb)-\Hb(\xb)\|\geq \eta\big)
&\leq 2d\exp\bigg(-\frac{c\eta^2s}{\lambda_{\max}^2R^2}\bigg) =\exp\bigg(-\log(2d)\Big(\frac{c\eta^2s}{\lambda_{\max}^2R^2\log(2d)}-1\Big)\bigg) \\
&\leq 2\exp\bigg(-\frac{c'\eta^2s}{\lambda_{\max}^2R^2\log(2d)}\bigg) \leq 2\exp\bigg(-\frac{c''\eta}{\lambda_{\max}R\sqrt{\log (d)/s}}\bigg),
\end{align*}
where the last step follows because if the expression in the exponent is larger than $-\log(2)$, then the bound is vacuous, since the probability must lie in $[0,1]$. Thus, it follows that $\|\hHb(\xb)-\Hb(\xb)\|=\|\Eb(\xb)\|$ is a sub-exponential random variable with parameter $\Upsilon_E=O(\lambda_{\max}R\sqrt{\log(d)/s} )$ \cite[see Section 2.7 of][]{Vershynin2018High}. The claim now easily follows.

Next, suppose that $f_i$ are convex. This means that all $\nabla^2f_i(\xb)$ are positive semidefinite, and so is $\hHb(\xb)$. Thus we can use the matrix Chernoff inequality \citep[Theorem~1.1][]{Tropp2011User}, which provides a sharper guarantee:
\begin{align*}
\mP\big(\|\hHb(\xb)-\Hb(\xb)\|\geq \eta\big)
&\leq 2d\exp\bigg(-c\min\Big\{\frac{\eta^2s}{\lambda_{\min}\lambda_{\max}R},\frac{\eta s}{\lambda_{\max}R}\Big\}\bigg) \\
  &\leq
    2\exp\bigg(-c'\min\Big\{\frac{\eta^2s}{\lambda_{\min}\lambda_{\max}R\log(d)},\frac{\eta
    s}{\lambda_{\max}R\log(d)}\Big\}\bigg)
  \\
  &\leq
    2\exp\bigg(-c''\min\Big\{\frac{\eta}{\sqrt{\lambda_{\min}\lambda_{\max}R\log(d)/s}},\frac{\eta
    }{\lambda_{\max}R\log(d)/s}\Big\}\bigg)    
    \\
  &\leq 2\exp\bigg(-\frac{c''\eta}{\sqrt{\lambda_{\min}\lambda_{\max}R\log (d)/s} + \lambda_{\max}R\log(d)/s}\bigg). 
\end{align*}
It follows that $\|\Eb(\xb)\|$ is a sub-exponential random variable with parameter $\Upsilon_E=O( \lambda_{\max}R\log(d)/s \\+ \sqrt{\lambda_{\min}\lambda_{\max}R\log(d)/s}) = O(\lambda_{\min}(\sqrt{\kappa R\log(d)/s}+\kappa R\log(d)/s))$, and we get the claim. 
\end{proof}

\begin{lemma}
Consider Newton sketch as in \eqref{equ:SKN} with $\Sb\in \mR^{s\times n}$ consisting of i.i.d. $\mathcal N(0,1/s)$ entries. Then, we have $\Upsilon = O(\kappa(\sqrt{d/s}+ d/s))$. 
\end{lemma}

\begin{proof}
In fact, the result holds as long as $\sqrt s\,\Sb$ has independent mean zero unit variance sub-Gaussian entries. From a standard result on the concentration of sub-Gaussian matrices \citep[Theorem~4.6.1 in][]{Vershynin2018High}, there exists a constant $c\geq 1$ such that, with probability $1 - 2\exp(-t^2)$,
\begin{equation}\label{nequ:4}
\|\Hb(\xb)^{-1/2}\hHb(\xb) \Hb(\xb)^{-1/2}-\Ib\|\leq \delta\vee \delta^2 \quad\quad \text{for any }\quad \delta \geq c\rbr{\sqrt{d/s} + t/\sqrt{s}}.
\end{equation}
Thus, for any $\eta>0$ such that $\eta \wedge\sqrt{\eta} \geq c\sqrt{d/s}\Longleftarrow \eta\geq c^2(d/s + \sqrt{d/s})$, we let $t = \sqrt{s}(\eta\wedge\sqrt{\eta})/c$ and have
\begin{equation*}
c\rbr{\sqrt{d/s} + t/\sqrt{s}} = c\sqrt{d/s} + \eta \wedge\sqrt{\eta}\leq 2(\eta\wedge\sqrt{\eta}).
\end{equation*}
Plugging $\delta = 2(\eta\wedge\sqrt{\eta})$ into \eqref{nequ:4}, we obtain for a small constant $c'>0$ that
\begin{equation*}
P\rbr{\|\Hb(\xb)^{-1/2}\hHb(\xb) \Hb(\xb)^{-1/2}-\Ib\|\geq 4\eta}\leq 2\exp(-s\cdot (\eta\wedge \eta^2)/c^2) \leq 2\exp\Big(-\frac{c'\eta}{\sqrt{1/s}+1/s}\Big).
\end{equation*}
This further implies that, for a small constant $c''>0$, 
\begin{equation*}
\mP(\|\Hb(\xb)^{-1/2}\Eb(\xb) \Hb(\xb)^{-1/2}\|\geq
\eta)\leq 2\exp(-\frac{c''\eta}{\sqrt{d/s}+d/s}),\quad \forall \eta>0.
\end{equation*}
Thus, $\|\Hb(\xb)^{-1/2}\Eb(\xb) \Hb(\xb)^{-1/2}\|$ is sub-exponential with constant $K=O(\sqrt{d/s}+d/s)$. The claim follows by noting that $\|\Eb(\xb)\|\leq \lambda_{\max}\cdot \|\Hb(\xb)^{-1/2}\Eb(\xb) \Hb(\xb)^{-1/2}\|$.
\end{proof}

\section{Preliminary Lemmas}

\begin{lemma}\label{lem:pre}
Let $d\geq 1$ and $\delta\in(0,1)$. If $d/\delta\geq e$, then for any $0<a\leq 1$,
\begin{equation*}
t\geq \frac {2\log(d/(\delta a))}{a}\quad\Longrightarrow\quad\frac{\log(dt/\delta)}{t} \leq a,
\end{equation*}
and moreover, the right hand side inequality fails if $t= \log(d/(\delta a))/a$ and $d/\delta >e$.
\end{lemma}

\begin{proof}
Let $t = 2\log(d/(\delta a))/a$. We will show for this $t$ that $t\geq \log(dt/\delta)/a$. We have
\begin{align*}
t - \frac{\log(dt/\delta)}{a}
& =  \frac{2\log(d/(\delta a))}{a} - \frac{\log(d/(\delta a)) + \log(2\log(d/(\delta a)))}{a}
\\
&=\frac{\log(d/(\delta a)) - \log(2\log(d/(\delta a)))}{a}\\
&= \frac1a\log\Big(\frac{x}{2\log(x)}\Big)\geq 0,\qquad\text{for}\quad x=\frac{d}{\delta a},
\end{align*}
where the last step is due to the fact that $x\geq 2\log(x)$ for all $x>0$. Now, we consider the function $f(t) = \log(dt/\delta)/t$. For any $t\geq1$, its derivative is bounded as
\begin{align*}
\frac{\partial f(t)}{\partial t} = \frac{1-\log(dt/\delta)}{t^2}\leq 0,
\end{align*}
which implies that for $t\geq2\log(d/(\delta a))/a \geq 1$ we get $\log(dt/\delta)/t =f(t)\leq f(2\log(d/(\delta a))/a)\leq a$. Finally, if $t=\log(d/(\delta a))/a$ and $d/\delta >e$, then $\log(dt/\delta)/a=\log(d/(\delta a))/a + \log(\log(d/(\delta a)))/a >t$, which completes the proof.
\end{proof}

\begin{lemma}\label{lem:pre:1}
Define the neighborhood $\bar{\N}_{\nu} = \{\xb: f(\xb) \leq f(\ttxb) + \nu\lambda_{\min}^3/L^2\}$ for $\nu>0$. Suppose Assumption \ref{ass:4} holds, then the following relation holds
\begin{equation*}
\bar{\N}_{\nu^2/3} \subseteq \N_{\nu} \subseteq \bar{\N}_{\nu^2}, \quad \text{for any } \nu \in (0, 1]
\end{equation*}
where $\N_{\nu}$ is defined in \eqref{Nnu}.	
\end{lemma}

\begin{proof}
By Assumption \ref{ass:4}, we have
\begin{equation}\label{pequ:A3}
\abr{f(\xb) - f(\ttxb) - \frac{1}{2}\nbr{\xb - \ttxb}_{\ttHb}^2} \leq \frac{L}{6}\nbr{\xb - \ttxb}^3.
\end{equation}
Suppose $\xb \in \bar{\N}_{\nu^2/3}$ for $\nu\in(0, 1]$, we have
\begin{equation}\label{pequ:A4}
f(\ttxb) + \frac{\lambda_{\min}}{2}\|\xb - \ttxb\|^2\leq f(\xb) \leq f(\ttxb) + \frac{\nu^2\lambda_{\min}^3}{3L^2} \Longrightarrow \|\xb - \ttxb\| \leq \frac{\sqrt{2}\nu\lambda_{\min}}{\sqrt{3}L}.
\end{equation}
Combining \eqref{pequ:A3} and \eqref{pequ:A4}, we know that $\xb \in \bar{\N}_{\nu^2/3}$ for $\nu\in(0, 1]$ leads to
\begin{align*}
f(\ttxb) + \frac{\nu^2\lambda_{\min}^3}{3L^2} \geq f(\xb) & \stackrel{\mathclap{\eqref{pequ:A3}}}{\geq} f(\ttxb) + \frac{1}{2}\|\xb - \ttxb\|_{\ttHb}^2 - \frac{L}{6}\|\xb-\ttxb\|^3	\\
& \geq f(\ttxb) + \frac{1}{2}\nbr{\xb - \ttxb}_{\ttHb}^2 - \frac{L}{6}\nbr{\xb - \ttxb}\cdot \frac{\nbr{\xb - \ttxb}_{\ttHb}^2}{\lambda_{\min}}\\
& = f(\ttxb) + \rbr{\frac{1}{2}-\frac{L\|\xb - \ttxb\|}{6\lambda_{\min}}}\|\xb - \ttxb\|_{\ttHb}^2\\
& \stackrel{\mathclap{\eqref{pequ:A4}}}{\geq} f(\ttxb) + \rbr{\frac{1}{2} - \frac{\sqrt{2}\nu}{6\sqrt{3}}}\|\xb - \ttxb\|_{\ttHb}^2 \geq f(\ttxb) + \frac{\|\xb - \ttxb\|_{\ttHb}^2}{3},
\end{align*}
which further implies $\|\xb - \ttxb\|_{\ttHb} \leq \nu\lambda_{\min}^{3/2}/L$ and $\xb\in \N_{\nu}$. On the other hand, suppose $\xb \in \N_{\nu}$, we have
\begin{align*}
f(\xb)  &\stackrel{\mathclap{\eqref{pequ:A3}}}{\leq} f(\ttxb) + \frac{1}{2}\nbr{\xb - \ttxb}_{\ttHb}^2 + \frac{L}{6}\nbr{\xb - \ttxb}^3 \leq f(\ttxb) + \rbr{ \frac{1}{2} + \frac{L\nbr{\xb - \ttxb}_{\ttHb}}{6\lambda_{\min}^{3/2}}} \nbr{\xb - \ttxb}_{\ttHb}^2\\
& \leq f(\ttxb) + \rbr{ \frac{1}{2} + \frac{\nu}{6}} \nbr{\xb - \ttxb}_{\ttHb}^2\leq f(\ttxb)  + \nbr{\xb - \ttxb}_{\ttHb}^2 \leq f(\ttxb) + \frac{\nu^2\lambda_{\min}^3}{L^2}.
\end{align*}
Thus, we have $\xb \in \bar{\N}_{\nu^2}$. This completes the proof.
\end{proof}

\begin{lemma}\label{lem:pre:2}
	Suppose Assumption \ref{ass:4} holds, then we have
	\begin{equation*}
		\|\Hb(\xb_1)^{-1/2}\Hb(\xb_2)\Hb(\xb_1)^{-1/2} - \Ib\| \leq \frac{L}{\lambda_{\min}^{3/2}}\|\xb_1 - \xb_2\|_{\ttHb}, \quad \forall \xb_1, \xb_2\in \mR^d.
	\end{equation*}
\end{lemma}

\begin{proof}
We note that
\begin{equation*}
\|\Hb(\xb_1)^{-1/2}\Hb(\xb_2)\Hb(\xb_1)^{-1/2} - \Ib\|  \leq \frac{1}{\lambda_{\min}}\|\Hb(\xb_1) - \Hb(\xb_2)\| \leq \frac{L}{\lambda_{\min}}\|\xb_1 - \xb_2\| \leq \frac{L}{\lambda_{\min}^{3/2}}\|\xb_1 - \xb_2\|_{\ttHb}.
\end{equation*}
This completes the proof.
\end{proof}

\section{Proofs}

\subsection{Proof of Theorem \ref{thm:2}}\label{pf:thm:2}

For any $t\geq 0$ and scalar $\theta>0$, we have from Assumption \ref{ass:3} that
\begin{align*}
\mE_i\sbr{\exp\rbr{\theta z_{i,t}\Eb_i}}  & = \Ib + \theta z_{i,t}\mE_i[\Eb_i] + \sum_{j=2}^{\infty} \frac{(\theta z_{i,t})^j\mE_i[\Eb_i^j]}{j!}
\preceq \Ib + \sum_{j=2}^{\infty}(\theta z_{i,t}\Upsilon_E)^{j-2}\cdot \frac{(\theta z_{i,t}\Upsilon_E)^2}{2}\cdot\Ib \\
& = \Ib + \frac{(\theta z_{i,t}\Upsilon_E)^2}{2(1 - \theta z_{i,t} \Upsilon_E)}\cdot \Ib \preceq\exp\rbr{ \frac{(\theta z_{i,t}\Upsilon_E)^2}{2(1 - \theta z_{i,t}\Upsilon_E)} \cdot \Ib}, \quad \forall i = 0,1,\ldots,t,
\end{align*}
where the second equality holds for $\cbr{\theta: \theta z_{i,t}\Upsilon_E < 1}$. Therefore, we let $\Theta_t \coloneqq\{\theta: \theta < 1/(z_t^{(\max)}\Upsilon_E )\}$, and have for all $i = 0,1,\ldots, t$,
\begin{equation*}
\mE_i\sbr{\exp\rbr{\theta z_{i,t}\Eb_i}}  \preceq \exp\rbr{\frac{(\theta z_{i,t}\Upsilon_E)^2}{2(1 - \theta z_t^{(\max)} \Upsilon_E) } \cdot \Ib}, \quad \forall\theta\in \Theta_t.
\end{equation*}
The above inequality suggests that the condition \eqref{equ:exp:cond} in Lemma \ref{lem:2} holds with
\begin{equation*}
g_t(\theta) = \frac{\theta^2}{2(1 - \theta z_t^{(\max)} \Upsilon_E) }\quad \text{ and }\quad \Ub_i = z_{i,t}^2\Upsilon_E^2\cdot \Ib.
\end{equation*}
Let $\sigma^2 = \|\sum_{i=0}^{t}\Ub_i\| = \Upsilon_E^2\sum_{i=0}^{t}z_{i,t}^2$ in Lemma \ref{lem:2}. Then we have for any $\eta\geq 0$,
\begin{equation*}
\mP\rbr{\nbr{\barEb_t} \geq \eta} \leq 2d\cdot \inf_{\theta \in \Theta_t}\exp\rbr{-\theta\eta + g_t(\theta)\sigma^2}.
\end{equation*}
Plugging $\theta = \eta/(z_t^{(\max)} \Upsilon_E \eta + \sigma^2) \in \Theta_t$ into the above right hand side, we obtain
\begin{equation*}
\mP\rbr{\|\barEb_t\| \geq \eta} \leq 2d\cdot\exp\rbr{-\frac{\eta^2/2}{ \sigma^2 + z_t^{(\max)} \Upsilon_E \eta}  }.
\end{equation*}
This completes the proof.

\subsection{Proof of Lemma \ref{lem:3}}\label{pf:lem:3}

Let us define the event
\begin{equation*}
\bar{\EE} = \bigcap_{t=0}^{\infty}\cbr{\nbr{\barEb_t} \leq 8\Upsilon_E\rbr{\sqrt{\frac{\log(d(t+1)/\delta)}{t+1}}\vee \frac{\log(d(t+1)/\delta)}{t+1} } }.
\end{equation*}
By Theorem \ref{thm:2} and \eqref{equ:equal:con}, we know $\mP(\bar{\EE}) \geq 1-\delta\pi^2/6$. We also note that
\begin{align}\label{nequ:1}
\|\barEb_t\|\leq \epsilon\lambda_{\min} & \Longleftarrow 8\Upsilon_E\rbr{\sqrt{\frac{\log(d(t+1)/\delta)}{t+1}}\vee \frac{\log(d(t+1)/\delta)}{t+1} } \leq \epsilon\lambda_{\min} \nonumber\\
& \Longleftarrow \frac{\log(d(t+1)/\delta)}{t+1} \leq \frac{\epsilon^2}{64\Upsilon^2} \wedge1 = \rbr{\frac{\epsilon}{8\Upsilon}\wedge 1}^2 \nonumber\\
&\Longleftarrow t \geq 4\rbr{\frac{8\Upsilon}{\epsilon}\vee 1}^2 \log\cbr{\frac{d}{\delta}\cdot\rbr{\frac{8\Upsilon}{\epsilon}\vee 1}}=\T_1. \quad (\text{Lemma \ref{lem:pre}})
\end{align}
Therefore, with probability at least $1-\delta\pi^2/6$, we have for any $t\geq \T_1$ that $\|\barEb_t\|\leq \epsilon\lambda_{\min}$. This shows that the event $\EE$ defined in \eqref{event:EE} happens, and
\begin{equation*}
(1-\epsilon)\lambda_{\min}\cdot\Ib\preceq\tHb_t \stackrel{\eqref{equ:simple}}{=} \frac{1}{t+1}\sum_{i=0}^{t}\Hb_i + \barEb_t \preceq (\lambda_{\max} + \epsilon\lambda_{\min})\cdot\Ib\preceq (1+\epsilon)\lambda_{\max} \cdot\Ib.
\end{equation*}
This completes the proof.

\subsection{Proof of Lemma \ref{lem:4}}\label{pf:lem:4}

By Lemma \ref{lem:3}, we know that $\pb_t = -(\tHb_t)^{-1}\nabla f_t$ for $t\geq \T_1$. We apply Taylor's expansion:
\begin{align*}
	f(\xb_t - \mu_t(\tHb_t)^{-1}\nabla f_t ) & \leq f(\xb_t) - \mu_t \nabla f_t^\top(\tHb_t)^{-1}\nabla f_t + \frac{\lambda_{\max}\mu_t^2}{2}\nabla f_t^\top(\tHb_t)^{-2}\nabla f_t \\
	& \leq f(\xb_t) - \mu_t \nabla f_t^\top(\tHb_t)^{-1}\nabla f_t +\frac{\kappa\mu_t^2}{2(1-\epsilon)}\nabla f_t^\top(\tHb_t)^{-1}\nabla f_t \quad (\text{Lemma \ref{lem:3}})\\
	& =  f(\xb_t) - \mu_t\rbr{1 - \frac{\kappa\mu_t}{2(1-\epsilon)}}\nabla f_t^\top(\tHb_t)^{-1}\nabla f_t.
\end{align*}
Thus, the Armijo condition is satisfied if
\begin{equation*}
1 - \frac{\kappa\mu_t}{2(1-\epsilon)} \geq \beta \Longleftrightarrow \mu_t \leq \frac{2(1-\beta)(1-\epsilon)}{\kappa}.
\end{equation*}
Therefore, the backtracking line search on Line 7 leads to a stepsize satisfying
\begin{equation}\label{npequ:1}
\mu_t\geq \frac{2\rho(1-\beta)(1-\epsilon)}{\kappa}.
\end{equation}
Moreover, we apply Lemma \ref{lem:3} and strong convexity of $f$, and have
\begin{equation}\label{npequ:2}
\nabla f_t^\top(\tHb_t)^{-1}\nabla f_t \geq \frac{1}{(1+\epsilon)\lambda_{\max}}\|\nabla f_t\|^2 \geq \frac{2}{(1+\epsilon)\kappa}(f(\xb_t) - f(\ttxb)).
\end{equation}
Combining \eqref{npequ:1}, \eqref{npequ:2} with the Armijo condition, we have
\begin{align}\label{pequ:B1}
f(\xb_{t+1}) - f(\ttxb) &\leq f(\xb_t) - f(\ttxb) - \mu_t \beta \nabla f_t^\top(\tHb_t)^{-1}\nabla f_t \nonumber\\
& \leq f(\xb_t) - f(\ttxb) - \frac{2\rho\beta(1-\beta)(1-\epsilon)}{\kappa}\cdot \frac{2}{(1+\epsilon)\kappa}(f(\xb_t) - f(\ttxb)) \nonumber\\
& = (1-\phi)(f(\xb_t) - f(\ttxb)),
\end{align}
which shows the first part of the statement. Furthermore, we apply \eqref{pequ:B1} recursively, apply strong convexity, and have
\begin{equation*}
\|\xb_t - \ttxb\|^2\leq \frac{2(f(\xb_t)-f(\ttxb))}{\lambda_{\min}} \\
\stackrel{\eqref{pequ:B1}}{\leq} \frac{2(1-\phi)^{t-\T_1}}{\lambda_{\min}}(f(\xb_{\T_1}) - f(\ttxb)) \leq \frac{2(1-\phi)^{t-\T_1}}{\lambda_{\min}}(f(\xb_0) - f(\ttxb)).
\end{equation*}
The last statement follows from $\|\xb_t-\ttxb\|_{\ttHb}^2\leq \lambda_{\max}\|\xb_t - \ttxb\|^2$. This completes the proof.

\subsection{Proof of Corollary \ref{cor:1}}\label{pf:cor:1}

We condition on the event \eqref{event:EE}, which happens with probability $1-\delta\pi^2/6$. By Lemma \ref{lem:4}, we know that, for $t\geq \T_1$,
\begin{equation*}
f(\xb_t) - f(\ttxb) \leq (1 - \phi)^{t-\T_1}(f(\xb_{\T_1}) - f(\ttxb)) \leq (1 - \phi)^{t-\T_1}(f(\xb_0) - f(\ttxb)).
\end{equation*}
By Lemma \ref{lem:pre:1}, to have $\xb_t \in \N_\nu$, it suffices to have $\xb_t \in \bar{\N}_{\nu^2/3}$. Thus, we let
\begin{equation}\label{equ:new}
(1 - \phi)^{t-\T_1}(f(\xb_0) - f(\ttxb)) \leq \frac{\nu^2\lambda_{\min}^3}{3L^2} \Longleftarrow t-\T_1 \geq \frac{\log\rbr{\frac{3L^2(f(\xb_0) - f(\ttxb))}{\nu^2\lambda_{\min}^3}}}{\log\rbr{\frac{1}{1-\phi}}}\Longleftarrow t \geq \T_1 + \T_2,
\end{equation}
where $\T_2$ is defined in \eqref{T2} and the implication uses the fact that $\log(1/(1-\phi))\geq \phi$. Furthermore, since $f(\xb_t)$ is always decreasing, we know after $\T_1+T_2$ iterations $\xb_t$ stays in $\bar{\N}_{\nu^2/3}$, and hence stays in $\N_{\nu}$ (cf. Lemma \ref{lem:pre:1}). This completes the proof.

\subsection{Proof of Lemma \ref{lem:5}}\label{pf:lem:5}

It suffices to show 
\begin{equation*}
f(\xb_t - (\tHb_t)^{-1}\nabla f_t) \leq f(\xb_t) - \beta\nabla f_t^\top(\tHb_t)^{-1}\nabla f_t.
\end{equation*}
By Assumption \ref{ass:4}, we have
\begin{align}\label{npequ:3}
& f(\xb_t - (\tHb_t)^{-1}\nabla f_t) \nonumber\\
& \leq f(\xb_t) - \nabla f_t^\top(\tHb_t)^{-1}\nabla f_t + \frac{1}{2}\nabla f_t^\top (\tHb_t)^{-1}\Hb_t(\tHb_t)^{-1}\nabla f_t + \frac{L}{6}\|(\tHb_t)^{-1}\nabla f_t\|^3 \nonumber\\
& \leq f(\xb_t) - \frac{1}{2}\nabla f_t^\top(\tHb_t)^{-1}\nabla f_t + \frac{1}{2}\nabla f_t^\top(\tHb_t)^{-1}(\Hb_t - \tHb_t)(\tHb_t)^{-1}\nabla f_t + \frac{L}{6}\|(\tHb_t)^{-1}\nabla f_t\|^3 \nonumber\\
& \leq f(\xb_t) - \frac{1}{2}\nabla f_t^\top(\tHb_t)^{-1}\nabla f_t + \frac{1}{2}\|(\tHb_t)^{-1/2}(\Hb_t - \tHb_t)(\tHb_t)^{-1/2}\|\nabla f_t^\top(\tHb_t)^{-1}\nabla f_t \nonumber\\
& \quad + \frac{L}{6}\|(\tHb_t)^{-1}\nabla f_t\| \nabla f_t^\top(\tHb_t)^{-2}\nabla f_t \nonumber\\
& \leq f(\xb_t) - \frac{1}{2}\nabla f_t^\top (\tHb_t)^{-1}\nabla f_t + \frac{\psi}{2(1-\psi)}\nabla f_t^\top (\tHb_t)^{-1}\nabla f_t + \frac{L\|(\tHb_t)^{-1}\nabla f_t\|}{6(1-\psi)\lambda_{\min}}\nabla f_t^\top(\tHb_t)^{-1}\nabla f_t \nonumber\\
& = f(\xb_t) - \rbr{\frac{1}{2} - \frac{\psi}{2(1-\psi)} - \frac{L\|(\tHb_t)^{-1}\nabla f_t\|}{6(1-\psi)\lambda_{\min}}}\nabla f_t^\top (\tHb_t)^{-1}\nabla f_t.
\end{align}
For term $\|(\tHb_t)^{-1}\nabla f_t\|$, we let $\Hb_t^\eta = \Hb(\xb_t^\eta)$ with $\xb_t^\eta = \ttxb + \eta(\xb_t - \ttxb)$ for some $\eta\in(0, 1)$, and have
\begin{align}\label{npequ:4}
&\|(\tHb_t)^{-1}\nabla f_t\|^2  = (\xb_t - \ttxb)^\top (\ttHb)^{1/2}(\ttHb)^{-1/2}\Hb_t^\eta(\tHb_t)^{-2}\Hb_t^\eta(\ttHb)^{-1/2}(\ttHb)^{1/2}(\xb_t-\ttxb) \nonumber\\
& \leq \frac{\|\xb_t - \ttxb\|_{\ttHb}^2}{(1-\psi)\lambda_{\min}}\|(\ttHb)^{-1/2}\Hb_t^\eta(\tHb_t)^{-1}\Hb_t^\eta(\ttHb)^{-1/2}\| \nonumber\\
& \leq \frac{\nu^2\lambda_{\min}^2}{(1-\psi)L^2}\|(\ttHb)^{-1/2}(\Hb_t^\eta)^{1/2}\|^2\|(\Hb_t^\eta)^{1/2}(\tHb_t)^{-1}(\Hb_t^\eta)^{1/2}\| \quad (\text{since } \xb_t\in \N_{\nu}) \nonumber\\
& \leq \frac{\nu^2\lambda_{\min}^2}{(1-\psi)L^2}\|(\ttHb)^{-1/2}\Hb_t^\eta(\ttHb)^{-1/2}\|\|(\Hb_t^\eta)^{1/2}(\Hb_t)^{-1}(\Hb_t^\eta)^{1/2}\|\|\Hb_t^{1/2}(\tHb_t)^{-1}\Hb_t^{1/2}\| \nonumber\\
& \leq \frac{\nu^2\lambda_{\min}^2}{(1-\psi)^2L^2}\|(\ttHb)^{-1/2}\Hb_t^\eta(\ttHb)^{-1/2}\| \|(\Hb_t)^{-1/2}\Hb_t^\eta(\Hb_t)^{-1/2}\|.
\end{align}
Noting that $\xb_t^\eta\in\N_{\nu}$ if $\xb_t \in \N_\nu$, we apply Lemma \ref{lem:pre:2} and have
\begin{equation}\label{npequ:5}
\|(\ttHb)^{-1/2}\Hb_t^\eta(\ttHb)^{-1/2}\| \vee \|(\Hb_t)^{-1/2}\Hb_t^\eta(\Hb_t)^{-1/2}\| \leq 1 + \nu\leq 2.
\end{equation}
Combining \eqref{npequ:3}, \eqref{npequ:4}, and \eqref{npequ:5}, we finally obtain
\begin{equation*}
f(\xb_t - (\tHb_t)^{-1}\nabla f_t) \leq f(\xb_t) - \rbr{\frac{1}{2} - \frac{\psi}{2(1-\psi)} - \frac{\nu}{3(1-\psi)^2}}\nabla f_t^\top(\tHb_t)^{-1}\nabla f_t.
\end{equation*}
Thus, it suffices to let
\begin{equation*}
\frac{1}{2} - \frac{\psi}{2(1-\psi)} - \frac{\nu}{3(1-\psi)^2} \geq \beta \Longleftarrow \frac{\psi}{2(1-\psi)} \vee \frac{\nu}{3(1-\psi)^2} \leq \frac{1}{2}\rbr{\frac{1}{2}-\beta}
\end{equation*}
as implied by the conditions stated in the lemma. This completes the proof.

\subsection{Proof of Lemma \ref{lem:6}}\label{pf:lem:6}

Since $\xb_t\in \N_{\nu}$ with $\nu \leq 2/3\cdot (0.5-\beta)$, we apply Lemma \ref{lem:pre:2} and have
\begin{equation*}
\|(\ttHb)^{-1/2}(\Hb_t - \ttHb)(\ttHb)^{-1/2}\| \leq \frac{L}{\lambda_{\min}^{3/2}}\|\xb_t - \ttxb\|_{\ttHb} \leq \nu\leq 0.5-\beta.
\end{equation*}
Thus, we obtain
\begin{equation}\label{pequ:B2}
\rbr{0.5+\beta}\ttHb\preceq \Hb_t\preceq \rbr{1.5-\beta}\ttHb. 
\end{equation}
Since $\pb_t = -(\tHb_t)^{-1}\nabla f_t$ and $\mu_t=1$, we have
\begin{align}\label{npequ:6}
& \nbr{\xb_{t+1} - \ttxb}_{\ttHb} = \|\xb_t - (\tHb_t)^{-1}\nabla f_t - \ttxb\|_{\ttHb} \nonumber\\
& \stackrel{\mathclap{\eqref{pequ:B2}}}{\leq}\frac{1}{\sqrt{0.5+\beta}}\|\xb_t - (\tHb_t)^{-1}\nabla f_t - \ttxb\|_{\Hb_t} \nonumber\\
& \leq \frac{1}{\sqrt{0.5+\beta}}\cbr{\|\xb_t - \Hb_t^{-1}\nabla f_t - \ttxb\|_{\Hb_t} + \|\Hb_t^{-1}\nabla f_t - (\tHb_t)^{-1}\nabla f_t\|_{\Hb_t}}.
\end{align}
For the second term on the right hand side, we have
\begin{align}\label{npequ:7}
& \|\Hb_t^{-1}\nabla f_t - (\tHb_t)^{-1}\nabla f_t\|_{\Hb_t} \nonumber\\
& = \sqrt{\nabla f_t^\top(\Hb_t^{-1} - (\tHb_t)^{-1})\Hb_t(\Hb_t^{-1} - (\tHb_t)^{-1})\nabla f_t} \nonumber\\
& = \sqrt{\nabla f_t^\top\Hb_t^{-1/2}(\Ib - \Hb_t^{1/2}(\tHb_t)^{-1} \Hb_t^{1/2})^2\Hb_t^{-1/2}\nabla f_t} \nonumber\\
& \leq \|\Ib - \Hb_t^{1/2}(\tHb_t)^{-1} \Hb_t^{1/2}\|\cdot\|\Hb_t^{-1}\nabla f_t\|_{\Hb_t}  \nonumber\\
& \leq  \|\Ib - \Hb_t^{1/2}(\tHb_t)^{-1} \Hb_t^{1/2}\|\cbr{\|\xb_t - \Hb_t^{-1}\nabla f_t - \ttxb\|_{\Hb_t} + \|\xb_t - \ttxb\|_{\Hb_t}} \nonumber\\
& \stackrel{\mathclap{\eqref{pequ:B2}}}{\leq}\|\Ib - \Hb_t^{1/2}(\tHb_t)^{-1} \Hb_t^{1/2}\|\cbr{\|\xb_t - \Hb_t^{-1}\nabla f_t - \ttxb\|_{\Hb_t} + \sqrt{1.5-\beta}\|\xb_t - \ttxb\|_{\ttHb}}.
\end{align}
Combining \eqref{npequ:6} and \eqref{npequ:7},
\begin{align}\label{pequ:B3}
\|\xb_{t+1} - \ttxb\|_{\ttHb}& \leq\frac{1}{\sqrt{0.5+\beta}} \big\{(1+\|\Ib - \Hb_t^{1/2}(\tHb_t)^{-1} \Hb_t^{1/2}\|)\|\xb_t - \Hb_t^{-1}\nabla f_t - \ttxb\|_{\Hb_t} \nonumber\\
& \quad + \sqrt{1.5-\beta}\|\Ib - \Hb_t^{1/2}(\tHb_t)^{-1} \Hb_t^{1/2}\|\cdot \|\xb_t - \ttxb\|_{\ttHb}\big\}.
\end{align}
Since $\|\Ib - \Hb_t^{-1/2}\tHb_t \Hb_t^{-1/2}\| \leq \psi\leq 1/3$, we have
\begin{equation*}
(1 - \|\Ib - \Hb_t^{-1/2}\tHb_t \Hb_t^{-1/2}\| )\cdot \Ib \preceq \Hb_t^{-1/2}\tHb_t \Hb_t^{-1/2} \preceq (1+\|\Ib - \Hb_t^{-1/2}\tHb_t \Hb_t^{-1/2}\|)\cdot\Ib
\end{equation*}
and further
\begin{equation*}
-\frac{\|\Ib - \Hb_t^{-1/2}\tHb_t \Hb_t^{-1/2}\| }{1+\|\Ib - \Hb_t^{-1/2}\tHb_t \Hb_t^{-1/2}\| }\cdot \Ib\preceq \Hb_t^{1/2}(\tHb_t)^{-1} \Hb_t^{1/2} - \Ib \preceq \frac{\|\Ib - \Hb_t^{-1/2}\tHb_t \Hb_t^{-1/2}\| }{1-\|\Ib - \Hb_t^{-1/2}\tHb_t \Hb_t^{-1/2}\| }\cdot\Ib.
\end{equation*}
Thus, we have
\begin{equation*}
\|\Ib - \Hb_t^{1/2}(\tHb_t)^{-1} \Hb_t^{1/2}\| \leq \frac{\|\Ib - \Hb_t^{-1/2}\tHb_t \Hb_t^{-1/2}\| }{1 - \|\Ib - \Hb_t^{-1/2}\tHb_t \Hb_t^{-1/2}\| }\leq \frac{3}{2}\|\Ib - \Hb_t^{-1/2}\tHb_t \Hb_t^{-1/2}\|.
\end{equation*}
Combining the above inequality with \eqref{pequ:B3}, we obtain
\begin{align}\label{pequ:B4}
&\|\xb_{t+1} - \ttxb\|_{\ttHb} \nonumber\\
&  \leq \frac{3}{2\sqrt{0.5+\beta}} \|\xb_t - \Hb_t^{-1}\nabla f_t - \ttxb\|_{\Hb_t} + \frac{3}{2}\sqrt{\frac{1.5-\beta}{0.5+\beta}}\cdot\|\Ib - \Hb_t^{-1/2}\tHb_t \Hb_t^{-1/2}\|  \cdot \|\xb_t - \ttxb\|_{\ttHb} \nonumber\\
& \leq \frac{3\sqrt{2}}{2}\|\xb_t - \Hb_t^{-1}\nabla f_t - \ttxb\|_{\Hb_t}  + \frac{3\sqrt{3}}{2}\|\Ib - \Hb_t^{-1/2}\tHb_t \Hb_t^{-1/2}\|  \cdot \|\xb_t - \ttxb\|_{\ttHb}.
\end{align}
Furthermore,
\begin{align*}
\|\xb_t - \Hb_t^{-1}\nabla f_t - \ttxb\|_{\Hb_t} & = \|\Hb_t^{-1/2}(\Hb_t(\xb_t - \ttxb) - \nabla f_t)\|\\
& \leq \frac{1}{\sqrt{\lambda_{\min}}}\nbr{\Hb_t(\xb_t - \ttxb) - \rbr{\int_0^1\Hb(\ttxb + \tau(\xb_t - \ttxb))d\tau }(\xb_t-\ttxb)}\\
& \leq \frac{\|\xb_t -\ttxb\|^2}{\sqrt{\lambda_{\min}}}\cdot \int_0^1(1-\tau)L d\tau \leq \frac{L}{2\lambda_{\min}^{3/2}}\|\xb_t-\ttxb\|_{\ttHb}^2.
\end{align*}
Combining the above inequality with \eqref{pequ:B4}, we complete the proof.

\subsection{Proof of Theorem \ref{thm:3}}\label{pf:thm:3}

We suppose the event \eqref{event:EE} happens, which has probability $1-\delta\pi^2/6$. By Lemma \ref{lem:4} and Corollary \ref{cor:1}, we know that $\xb_t$ converges $R$-linearly for all $t\geq \T_1$, and $\xb_{\T:\T+\J}\in \N_{\nu}$ for any $\J\geq 0$.~Thus, it suffices to study the convergence after $\T+\J$. For any $t\geq 0$, to apply Lemmas \ref{lem:5} and \ref{lem:6}, we characterize the difference between $\tHb_{\T+\J+t}$ and $\Hb_{\T+\J+t}$. We have
\begin{align}\label{npequ:8}
& \|\Hb_{\T+\J + t}^{-1/2}\tHb_{\T+\J + t}\Hb_{\T+\J + t}^{-1/2} -\Ib\|  \nonumber\\
& \leq \|\Hb_{\T+\J + t}^{-1/2}(\tHb_{\T+\J + t} - \ttHb)\Hb_{\T+\J + t}^{-1/2}\| + \|\Hb_{\T+\J + t}^{-1/2}\ttHb\Hb_{\T+\J + t}^{-1/2} - \Ib\| \nonumber\\
& \stackrel{\mathclap{\text{Lemma \ref{lem:pre:2}}}}{\leq}\;\;\; \|\Hb_{\T+\J + t}^{-1/2}(\tHb_{\T+\J + t} - \ttHb)\Hb_{\T+\J + t}^{-1/2}\|  + \frac{L}{\lambda_{\min}^{3/2}}\|\xb_{\T+\J+t} - \ttxb\|_{\ttHb}. 
\end{align}
For the first term $\|\Hb_{\T+\J + t}^{-1/2}(\tHb_{\T+\J + t} - \ttHb)\Hb_{\T+\J + t}^{-1/2}\|$, we apply the formula \eqref{equ:simple} and have
\begin{align}\label{npequ:9}
&\|\Hb_{\T+\J + t}^{-1/2}(\tHb_{\T+\J + t} - \ttHb)\Hb_{\T+\J + t}^{-1/2}\| \nonumber\\
& \stackrel{\mathclap{\eqref{equ:simple}}}{\leq} \frac{1}{\lambda_{\min}}\nbr{\barEb_{\T+\J+t}}  + \frac{1}{\T+\J + t + 1}\bigg\{\sum_{j=0}^{\T-1}\|\Hb_{\T+\J + t}^{-1/2}(\Hb_j - \ttHb)\Hb_{\T+\J + t}^{-1/2}\| \nonumber\\
&\hskip6cm + \sum_{j=\T}^{\T+\J+t}\|\Hb_{\T+\J + t}^{-1/2}(\Hb_j - \ttHb)\Hb_{\T+\J + t}^{-1/2}\| \bigg\} \nonumber\\
& \leq \frac{1}{\lambda_{\min}}\|\barEb_{\T+\J+t}\| + \frac{2\T\kappa}{\T+\J+t+1} + \frac{L}{\lambda_{\min}(\T+\J+t+1)}\sum_{j=\T}^{\T+\J+t}\|\xb_j - \ttxb\| \nonumber\\
& \stackrel{\mathclap{\eqref{event:EE}}}{\leq} 8\Upsilon\sqrt{\frac{\log(d(\T+\J+t+1)/\delta)}{\T+\J+t+1}} + \frac{2\T\kappa}{\T+\J+t+1} \nonumber\\ 
&\hskip0.5cm +\frac{L}{\lambda_{\min}(\T+\J+t+1)}\sum_{j=\T}^{\T+\J+t}\sqrt{\frac{2(f(\xb_0) - f(\ttxb))}{\lambda_{\min}}}(1 - \phi)^{(j-\T_1)/2} \quad (\text{also use Lemma \ref{lem:4}}) \nonumber\\
& \leq 8\Upsilon\sqrt{\frac{\log(d(\T+\J+t+1)/\delta)}{\T+\J+t+1}} + \frac{2\T\kappa}{\T+\J+t+1} + \frac{L\sqrt{2(f(\xb_0) - f(\ttxb))}}{\lambda_{\min}^{3/2}(\T+\J+t+1)}(1-\phi)^{\T_2/2}\sum_{j=0}^{\infty}(1-\phi)^{j/2} \nonumber\\
& \stackrel{\mathclap{\eqref{equ:new}}}{\leq}\; 8\Upsilon\sqrt{\frac{\log(d(\T+\J+t+1)/\delta)}{\T+\J+t+1}} + \frac{2\T\kappa}{\T+\J+t+1} + \frac{\sqrt{2}\nu}{\sqrt{3}(\T+\J+t+1)(1-\sqrt{1-\phi})} \nonumber\\
& \leq 8\Upsilon\sqrt{\frac{\log(d(\T+\J+t+1)/\delta)}{\T+\J+t+1}} + \frac{2\T\kappa}{\T+\J+t+1} + \frac{2\nu}{(\T+\J+t+1)\phi} \nonumber\\
& \leq 8\Upsilon\sqrt{\frac{\log(d(\T+\J+t+1)/\delta)}{\T+\J+t+1}} + \frac{4\T\kappa}{\T+\J+t+1} = \rho_t.\quad (\text{since } \nu/\phi \leq1/\phi\leq\T\kappa)
\end{align}
Combining \eqref{npequ:8} and \eqref{npequ:9} together, we obtain
\begin{equation}\label{pequ:B5}
\|\Hb_{\T+\J + t}^{-1/2}\tHb_{\T+\J + t}\Hb_{\T+\J + t}^{-1/2} -\Ib\|\leq \rho_t+\frac{L}{\lambda_{\min}^{3/2}}\|\xb_{\T+\J+t} - \ttxb\|_{\ttHb} \leq \rho_t + \nu\leq \rho_0 + \nu.
\end{equation}
Thus, in order to apply Lemma \ref{lem:6} for all $\{\xb_{\T+\J+t}\}_{t\geq 0}$, we require $\nu\leq 2/3\cdot (0.5-\beta)$ and
\begin{align*}
\rho_0 + \nu \leq \frac{0.5-\beta}{1.5-\beta} & \xLeftarrow{\J = 4\T\kappa/\nu} 8\Upsilon\sqrt{\frac{\log(d(\T+\J)/\delta)}{\T+\J}} + 2\nu \leq \frac{0.5-\beta}{1.5-\beta} \\
& \Longleftarrow 8\Upsilon\sqrt{\frac{\log(d(\T+\J)/\delta)}{\T+\J}}  \vee\nu\leq \frac{1}{3}\frac{0.5-\beta}{1.5-\beta}\\
& \stackrel{\eqref{cond:epsnu}}{\Longleftarrow} 8\Upsilon\sqrt{\frac{\log(d(\T+\J)/\delta)}{\T+\J}}  \leq \epsilon \; \text{ and }\; \epsilon\vee \nu \leq \frac{1}{3}\frac{0.5-\beta}{1.5-\beta},
\end{align*}
which is implied by \eqref{cond:epsnu} and \eqref{nequ:1}. Thus, for any $t\geq 0$, we apply Lemma \ref{lem:6} and obtain
\begin{align}\label{pequ:B6}
& \|\xb_{\T+\J+t+1}-\ttxb\|_{\ttHb} \nonumber\\
& \leq 3\cbr{\frac{L}{\lambda_{\min}^{3/2}}\|\xb_{\T+\J+t}-\ttxb\|_{\ttHb}^2 + \|\Hb_{\T+\J + t}^{-1/2}\tHb_{\T+\J + t}\Hb_{\T+\J + t}^{-1/2} -\Ib\|\cdot \|\xb_{\T+\J+t} - \ttxb\|_{\ttHb}} \nonumber\\
&\stackrel{\mathclap{\eqref{pequ:B5}}}{\leq} 3\cbr{\frac{2L}{\lambda_{\min}^{3/2}}\|\xb_{\T+\J+t}-\ttxb\|_{\ttHb}^2 +  \rho_t\|\xb_{\T+\J+t}-\ttxb\|_{\ttHb}}.
\end{align}
We then claim that
\begin{equation}\label{pequ:B7}
\frac{2L}{\lambda_{\min}^{3/2}}\|\xb_{\T+\J+t}-\ttxb\|_{\ttHb} \leq 3\rho_t, \quad \forall t\geq 0.
\end{equation}
We prove \eqref{pequ:B7} by induction. For $t = 0$, we note that
\begin{equation*}
\frac{2L}{\lambda_{\min}^{3/2}}\|\xb_{\T+\J}-\ttxb\|_{\ttHb} \leq 2\nu = \frac{12\T\kappa}{6\T\kappa/\nu}\leq \frac{12\T\kappa}{\T + 4\T\kappa/\nu +1} = \frac{12\T\kappa}{\T + \J +1} \leq 3\rho_0.
\end{equation*}
Suppose \eqref{pequ:B7} holds for $t\geq 0$, then \eqref{pequ:B6} leads to
\begin{equation*}
\frac{2L}{\lambda_{\min}^{3/2}}\|\xb_{\T+\J+t+1}-\ttxb\|_{\ttHb} \leq \frac{2L}{\lambda_{\min}^{3/2}}\cdot 12\rho_t\|\xb_{\T+\J+t}-\ttxb\|_{\ttHb} \stackrel{\eqref{pequ:B7}}{\leq} 36\rho_t^2.
\end{equation*}
On the other hand, using $(\T+\J+t+2)\leq 2(\T+\J+t+1)$, $\forall t\geq 0$, we know $\rho_t/2 \leq \rho_{t+1}$. Thus, it suffices to show $36\rho_t^2\leq 3\rho_t/2\Longleftrightarrow 24\rho_t\leq 1$. Since $\rho_t$ is decreasing, we require $24\rho_0\leq 1$. Note that
\begin{equation*}
\rho_0 \leq 8\Upsilon\sqrt{\frac{\log(d\T/\delta)}{\T}} + \nu \stackrel{\eqref{T2}}{\leq} \epsilon + \nu.
\end{equation*}
Thus, $24\rho_0\leq 1$ is guaranteed by \eqref{cond:epsnu}. Combining \eqref{pequ:B6} and \eqref{pequ:B7} completes the proof.

\subsection{Proof of Corollary \ref{cor:2}}\label{pf:cor:2}

We only need to note that
\begin{align*}
& \frac{4\T\kappa}{\T+\J+t+1} \leq 8\Upsilon\sqrt{\frac{\log(d(\T+\J+t+1)/\delta)}{\T+\J+t+1}}\\
& \Longleftrightarrow  \frac{\T^2\kappa^2}{4\Upsilon^2}\leq (\T+\J+t+1)\log(d(\T+\J+t+1)/\delta) \Longleftarrow \T+\J + t+1 \geq \frac{\T^2\kappa^2}{4\Upsilon^2\log(d\T/\delta)}. 
\end{align*}
This completes the proof.

\subsection{Proof of Lemma \ref{lem:7}}\label{pf:lem:7}

We characterize $\sum_{i=0}^{t}z_{i,t}^2$ and $z_t^{(\max)}$ in \eqref{equ:sub:concen}, where $z_{i,t} = (w_i-w_{i-1})/w_t$ and $z_t^{(\max)} = \max_{i\in\{0,\ldots,t\}}z_{i,t}$. We have
\begin{multline}\label{npequ:10}
\sum_{i=0}^tz_{i,t}^2 = \frac{1}{w(t)^2}\sum_{i=0}^t(w(i) - w(i-1))^2 \leq \frac{1}{w(t)^2}\sum_{i=0}^{t}(w'(i))^2 \leq \frac{1}{w(t)^2}\int_0^{t+1}(w'(i))^2 \; di\\
= \frac{1}{w(t)^2}\int_{0}^{t+1} (w(i)w'(i))' - w(i)w''(i) di
\leq \frac{w(t+1)w'(t+1)}{w(t)^2} \leq \Psi^2\frac{w'(t)}{w(t)},
\end{multline} 
where the first two inequalities use the fact that $w'(t)$ is non-negative and  non-decreasing; the second last inequality uses the fact that $w(t)w''(t)\geq 0$, $\forall t\geq -1$; and the last inequality uses Assumption \ref{ass:5}(v). By the same derivation, we have
\begin{equation}\label{npequ:11}
z_t^{(\max)} = \max_{i\in \{0,\ldots,t\}}z_{i,t} = \frac{\max_{i\in\{0,\ldots,t\}}(w(i)-w(i-1))}{w(t)} \leq \frac{\max_{i\in\{0,\ldots,t\}} w'(i)}{w(t)} = \frac{w'(t)}{w(t)}.
\end{equation}
Plugging \eqref{npequ:10} and \eqref{npequ:11} into \eqref{equ:sub:concen} completes the proof.

\subsection{Proof of Lemma \ref{lem:8}}\label{pf:lem:8}

By Lemma \ref{lem:7}, we know \eqref{con:barE} holds for any $t$. By the definition of $\I_1$ in \eqref{I1}, we know 
\begin{equation*}
\log\rbr{\frac{d(t+1)}{\delta}}\frac{w'(t)}{w(t)} \leq \rbr{\frac{\epsilon}{8\Upsilon\Psi}}^2 \wedge 1, \quad \forall t\geq \I_1,
\end{equation*}
which implies $\|\barEb_t\| \leq \epsilon\lambda_{\min}$. Using \eqref{equ:simple} and noting that $\lambda_{\min}\cdot\Ib\preceq\sum_{i=0}^{t}(w_i-w_{i-1})\Hb_i/w_t\preceq \lambda_{\max} \cdot\Ib$, we complete the proof.

\subsection{Proof of Theorem \ref{thm:4}}\label{pf:thm:4}

We follow the same proof structure as Theorem \ref{thm:3}. We suppose the event \eqref{event:EE:new} happens, which occurs with probability $1 - \delta\pi^2/6$. We only need to study the convergence after $\I + \U$ where $\U$ is chosen such that $w(\I+\U) = 2w(\I-1)\kappa/\nu$. We claim that
\begin{equation}\label{pequ:B8}
\|\xb_{\I+\U+t+1} - \ttxb\|_{\ttHb} \leq 6\theta_t\|\xb_{\I+\U+t} - \ttxb\|_{\ttHb}\; \text{ and }\; \frac{2L}{\lambda_{\min}^{3/2}}\|\xb_{\I+\U+t+1} - \ttxb\|_{\ttHb} \leq\theta_{t+1}, \; \forall t\geq 0.
\end{equation}
We prove \eqref{pequ:B8} by induction. Before showing \eqref{pequ:B8}, we note that 
\begin{equation}\label{nequ:3}
\theta_t \leq \epsilon + 3\nu \leq 1/(12\Psi),\quad\quad \theta_{t+1} \geq \theta_t/\Psi, \quad\quad \forall t\geq 0.
\end{equation}
The first result is implied by \eqref{epsnu} and the second result is implied by Assumption \ref{ass:5}(v). For $t = 0$, we characterize the difference between $\tHb_{\I+\U}$ and $\Hb_{\I+\U}$. We have
\begin{align}\label{npequ:12}
&\|\Hb_{\I+\U}^{-1/2}\tHb_{\I+\U}\Hb_{\I+\U}^{-1/2} -\Ib\| \nonumber\\ 
&\leq \|\Hb_{\I+\U}^{-1/2}(\tHb_{\I+\U} - \ttHb)\Hb_{\I+\U}^{-1/2}\| + \|\Hb_{\I+\U}^{-1/2}\ttHb\Hb_{\I+\U}^{-1/2} - \Ib\| \nonumber\\
& \leq \|\Hb_{\I+\U}^{-1/2}(\tHb_{\I+\U} - \ttHb)\Hb_{\I+\U}^{-1/2}\|  + \frac{L}{\lambda_{\min}^{3/2}}\|\xb_{\I+\U} - \ttxb\|_{\ttHb}.\quad (\text{Lemma \ref{lem:pre:2}})
\end{align}
For the first term on the right hand side, we have
\begin{align}\label{npequ:13}
&\|\Hb_{\I+\U}^{-1/2}(\tHb_{\I+\U} - \ttHb)\Hb_{\I+\U}^{-1/2}\| \nonumber\\ 
&\stackrel{\mathclap{\eqref{equ:simple}}}{\leq} \frac{\|\barEb_{\I+\U}\|}{\lambda_{\min}}  + \frac{1}{w(\I+\U)}\sum_{j=0}^{\I+\U}(w(j) - w(j-1))\|\Hb_{\I+\U}^{-1/2}(\Hb_j - \ttHb)\Hb_{\I+\U}^{-1/2}\| \nonumber\\
&=\frac{\|\barEb_{\I+\U}\|}{\lambda_{\min}}  + \frac{1}{w(\I+\U)}\cbr{\rbr{\sum_{j=0}^{\I-1} + \sum_{j=\I}^{\I+\U}}(w(j) - w(j-1))\|\Hb_{\I+\U}^{-1/2}(\Hb_j - \ttHb)\Hb_{\I+\U}^{-1/2}\|} \nonumber\\
&\leq \frac{\|\barEb_{\I+\U}\|}{\lambda_{\min}}  + \frac{2w(\I-1)\kappa}{w(\I+\U)} + \frac{L}{\lambda_{\min}w(\I+\U)}\sum_{j=\I}^{\I+\U}(w(j) - w(j-1))\|\xb_j - \ttxb\| \nonumber\\
&\leq \frac{\|\barEb_{\I+\U}\|}{\lambda_{\min}} + \frac{2w(\I-1)\kappa}{w(\I+\U)} + \sum_{j=\I}^{\I+\U}\frac{L(w(j) - w(j-1))}{\lambda_{\min}w(\I+\U)}\cbr{\frac{2(f(\xb_0) - f(\ttxb))}{\lambda_{\min}}(1-\phi)^{j - \I_1}}^{1/2} \; (\text{Lemma \ref{lem:4}}) \nonumber\\
&= \frac{\|\barEb_{\I+\U}\|}{\lambda_{\min}} + \frac{2w(\I-1)\kappa}{w(\I+\U)} + \cbr{\frac{2L^2(f(\xb_0)-f(\ttxb))}{\lambda_{\min}^{3}}(1-\phi)^{\T_2} }^{1/2}\sum_{j=0}^{\U}\frac{w(\I+j) - w(\I+j-1)}{w(\I+\U)}(1-\phi)^{j/2} \nonumber\\
& \stackrel{\mathclap{\eqref{equ:new}}}{\leq}\; \frac{\|\barEb_{\I+\U}\|}{\lambda_{\min}} + \frac{2w(\I-1)\kappa}{w(\I+\U)} + \nu\sum_{j=0}^{\U}\frac{w(\I+j) - w(\I+j-1)}{w(\I+\U)} \nonumber\\
&\leq \frac{\|\barEb_{\I+\U}\|}{\lambda_{\min}} + \frac{2w(\I-1)\kappa}{w(\I+\U)} + \nu =  \frac{\|\barEb_{\I+\U}\|}{\lambda_{\min}} + \frac{4w(\I-1)\kappa}{w(\I+\U)} \leq \theta_0.
\end{align}
Combining \eqref{npequ:12} and \eqref{npequ:13}, we know that, to apply Lemma \ref{lem:6}, we need $\nu \leq 2/3\cdot(0.5-\beta)$~and
\begin{equation*}
\theta_0 + \nu \leq\frac{0.5-\beta}{1.5-\beta} \stackrel{\eqref{nequ:3}}{\Longleftarrow} \epsilon + 4\nu \leq \frac{0.5-\beta}{1.5-\beta}\Longleftarrow \epsilon\vee \nu \leq \frac{1}{5}\frac{0.5-\beta}{1.5-\beta},
\end{equation*}
as implied by \eqref{epsnu}. Thus, Lemma \ref{lem:6} leads to
\begin{equation}
\|\xb_{\I+\U+1} - \ttxb\|_{\ttHb} \leq 3\cbr{\frac{2L}{\lambda_{\min}^{3/2}}\|\xb_{\I+\U} - \ttxb\|_{\ttHb}^2 + \theta_0\|\xb_{\I+\U} - \ttxb\|_{\ttHb} } \leq 6\theta_0\|\xb_{\I+\U} - \ttxb\|_{\ttHb},
\end{equation}
where the last inequality is due to $\xb_{\I+\U}\in \N_{\nu}$ and $2\nu \leq \theta_0$. Furthermore, we can see that
\begin{equation*}
\frac{2L}{\lambda_{\min}^{3/2}}\|\xb_{\I+\U+1} - \ttxb\|_{\ttHb} \leq \frac{2L}{\lambda_{\min}^{3/2}} \cdot 6\theta_0\|\xb_{\I+\U} - \ttxb\|_{\ttHb} \leq 6\theta_0^2=(6\Psi\theta_0)\cdot\frac{\theta_0}{\Psi}\stackrel{\eqref{nequ:3}}{\leq} \theta_1.
\end{equation*}
Thus, \eqref{pequ:B8} holds for $t = 0$. Suppose \eqref{pequ:B8} holds for $t-1$ with $t\geq 1$, we prove \eqref{pequ:B8} for $t$. We still characterize the difference between $\tHb_{\I+\U+t}$ and $\Hb_{\I+\U+t}$. We have
\begin{align}\label{npequ:14}
& \|\Hb_{\I+\U + t}^{-1/2}\tHb_{\I+\U + t}\Hb_{\I+\U + t}^{-1/2} -\Ib\| \nonumber\\
& \leq \|\Hb_{\I+\U + t}^{-1/2}(\tHb_{\I+\U + t} - \ttHb)\Hb_{\I+\U + t}^{-1/2}\| + \|\Hb_{\I+\U + t}^{-1/2}\ttHb\Hb_{\I+\U + t}^{-1/2} - \Ib\| \nonumber\\
& \stackrel{\mathclap{\text{Lemma \ref{lem:pre:2}}}}{\leq}\;\;\; \|\Hb_{\I+\U + t}^{-1/2}(\tHb_{\I+\U + t} - \ttHb)\Hb_{\I+\U + t}^{-1/2}\|  + \frac{L}{\lambda_{\min}^{3/2}}\|\xb_{\I+\U+t} - \ttxb\|_{\ttHb}.
\end{align}
For the first term on the right hand side, we have
\begin{align}\label{npequ:15}
&\|\Hb_{\I+\U + t}^{-1/2}(\tHb_{\I+\U + t} - \ttHb)\Hb_{\I+\U + t}^{-1/2}\| \nonumber\\
&\stackrel{\mathclap{\eqref{equ:simple}}}{\leq} \frac{\|\barEb_{\I+\U+t}\|}{\lambda_{\min}} + \frac{1}{w(\I+\U +t)}\cbr{\rbr{\sum_{j=0}^{\I-1} + \sum_{j=\I}^{\I+\U} + \sum_{j=\I+\U+1}^{\I+\U+t}}(w(j) - w(j-1))\|\Hb_{\I+\U + t}^{-1/2}(\Hb_j - \ttHb)\Hb_{\I+\U + t}^{-1/2}\| } \nonumber\\
&\leq \frac{\|\barEb_{\I+\U+t}\|}{\lambda_{\min}} + \frac{2w(\I-1)\kappa + w(\I+\U)\nu}{w(\I+\U +t)} + \sum_{j=1}^{t}\frac{L(w(\I+\U+j) - w(\I+\U+j-1))}{\lambda_{\min}^{3/2}w(\I+\U +t)}\frac{\|\xb_{\I+\U} -\ttxb\|_{\ttHb}}{(2\Psi)^j} \nonumber\\
& \leq \frac{\|\barEb_{\I+\U+t}\|}{\lambda_{\min}} + \frac{2w(\I-1)\kappa + w(\I+\U)\nu}{w(\I+\U +t)} + \frac{\nu}{w(\I+\U +t)}\sum_{j=1}^{t}\frac{w(\I+\U+j)}{(2\Psi)^j} \nonumber\\
& \leq \frac{\|\barEb_{\I+\U+t}\|}{\lambda_{\min}} + \frac{2w(\I-1)\kappa + w(\I+\U)\nu}{w(\I+\U +t)} + \frac{\nu w(\I+\U)}{w(\I+\U +t)}\sum_{j=1}^{t}\frac{w(\I+\U+j)}{w(\I+\U)(2\Psi)^j} \nonumber\\
& \leq \frac{\|\barEb_{\I+\U+t}\|}{\lambda_{\min}} + \frac{2w(\I-1)\kappa + 2w(\I+\U)\nu}{w(\I+\U +t)}\quad (\text{Assumption \ref{ass:5}(v)}) \nonumber\\
& \leq\theta_t,
\end{align}
where the second inequality uses the hypothesis and the fact that the convergence rate $6\theta_t\leq 1/(2\Psi)$. Thus, combining \eqref{npequ:14} and \eqref{npequ:15},
\begin{equation*}
\|\Hb_{\I+\U + t}^{-1/2}\tHb_{\I+\U + t}\Hb_{\I+\U + t}^{-1/2} -\Ib\| \leq \theta_t + \nu \stackrel{\eqref{nequ:3}}{\leq} \epsilon + 4\nu \stackrel{\eqref{epsnu}}{\leq}\frac{0.5-\beta}{1.5-\beta}.
\end{equation*}
Thus, the conditions of Lemma \ref{lem:6} are satisfied, which leads to
\begin{multline}\label{nequ:2}
\|\xb_{\I+\U+t+1} - \ttxb\|_{\ttHb} \leq 3\cbr{\frac{2L}{\lambda_{\min}^{3/2}}\|\xb_{\I+\U+t} - \ttxb\|_{\ttHb}^2 + \theta_t\|\xb_{\I+\U+t} - \ttxb\|_{\ttHb}   } \\ \leq 6\theta_t\|\xb_{\I+\U+t} - \ttxb\|_{\ttHb}.
\end{multline}
The last inequality uses the hypothesis. This shows the first result in \eqref{pequ:B8}. For the second result, we have
\begin{equation*}
\frac{2L}{\lambda_{\min}^{3/2}}\|\xb_{\I+\U+t+1} - \ttxb\|_{\ttHb} \stackrel{\eqref{nequ:2}}{\leq} \frac{2L}{\lambda_{\min}^{3/2}} \cdot 6\theta_t\|\xb_{\I+\U+t} - \ttxb\|_{\ttHb} \leq 6\theta_t^2 = (6\theta_t\Psi)\cdot\frac{\theta_t}{\Psi}\stackrel{\eqref{nequ:3}}{\leq} \theta_{t+1}.
\end{equation*}
This completes the induction and finishes the proof.

\end{document}